\documentclass{jfm}
\usepackage{epstopdf, epsfig}

\newtheorem{theorem}{Theorem}
\usepackage{tabularx}
\usepackage{graphicx}% Include figure files
\usepackage{longtable}
\usepackage{pgfplots}
\usepackage{bm}% bold math

\usepackage{array}
\usepackage{makecell}
\usepackage{natbib}
\usepackage{physics}
\usepackage{float}
\usepackage{todo}
\usepackage{enumitem}
\usepackage{multirow}
\usepackage{tabularx}
\usepackage{gensymb}
\usepackage{amsmath}
\usepackage{enumerate}
\usepackage{mathrsfs}

\DeclareMathOperator*{\argmin}{arg\,min}

\newcommand{\red}[1]{\textcolor{black}{#1}}

\newcommand{\blue}[1]{\ifmmode {\textcolor{black}{#1}}\else {\textcolor{black}{#1}}\fi}

\usepackage{xcolor}

\usepackage[version=4]{mhchem} 

\usepackage{upgreek}
\usepackage{graphicx,amsmath,amsfonts,amssymb}
\usepackage{hyperref}
\usepackage{subfigure}
\usepackage{algorithm,algorithmic}
\usepackage{mathtools,bm}
\usepackage{multirow,booktabs,makecell}
\usepackage{tikz,url}
\usetikzlibrary{shapes,arrows}
\tikzstyle{process} = [rectangle,minimum width=2cm,minimum height=1cm,text centered,text width =4cm,draw=black]

\newcommand{\bmv}{\bm{v}}

\newcommand{\bmm}{\bm{m}}
\newcommand{\bmn}{\bm{n}}
\newcommand{\bmq}{\bm{q}}
\newcommand{\bmu}{\bm{u}}

\newcommand{\m}{\bmm}

\newcommand{\intst}[1]{\langle\!\langle#1\rangle\!\rangle_h}
\newcommand{\ints}[1]{\langle#1\rangle_h}

\newtheorem{proposition}{Proposition}[section]
 \newtheorem{definition}[theorem]{Definition}
 \newtheorem{example}[theorem]{Example}

\newtheorem{remark}[theorem]{Remark}

\numberwithin{equation}{section}

\newcommand{\nn}{\bm{n}}
\newcommand{\uu}{\bm{u}}

\newcommand{\mm}{\bm{m}}
\newcommand{\so}{s}

\newcommand{\nnu}{\bm{\nu}}
\newcommand{\pphi}{\phi}
\newcommand{\Pphi}{\Phi}

\newcommand{\bPphi}{\bm{\Phi}}

\newcommand{\ssigma}{\bm{\sigma}}

\newcommand{\stint}[1]{\langle\!\langle #1 \rangle\!\rangle_h}
\newcommand{\sint}[1]{\langle #1 \rangle_h}

\def\E{\mathcal{E}}
\def\pstar{\mathcal{P}_*}

\shorttitle{MFC of droplet dynamics}
\shortauthor{G. Fu, H. Ji, W. Pazner and W. Li}

\title{Mean field control of droplet dynamics with high order finite element computations}

\author{Guosheng Fu\aff{1},
  Hangjie Ji\aff{2},
Will Pazner\aff{3},
  \and Wuchen Li\aff{4} \corresp{ \email{wuchen@mailbox.sc.edu}}}

\affiliation{\aff{1} Department of Applied and Computational Mathematics and Statistics, University of Notre Dame, USA.
\aff{2} Department of Mathematics,
North Carolina State University, USA.
\aff{3}  Fariborz Maseeh Department of Mathematics and Statistics, Portland State University, Portland, USA.
\aff{4} Department of Mathematics, University of South Carolina, Columbia, USA.}

\date{\today}% It is always \today, today,
\begin{document}

\maketitle 

\begin{abstract}
Liquid droplet dynamics are widely used in biological and engineering applications, which contain complex interfacial instabilities and pattern formation such as droplet merging, splitting, and transport. This paper studies a class of mean field control formulations for these droplet dynamics, which can be used to control and manipulate droplets in applications. We first formulate the droplet dynamics as gradient flows of free energies in modified optimal transport metrics with nonlinear mobilities. 
We then design an optimal control problem for these gradient flows. \blue{As an example, a lubrication equation for a thin volatile liquid film laden with an active suspension is developed, with control achieved through its activity field.}
Lastly, we apply the primal-dual hybrid gradient algorithm with high-order finite element methods to simulate the proposed mean field control problems. Numerical examples, including droplet formation, bead-up/spreading, transport, and merging/splitting on a two-dimensional spatial domain,  demonstrate the effectiveness of the proposed mean field control mechanism.

\end{abstract}

\begin{keywords}
droplet dynamics; optimal transport; mean field control; finite element methods; primal-dual hybrid gradient algorithms.
%Authors should not enter keywords on the manuscript, as these must be chosen by the author during the online submission process and will then be added during the typesetting process.
\end{keywords}

\section{Introduction}
\label{sec:intro}
The dynamics of liquid droplets on solid substrates have been extensively investigated for the past two decades due to their significant connections to a wide range of biological and engineering applications, including heat and mass transfer \citep{ji2018instability}, vapor and particle capture \citep{sadeghpour2021experimental,sadeghpour2019water}, filtration and digital microfluidics \citep{kim2001micropumping}. These droplet systems often exhibit complex pattern formation rendered by the interactions between the surface tension of the free interface and other physical effects. Fundamental droplet manipulation operations, such as droplet transport, merging, and splitting, have been explored experimentally through various mechanisms such as electrodewetting \citep{chu2023electrohydrodynamics,li2019ionic}, electrochemical oxidation \citep{khoshmanesh2017liquid}, and coalescence-induced propulsion \citep{10.1063/5.0124560}. \blue{For droplets composed of active matters, such as self-propelled swimmers and driven bio-filaments, recent studies have also explored methods to control these active droplets by manipulating the activity field or evaporation \citep{shankar2022optimal,chandel2024evaporation}.}
Developing robust control mechanisms for droplet dynamics by varying external fields is essential to optimize the manipulation of droplets for practical applications. In this work, we will focus on mean field control of droplet dynamics in volatile \blue{active} thin liquid films.  

Thin layers of viscous fluids spreading on solid substrates, often referred to as \emph{coating flows}, have been studied in the context of tear films in human eyes and surface painting processes.
When the solid substrate is hydrophobic or non-wetting, the fluid on the substrate spontaneously undergoes a sequence of instabilities and morphological changes, leading to the formation of dry spots and an array of interacting droplets \citep{glasner2003coarsening,ji2024coarsening}. This fascinating \emph{dewetting} phenomenon arises from the interplay of the intermolecular forces between the solid substrate and the fluid and the surface tension of the fluid. 

In the limit of low Reynolds number, lubrication theory and thin-film models for free-surface flows have been widely used to model the droplet dynamics \citep{oron1997long}. Specifically, 
a classical \blue{non-dimensional} long-wave thin-film equation can be cast into a gradient dynamics form \citep{thiele2016gradient}:
\begin{subequations}
\label{eq:model}
\begin{equation}
    \frac{\partial h}{\partial t} = \nabla\cdot \left( V_1(h)\nabla \frac{\delta}{\delta h}\mathcal{E}(h)\right) - V_2(h)\frac{\delta}{\delta h}\mathcal{E}(h), \quad \text{ on } [0, T] \times\Omega,
\label{gradientFlow}
\end{equation}
where $h(t,\bm{x})$ represents the free surface height of the fluid film, $\E$ is an energy functional, 
and $V_1(h)\ge 0$ and $V_2(h)\ge 0$ are mobility functions associated with mass-conserving and non-mass-conserving contributions to the dynamics.  We assume homogeneous Neumann boundary conditions $V_1(h)\nabla \frac{\delta}{\delta h}\mathcal{E}(h)\cdot \bm\nu = 0$ on the domain boundary $\partial\Omega$, where $\bm\nu$ is the outward normal direction on $\partial\Omega$.

For a volatile \blue{inactive} thin film on a hydrophobic substrate heated or cooled from below \citep{ajaev2001steady,ajaev2005spreading,ji2018instability}, vapor condensation or fluid evaporation occurs and leads to non-mass-conserving dynamics. In this case, typical mobility functions in \eqref{eq:model} take the forms
\begin{equation}
\label{mob}
V_1(h) = h^3,\quad V_2(h) = \displaystyle\frac{\gamma}{h + K}, 
\end{equation}
where $V_1(h)$ originates from the no-slip boundary condition at the liquid-solid interface, $V_2(h)$ characterizes the non-mass-conserving liquid evaporation or condensation,
$\gamma \ge 0$ is a phase change rate, and $K > 0$ is a kinetic parameter.
The energy $\mathcal{E}(h)$ is given by
\begin{equation}
% \mathcal{E}(h) = \int_{\Omega} \tfrac{1}{2}
% |{\nabla{h}}|^2 + U(h)\, dx\,,
\mathcal{E}(h) = \int_{\Omega} \frac{\alpha^2}{2}
|{\nabla{h}}|^2 + U(h)\, dx\,,
%\qquad 
%\frac{d\mathcal{E}}{dt} = -
%\int_0^L h^3\|\nabla P\|^2~dx  \le 0\,
\label{Energy}
\end{equation}
where \blue{$\tfrac{\alpha^2}{2}
|{\nabla{h}}|^2$ represents the contribution of the surface energy of the free interface with $\alpha>0$}, and
% $\alpha > 0$, 
% and
$U(h)$ is a local free energy relating to the wettability property of the substrate \citep{bertozzi2001dewetting} and the evaporation and condensation effects. When the substrate is partially wetting or hydrophobic, a simple free energy is chosen as 
\blue{
$
U(h)= \frac{1}{3}(\frac{\epsilon}{h})^3
-\frac{1}{2}(\frac{\epsilon}{h})^2 - \pstar h,
$
}
and the corresponding disjoining pressure $\Pi(h)$ are given by
\begin{equation}
\label{disj}
U'(h)=\Pi(h)-\pstar,\qquad
\Pi(h)= \frac{\epsilon^2}{h^3}\left(1-\frac{\epsilon}{h}\right).
\end{equation}
Here, \blue{the parameter $\epsilon$ in $\Pi(h)$ sets a positive $O(\epsilon)$ lower bound for the liquid height at which the attractive van der Waals forces balance with the short-range Born repulsion. This lower bound also determines the thickness of a precursor layer connecting the droplets, which is commonly assumed in thin-film literature to model the behavior of the contact line and liquid films on a prewetted layer
\citep{oron2001dynamics,bertozzi2001dewetting,ji2018instability,dukler2020theory}.}
\blue{
The constant parameter $\pstar$ gives the influence of the temperature difference between the liquid film and the surrounding vapor phase. When the film is uniformly heated or cooled from below with a constant temperature imposed at the solid-liquid interface, $\pstar$ remains constant in space and time \citep{ji2018instability}. 
% \begin{equation}
%     \pstar = \Theta^*-\Theta
% \label{eq:pstar}
% \end{equation}
% gives the influence of the temperature difference between the liquid film and the surrounding vapor phase, where $\Theta = \Theta(x, y, t)$ is the rescaled temperature at the fluid-solid interface, $\Theta^*$ is the liquid saturation temperature
% , $C = \mu U/\sigma$ is the capillary number 
}

The dynamic pressure $P$ of the free surface is given by
\begin{equation}
 P(h) = \frac{\delta \mathcal{E}}{\delta h}=\Pi(h)-\pstar- \alpha^2{\nabla^2 h},
 % P(h) = \frac{\delta \mathcal{E}}{\delta h}=\Pi(h)-\pstar- {\nabla^2 h},
\label{eq:pressure_condense}
\end{equation}
where $\alpha^2\nabla^2 h$ gives the linearized curvature of the free surface.
\end{subequations}
We also have the following energy dissipation property: 
\begin{equation}
\label{e-law}
\frac{d\E}{dt} = - \mathcal{I}(h) \le 0, 
% \int_\Omega (V_1(h) |\nabla P|^2+V_2(h)P^2)~dx  \le 0.
\end{equation}
where the dissipation functional  
% above right-hand side (dissipation rate)
\begin{align}
\label{fisher}
\mathcal{I}(h) :=
\int_\Omega \left(V_1(h)|\nabla P(h)|^2 + V_2(h)|P(h)|^2\right)dx
\end{align}
is often named the generalized {\it Fisher information functional}. 

% paragraph on active droplet
\blue{
Droplets laden with a suspension of active matter, known as \emph{active drops}, have also been the focus of many studies in fluid mechanics \citep{marchetti2013hydrodynamics,maass2016swimming}. These active droplets consist of internally driven units or self-propelled particles that draw energy from the surrounding and induces active stresses to the background fluids, leading to more complex dynamics and pattern formation \citep{joanny2012drop}. Many studies have focused on the experiments, modeling, and fundamentals of active fluids in various applications \citep{loisy2019tractionless,trinschek2020thin,whitfield2016instabilities,aditi2002hydrodynamic}. For instance, \cite{adkins2022dynamics} experimentally and analytically studied the phase-separating fluid mixtures of active liquid interfaces driven by mechanical activities.
\cite{chandel2024evaporation} discussed the spontaneous puncturing of active droplets induced by evaporation-driven mass loss. \cite{shankar2022optimal} studied the optimal transport and control of mass-conserving active drops by controlling the activity.
For a comprehensive review, readers are referred to \cite{michelin2023self}. The coupling of non-mass-conserving dynamics and internal dynamics of active drops presents opportunities for designing new control mechanisms. 
In this work, we consider the mean field control of droplet dynamics by controlling the activity field.}

Despite the wealth of modeling and analytical results on droplet dynamics, the field of controlling these free surface flows is still in its early stages of development.
For instance, researchers have explored reduced-order-model-based control of liquid films governed by the classical Kuramoto-Sivashinsky (KS) equation, employing distributed control across the whole domain
\citep{armaou2000feedback, christofides2000global, lee2005reduced}. Boundary control and optimal control of the KS equation have also been studied in the works of \cite{katz2020finite,al2018linearized,coron2015fredholm, maghenem2022boundary,liu2001stability,tomlin2019optimal}.  

The literature on controlling thin-film equations is relatively limited. 
For example, \red{the work of
% \cite{wray2014electrostatic} 
\cite{WRAY2015172}
studied the control of evaporating particle laden droplets via an electric field to suppress the ``coffee-stain'' effect.}
\cite{klein2016optimal} investigated optimal control of a simplified thin-film equation with only the fourth-order term.  The work of \cite{samoilova2019feedback} considered a linear proportional control for suppressing the Marangoni instability in a thin liquid film evolving on a plane.
\cite{cimpeanu2021active} proposed an active control strategy of liquid film flows by incorporating information from reduced-order models. The work of \cite{wray2022electrostatic} focused on the electrostatic control for thin films underneath an inclined surface. 
\cite{shankar2022optimal} studied optimal transport and control of droplets of an active fluid. More recently, 
\cite{BISWAL2024133942} studied the optimal boundary control of a thin-film equation describing thin liquid films flowing down a vertical cylinder. The mean field control of reaction-diffusion equations \citep{Mielke11,FuOsherLi,LiLeeOsher22} and regularized conservation laws \citep{li2021controlling,li2022controlling} have been studied. In this direction, a recent work of \cite{gao2024mean} also discussed the control of coherent structures in turbulent flows using mean field games.

In this study, we demonstrate the application of mean field control (MFC) techniques for manipulating droplet dynamics within the framework of thin-film equations. The objective of MFC is to design and transport the 
% formulation of 
droplets governed by the classical lubrication theory. 
 See Figure \ref{fig:schematic} for an example of the transport and deformation of a droplet on a two-dimensional spatial domain from the initial surface height profile $h_0(x,y)$ to the target height profile \blue{$h_{T}(x,y)$}.
\blue{To demonstrate the application of optimal control and motivate the design of the mean field control formulation, in Section \ref{sec:model} we derive a lubrication model for a thin volatile active liquid film, whose dynamics can be controlled through its activity field.}
We then illustrate the formulation of optimal control of thin-film equations as below.  The constraint is given with the background of original physical dynamics, where the control variables contain both vector field and source terms with the above-mentioned nonlinear mobility functions  $V_1(h)$ and $V_2(h)$. The minimization is then taken under the kinetic energy originating from the generalized Fisher information functional, adding suitable potential energy and terminal functionals. We then derive two equivalent formulations, for the latter one we develop the minimization systems of the proposed MFC problems in Proposition \ref{MFCD}. They can be viewed as the forward-backward controlled systems of thin-film dynamics.

\begin{figure}
    \centering
\includegraphics[width=0.5\textwidth]{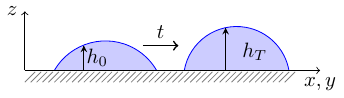}
    \caption{Schematic figure for the mean field control from the initial droplet profile $h_0(x,y)$ to the target droplet profile \blue{$h_{T}(x,y)$}.}
    \label{fig:schematic}
\end{figure}

We remark that the proposed MFC problem is motivated by the optimal transport theory \citep{villani2008optimal}. We study the optimal control problem associated with the gradient flow formulation from the thin-film equation in a generalized Wasserstein space. In particular, the control formulation itself is a generalization of the Benamou-Brenier formula \citep{benamou2000computational}, where we further consider the evolution of fluid dynamics under the thin-film equation as background. The proposed minimization system is the generalization of Wasserstein-2 type geodesics.

%{\color{red}Need a sketch of MFC system.}

In simulations, our approach also utilizes high-order finite element computations to achieve this objective. Compared to previous work in \cite{FuOsherLi}, we remark that 
the second-order Laplacian term in the dynamic pressure $P$ in \eqref{eq:pressure_condense} brings additional difficulties in the simulation of proposed MFC problems. We construct several new constraints and Lagrange multipliers associated with primal-dual hybrid gradient methods \citep{Chambolle, CarrilloWangWei2023_structurea}
to handle the constraints associated with the dynamic pressure $P$.

The structure of the paper is as follows. 
\blue{
In Section \ref{sec:model}, the model for viscous volatile thin films with an active suspension on a partially wetting substrate is formulated.}
In Section \ref{sec2}, we discuss the mean field control of droplet dynamics using the formulated model. In Section \ref{sec3}, the high-order space-time finite element discretization and its associated primal-dual hybrid gradient optimization solver for the proposed mean field control problem is presented. Numerical results for the mean field control of droplet dynamics using the developed high-order finite element computations are presented in Section \ref{sec4}, followed by concluding remarks and discussions in Section \ref{sec5}.

\section{Model formulation}
\label{sec:model}
\blue{
In this section, we develop a lubrication model for a thin volatile liquid film laden with an active suspension on a two-dimensional solid substrate. The liquid properties, including surface tension $\sigma$, dynamic viscosity $\mu$, and density $\rho$, are assumed constant. We follow the work of \cite{ajaev2005evolution} and \cite{ ji2018instability} to describe the evaporation and condensation effects using a one-sided model. To incorporate the active suspension into the film dynamics, we follow the approach of \cite{shankar2022optimal} and  consider an active stress that is proportional to the film thickness and describes strong ordering along the $x, y$ directions.
We show below the derivation of the governing equations and discuss the combined effects of active stresses and non-mass-conserving phenomena.
}

\blue{Following \cite{ajaev2001steady}, we choose the scales for the system as follows: the characteristic length scale $\mathscr{L}$ in the $x, y$ direction is set as the initial radius of the droplet $R_0$. The length scale $\mathscr{H}$ in the vertical direction $z$ is $C^{1/3}R_0$, where $C = \mu U/\sigma$ is the capillary number. We assume that the aspect ratio $\epsilon = \mathscr{H}/\mathscr{L} = C^{1/3} \ll 1$. The velocity scale in the $x, y$ direction is $U = k T_s^* / (\rho \mathcal{L} R_0)$, where $k$ is the thermal conductivity of the liquid, $\mathcal{L}$ is the latent heat of vaporization per unit mass, and $T_s^*$ is the saturation temperature. The characteristic vertical velocity is $C^{1/3}U$, and the pressure and time scales are given by $C^{1/3}\sigma/R_0$ and $R_0/U$, respectively. Given a dimensional temperature field $T^*$, we express the scaled non-dimensional temperature as $T = (T^* - T_s^*)/(C^{2/3}T_s^*)$.
Additionally, we define the in-plane position $\bm{x} = (x,y)$, $\nabla_{\perp} = (\partial_x, \partial_y)$, and the non-dimensional velocity field of the fluid $\bm{u} = (\bm{u}_{\perp}, w)$, where $\bm{u}_{\perp}$ represents the velocity in the $x$ and $y$ directions.
}

\blue{Under the lubrication approximation with $\epsilon\ll 1$, the non-dimensional Stokes equation reduces to the leading-order equations
\begin{equation}
    -\nabla_{\perp} \hat{P}  + \frac{\partial^2}{\partial z^2}\bm{u}_{\perp} = 0, \quad
    -\frac{\partial \hat{P}}{\partial z} = 0,
\label{eq:stokes_leadingOrder}
\end{equation}
where $\hat{P}$ is the non-dimensional pressure of the liquid.
At the solid-liquid interface $z = 0$, we impose the no-slip and no-penetration boundary conditions
\begin{equation}
\bm{u}_{\perp} = \bm{0}, \quad w = 0, \quad \text{at } z = 0.
\label{eq:NS-BCs}
\end{equation}
The kinematic boundary condition at the free interface $z = h(t,\bm{x})$ is given by
\begin{equation}
    \frac{\partial h}{\partial t} = w - \bm{u}_{\perp}\cdot \nabla_{\perp} h - J, \quad \text{at } z = h(t, \bm x),
\label{eq:kinematicBC}
\end{equation}
where $J$ represents the evaporative flux due to evaporation or condensation effects.
Using the incompressibility condition $\nabla\cdot \bm{u} = 0$ and boundary conditions \eqref{eq:NS-BCs} in the kinematic boundary condition \eqref{eq:kinematicBC}, we derive
\begin{equation}
    \frac{\partial h}{\partial t} + \nabla_{\perp} \cdot \left(\int_0^h \bm{u}_{\perp}~dz\right) = -J.
\label{eq:conservation}
\end{equation}
We assume that the temperature field $T(\bm{x},z)$ across the liquid is quasi-static and is linear in $z$ in the leading-order under the lubrication approximation, satisfying
\begin{equation}
    \frac{\partial^2 T}{\partial z^2} = 0.
\label{eq:temp1}
\end{equation}
This assumption is valid for liquid films under weak evaporation or condensation effects \citep{burelbach1988nonlinear}. 
At the solid-liquid interface, we impose the boundary condition
\begin{equation}
    T = \Theta \quad \text{at } z = 0,
\label{eq:temp_BC}
\end{equation}
where $\Theta$ represents a scaled temperature difference between the solid-liquid interface and the saturation temperature. 
For ${\Theta} < 0$, we anticipate dominant evaporation dynamics, while for large $\Theta > 0$, we expect condensation dynamics. For simplicity, we assume that $\Theta$ is constant in space and time.}

\blue{
At the liquid-air interface, the conservation of energy and the shear stress condition are expressed as
\begin{equation}
    J = -\frac{\partial T}{\partial z}, \quad \frac{\partial \bm{u}_{\perp}}{\partial z} = 0.
\label{eq:energy_shear}
\end{equation}
The flux $J$ is related to the scaled interfacial temperature $T^i$ and pressure jump at the liquid-vapor interface by
\begin{equation}
    KJ = \gamma (\hat{P}-P_v) + T^i,
\label{eq:J_P}
\end{equation}
where $P_v$ is the non-dimensional vapor pressure. The constants $K$ and $\gamma$ are defined by
$K = (\rho U \sqrt{2\pi \bar{R}T_s^*})/(2\rho_v \mathcal{L}C^{1/3})$ and $\gamma = \sigma/(\mathcal{L}\rho R_0 C^{1/3})$, where $\rho_v$ is the vapor density, and $\bar{R}$ is the gas constant per unit mass.
Solving equations \eqref{eq:temp1} -- \eqref{eq:J_P} yields the form of the evaporative flux
\begin{equation}
    J = \frac{\gamma(\hat{P}-P_v) + \Theta}{K+h}.
\label{eq:Jflux_form}
\end{equation}
From equations \eqref{eq:stokes_leadingOrder} and the boundary conditions \eqref{eq:NS-BCs}$_1$ and \eqref{eq:energy_shear}$_2$, we obtain
\begin{equation}
    \bm{u}_{\perp} = \tfrac{1}{2}(z^2 - 2hz)\nabla_{\perp}\hat{P}.
\end{equation}
}
\par \blue{We adopt an active stress tensor $\bm\tau^a = \hat{\eta} \zeta h (\hat{\bm n}\hat{\bm n} - \bm{I}/3)$ to account for the contribution of the active suspension to the stress tensor of the liquid, where $\hat{\bm n}$ is the orientation field of the active agents in the suspension \citep{aditi2002hydrodynamic}. 
Here, $\zeta(t, \bm x)$ represents the activity field originating from the forcing exerted by the active suspension, and $\hat{\eta}$ is a scaling parameter. The activity field $\zeta(t, \bm x)$ can take either sign: positive and negative values of $\zeta$ correspond to contractile and extensile stresses \citep{joanny2012drop}. This model assumes that the active stress depends on the local density of suspension and is applicable to droplets of coherently swimming suspensions and ordered collection of filaments \citep{loisy2019tractionless}. We further assume that the rapid orientational relaxation in the vertical direction is negligible, and the orientation field is almost parallel to the substrate with $\hat{\bm n} \simeq (n_1(t,\bm x), n_2(t, \bm x), 0)$, where the components $n_1$ and $n_2$ take the vertically averaged values for the activity strength
\citep{trinschek2020thin, shankar2022optimal}.}

\blue{The total stress in the liquid 
$\bm{\tau} = -(\hat{P}-\Pi(h))\bm{I} + \mu[\bm \nabla \bm u +  (\bm \nabla \bm u)^T]+ \bm{\tau}^a$
incorporates the liquid pressure, the disjoining pressure $\Pi(h)$, a viscous stress, and the active stress. 
The balance of normal stresses at the liquid-vapor interface in the leading order gives
\begin{equation}
    \hat{P} = P_v + \Pi(h) - \nabla_{\perp}^2 h - \eta h \zeta(t,\bm x) ,
\label{eq:normalStress}
\end{equation}
where $\Pi(h)$ is the disjoining pressure, $\nabla_{\perp}^2 h$ is the linearized surface tension, $P_v$ originates from the stress tensor of the vapor, the parameter $\eta = \hat{\eta}/3$, and the vapor recoil is neglected \citep{oron1997long}.
Substituting \eqref{eq:Jflux_form} -- \eqref{eq:normalStress} into equation \eqref{eq:conservation}
yields the evolution equation for $h$,
\begin{equation}
 \frac{\partial h}{\partial t} = \nabla\cdot \left[ h^3\nabla (\Pi(h) - \nabla^2 h - \eta\zeta(t,\bm x) h)\right] - \frac{\gamma\left(\Pi(h) - \nabla^2 h - \eta\zeta(t,\bm x) h - \pstar\right)}{K+h}.
\label{eq:evo}
\end{equation}
Here, we have replaced $\nabla_{\perp}$ by $\nabla$ for simplicity and rescaled the variables by $t \to 3t$, and $\gamma \to \gamma/3$
% and $\Theta\to \Theta/(3\gamma)$ 
to absorb a constant factor of $3$ in the mobility function, and the constant parameter $\pstar$ is given by $\pstar = \Theta/(3\gamma)$. This model assumes that the active suspension in the droplet is not influenced by the temperature field and neglects Marangoni effects. 
}

\blue{In this work, we consider the mean field control of the droplet dynamics via the activity field $\zeta(t, \bm{x})$. To separate the control variables from the uncontrolled ones, it is convenient to rewrite equation \eqref{eq:evo} as 
\begin{equation}
    \frac{\partial h}{\partial t} + \nabla \cdot (V_1(h)\bm{v}_1) - V_2(h) v_2 = \nabla\cdot \left[ V_1(h)\nabla P(h)\right] - V_2(h)P(h),
\label{generic_model}
\end{equation}
where the mobility functions $V_1(h)$ and $V_2(h)$ are defined in \eqref{mob}, 
$\bm{v}_1$ and $v_2$ encode the control variables to the system,
\begin{equation}
\bm v_1 = \nabla(\eta\zeta(t,\bm x) h),
\quad v_2 = \eta\zeta(t,\bm x) h,
% \quad v_2 = \eta\zeta(t,\bm x) h + \Theta(t) - \pstar,
\end{equation}
and $P(h)$ is the dynamic pressure for the volatile inactive thin film model defined in \eqref{eq:pressure_condense} with $\alpha = 1$.
}
\blue{For volatile inactive fluids with $\zeta\equiv 0$,
% and $\Theta \equiv \pstar$, 
the model \eqref{generic_model} is consistent with the volatile thin-film model \eqref{eq:model}.}

\blue{
The control variable $\zeta$ enters the system through both a diffusion term and a source term. In the mass-conserving case where $\gamma = 0$ (hence $V_2\equiv 0$), the activity field $\zeta$ affects the local contraction or spreading of the droplet, similar to the setting investigated in \cite{shankar2022optimal}. When $\gamma > 0$, the interplay between the activity field and the evaporation/condensation effects can lead to more complex and interesting dynamics. 
One important quantity is the total mass of the fluid, $\mathcal{M}(t)$, and its rate of change due to the non-mass-conserving contributions. By applying Neumann boundary conditions and integrating equation \eqref{generic_model}, we obtain
\begin{equation}
\label{mc}
    \mathcal{M}(t) = \int_{\Omega} h~d\bm x, \quad \frac{d\mathcal{M}}{dt} = \gamma \int_{\Omega} \frac{\zeta h - P}{K + h}~d\bm x,
\end{equation}
which indicates that for $\gamma > 0$, the mass may locally increase or decrease depending on the local relative importance of the activity field and the pressure, i.e., $\zeta h > P$ or $\zeta h < P $, respectively.
}

% \blue{
% To generalize the model \eqref{volatileActive_main} to other droplet models with controllable mass-conserving and non-mass-conserving fluxes, we consider the evolution equation
% \begin{equation}
%     \frac{\partial h}{\partial t} + \nabla \cdot (V_1(h)\bm{v}_1) - V_2(h) v_2 = \nabla\cdot \left[ V_1(h)\nabla P(h)\right] - V_2(h)P(h),
% \label{generic_model}
% \end{equation}
% where the drift vector field $\bm v_1$ and the reaction rate $v_2$ encode the controls to the system.
% For $\bm{v}_1 = \nabla(\zeta h)$, $v_2 = \zeta h + \Theta $, equation \eqref{generic_model} reduces to \eqref{volatileActive_main}. For nonactive fluid with a constant temperature field, $\zeta\equiv 0$ and $\Theta \equiv 0$, the model \eqref{generic_model} is consistent with the volatile thin-film model studied in \citep{ji2018instability, ajaev2005evolution}.
% }

\blue{
We remark that
the PDE \eqref{generic_model}, with varying functional forms of mobility functions, pressure, and controls, can be adapted to other types of control problems for both mass-conserving and non-mass-conserving thin-film models. While most existing works on thin-film control \citep{klein2016optimal, samoilova2019feedback,shankar2022optimal}
focus on the mass-conserving case with $V_2\equiv 0$, here we illustrate a few examples considered in the literature.
}
\begin{example}
The work of
% Klein and Prohl 
\cite{klein2016optimal} addresses an optimal control problem in the divergence form,
\begin{equation}
\label{eq:klein}
    \partial_t h = \partial_x (\lambda |h|^{a}\partial_x P) + \partial_x u , \quad P = -h_{xx},
\end{equation}
where $u(t,x)$ is the external control, $a > 1$, and $\lambda > 0$. \blue{This problem characterizes the control of thin film deposition on silicon wafers during electronic chip fabrication.}
The equation \eqref{eq:klein} is related to the model \eqref{generic_model} with $V_1 = \lambda |h|^{a}$, $V_2 \equiv 0$, and $\bm v_1 = -u/V_1$.
% and $\beta=1$. 
\end{example}
\begin{example}
The work of 
% Samoilova and Nepomnyashchy 
\cite{samoilova2019feedback} aimed to suppress the Marangoni instability in a thin film heated from below using a lubrication equation
\begin{equation}
    \partial_t h = \nabla\cdot \left[\tfrac{1}{3}h^3\nabla P + \tfrac{Ma}{2}h^2\nabla (\Theta-h)\right],
\label{eq:Marangoni_control}
\end{equation}
coupled with a heat transfer equation for the controlled temperature $\Theta$. One can rewrite \eqref{eq:Marangoni_control} into the form of \eqref{generic_model} by setting $V_1 = h^3/3$, $V_2 \equiv 0$, and
$\bm v_1 = \tfrac{3Ma}{2h}\nabla(h-\Theta)$.
% and $\beta=1$.
\end{example}
\begin{example}
In the recent work on optimal transport and control of active droplets by 
% Shankar et al. 
\cite{shankar2022optimal}, the active droplet is modelled by 
\begin{equation}
    \partial_t h = \partial_x \left[\tfrac{1}{3\eta}h^3\partial_x (P -\zeta h)\right], \quad P = -\gamma h_{xx},
\end{equation}
where $\zeta(t,x)$ represents the controllable activity of suspension in the droplet. Again, this problem corresponds to \eqref{generic_model} with $V_1 = h^3/(3\eta)$, $V_2\equiv 0$, and $\bm v_1 = (\zeta h)_x$.
% and $\beta=1$.
\end{example}
\blue{
\begin{remark}[Rescaling]
    For numerical studies throughout the remainder of the paper, we set the computational domain to be a unit square $\Omega = [0,1]^2$ for convenience. By rescaling the spatial scale $\bm{x} \to L\bm{x}$ and the time scale $t \to L^2 t$, from \eqref{eq:model} we obtain the rescaled model for volatile inactive liquid films,
$
    \frac{\partial h}{\partial t} = \nabla\cdot\left(V_1(h)\nabla P\right) - V_2(h) P$. Here 
    the rescaled dynamic pressure is given in \eqref{eq:pressure_condense}
% \begin{equation}
% \label{eq-r}
%     \frac{\partial h}{\partial t} = \nabla\cdot\left(V_1(h)\nabla P\right) - V_2(h) P,\quad P = \Pi(h) - \pstar - \alpha^2 \nabla^2 h,
% \end{equation}
with the constant $\alpha = 1/L$, where $L$ is the domain length before rescaling. 
For the model with the activity field
% and temperature field 
\eqref{generic_model}, to emphasize the importance of the control variables, we introduce the rescaling $t \to \beta L^2 t$, $\bm x \to L\bm x$, $\gamma \to \gamma/L^2$,  where $\beta = 1/\eta$. This rescaling leads to the following scaled model:
\begin{subequations}
\label{eq-c}
\begin{equation}
    \frac{\partial h}{\partial t} + \nabla \cdot (V_1(h)\bm{v}_1) - V_2(h) v_2 = \beta\left[\nabla\cdot \left( V_1(h)\nabla P\right) - V_2(h)P\right], 
    % \quad P = \Pi(h) - \pstar - \alpha^2 \nabla^2 h,
\label{scaled_generic_model}
\end{equation}
where the vector field $\bm v_1$ and the source term $v_2$ are defined as
% \begin{equation}
% \label{v1v2}
%  \bm v_1 = \nabla(\zeta(t,\bm x) h),\quad v_2 = \zeta(t,\bm x) h + \Theta(t) - \pstar.   
% \end{equation}
\begin{equation}
\label{v1v2}
 \bm v_1 = \nabla(\zeta(t,\bm x) h),\quad v_2 = \zeta(t,\bm x) h.   
\end{equation}
% and the mobility functions $V_1(h)$ and $V_2(h)$ are given in \eqref{mob}.
\end{subequations}
\end{remark}
We refer to the system \eqref{eq-c} as the active control form of the thin-film equation.
In the next section, we introduce a mean field control model, which selects the active fields 
$\bm v_1$ and $v_2$ (hence the active field $\zeta$) in an optimal manner in certain metrics.
}

\section{Mean field control of droplet dynamics}\label{sec2}
This section presents the main formulation of mean field control (MFC) problems for the thin-film equation \eqref{eq:model}. 
We follow our previous work on MFC for (second-order) reaction-diffusion systems \citep{fu2023generalized}. 
A byproduct of our MFC formulation is a new 
Jordan–Kinderlehrer–Otto (JKO) scheme for the PDE \eqref{eq:model}, which is similar to the variational time implicit scheme discussed in \cite{FuOsherLi}; see Remark \ref{jko} below. 

\blue{We note that while the general form of the proposed MFC problems has a similar structure to MFC for reaction-diffusion systems considered in our earlier work \citep{fu2023generalized},
two new challenges emerge for MFC of  \eqref{eq:model}. First, the energy functional \eqref{Energy} contains the gradient of the surface height, $\nabla h$, making \eqref{gradientFlow} a fourth-order PDE. Second, both mobility functions $V_1(h)$  and $V_2(h)$ are {\it convex} functions of $h$ (see \eqref{mob}), which make the MFC problem a {\it nonconvex} optimization problem \blue{(See Remark \ref{rk:convexity} below)}. These new features make the numerical discretization of the MFC for droplet dynamics significantly more challenging than the one for the second-order reaction-diffusion case.
}
% (see Remark X below).
% {\color{red} [Fix remark number]}

The current work mainly focuses on the MFC formulation of droplet dynamics and its associated high-order finite element discretization. We will address the first challenge and show how the MFC framework for reaction-diffusion systems developed in \cite{fu2023generalized} can be naturally adopted here using additional {\it auxiliary variables}. A corresponding high-order space-time finite element discretization and its solution procedure using the primal-dual hybrid gradient (PDHG) method will be presented in Section \ref{sec3}. We leave theoretical investigations on the (non)convexity issue of the proposed mean field control problem for future work. 

 \subsection{Droplet dynamics induced distances and MFCs}
 The energy dissipation law \eqref{e-law} and its associated Fisher information functional \eqref{fisher} naturally induce a metric distance between two positive surface heights $h_0$ and $h_1$, as we define in the following. 
\setcounter{theorem}{0}
\begin{definition}[Distance functional]
\label{Dis}
% \noindent\textbf{Definition}: {\em Scalar distance functional.} 
Define a distance functional $\mathrm{Dist}_{V_1,V_2}\colon \mathcal{M}\times \mathcal{M}\rightarrow\mathbb{R}_+$ as below, where 
the space $\mathcal{M} = \{h\in L^1(\Omega): h\geq 0\}$. Consider the following optimal control problem: 
\begin{subequations}\label{DisZ} 
\begin{equation}\label{Dis1}
\mathrm{Dist}_{V_1, V_2}(h_0, h_1)^2:=\inf_{h, \bm v_1, v_2}\quad\int_0^1\int_\Omega \left(|\bm v_1|^2 V_1(h)+ |v_2|^2V_2(h)\right)dxdt, 
\end{equation}
where the infimum is taken among $h(t,x)\colon [0,1]\times\Omega\rightarrow\mathbb{R}_+$,  $\bm v_1(t,x)\colon  {[0,1]}\times \Omega \rightarrow\mathbb{R}^d$, 
$v_2(t,x)\colon  [0,1]\times \Omega \rightarrow\mathbb{R}$, 
such that $h$ satisfies a reaction-diffusion type equation with drift vector field $\bm v_1$, drift mobility $V_1$, reaction rate $v_2$, reaction mobility $V_2$, connecting initial and terminal surface heights $h_0$, ${h_1}\in\mathcal{M}$: 
\begin{equation}\label{Dis2}
\left\{\begin{aligned}
&\partial_t h + \nabla\cdot( V_1(h) \bm v_1)=V_2(h)v_2,\quad (t,x)\in {[0,1]}\times \Omega,\\
&h(0, x)=h_0(x),\quad {h(1,x)=h_1(x)},
\end{aligned}\right.
\end{equation}
with no-flux boundary condition $V_1(h)\bm v_1\cdot\nnu|_{\partial\Omega}=0$.
\end{subequations}
\end{definition}
\blue{\begin{remark}[On convexity]
\label{rk:convexity}
We illustrate here that the optimization problem in Definition \ref{Dis} is a  nonconvex optimization problem with a linear constraint. To do so, we introduce 
variables $\bmm = V_1(h)\bmv_1$ and
$s= V_2(h)v_2$, which make the constraint \eqref{Dis2} a linear equation: $\partial_t h + \nabla\cdot\bmm -s = 0$.
Then the objective functional in \eqref{Dis1}
is \[
\mathcal{L}(h, \bmm, s) = 
\int_0^1\int_\Omega \left(\frac{|\bmm|^2}{V_1(h)}+ \frac{|s|^2}{V_2(h)}\right)dxdt.
\]
From the convexity of mobility functions $V_1(h)$ and $V_2(h)$, we can show that both terms $\frac{|\bmm|^2}{V_1(h)}=\sum_{i=1}^d\frac{|m_i|^2}{V_1(h)}$
and 
$\frac{s^2}{V_1(h)}$ fail to be convex in the respective variables. Here the vector $\bmm = (m_1,\cdots, m_d)$. In fact, the determinant of the Hessian matrix for the term $\frac{|m_1|^2}{V_1(h)}$ for the variables $(h,m_1)$ is 
\[
\det \begin{bmatrix}
\left(    2\frac{V_1'(h)^2}{V_1(h)^{3}}-
    \frac{V_1''(h)}{V_1(h)^{2}}\right)|m_1|^2&
    -2\frac{V_1(h)'}{V_1(h)^2}m_1 &\\
    -2\frac{V_1(h)'}{V_1(h)^2}m_1 & 
    2V_1^{-1}
\end{bmatrix}
= -2\frac{V_1''(h)}{V_1(h)^{3}}|m_1|^2\le 0,
\]
by positivity and convexity of $V_1(h)$. 
Hence, the term $\frac{|m_1|^2}{V_1(h)}$  is not convex.
\end{remark}
}
Using the above-defined distance functional and the thin-film equation \eqref{eq:model}, we define the following
mean field control (MFC) problem for droplet dynamics.
\begin{definition}[MFC for droplet dynamics]
\label{mfc}
Given a time domain $[0,T]$, $T>0$, a potential functional $\mathcal{F}\colon \mathcal{M}\rightarrow\mathbb{R}$, 
and a terminal functional $\mathcal{G}\colon \mathcal{M}\rightarrow\mathbb{R}$, 
\begin{subequations}\label{mfcA}
consider 
\begin{equation}\label{mfcA1}
\begin{split}
&\inf_{h, \bm v_1, v_2}\int_0^T\!\!\left[\int_\Omega 
\frac{1}{2}\left(|\bm v_1|^2 V_1(h)+ |v_2|^2V_2(h)\right)dx - \mathcal{F}(h)\right]\!dt+\mathcal{G}(h(T,\cdot)), 
\end{split}
\end{equation}
where the infimum is taken among $h(t,x)\colon [0,T]\times\Omega\rightarrow\mathbb{R}_+$,  $\bm v_1(t,x)\colon  [0,T]\times \Omega \rightarrow\mathbb{R}^d$, 
$v_2(t,x)\colon  [0,T]\times \Omega \rightarrow\mathbb{R}$,  such that
\begin{equation}\label{mfcA2}
\partial_t h  +\nabla\cdot( V_1(h) \bm v_1)-V_2(h)v_2=\beta
\left[\nabla\cdot(V_1(h)\nabla P(h))-V_2(h) P(h)\right],  
\end{equation}
with  boundary condition 
\begin{align}
\label{mfcA3}
\left.V_1(h) (\bm v_1-\beta\nabla P(h))\cdot\nnu\right|_{\partial\Omega} =0,
\end{align}
and  initial surface height $h(0,\cdot) = h_0$ in $\Omega$. 
Here $\beta\ge0$ is a non-negative number, which represents the strength of 
the droplet dynamics \eqref{eq:model} in the constraint of mean field control problem \eqref{mfcA}.
\end{subequations}
\end{definition}

\begin{remark}[JKO temporal discretization to \eqref{eq:model}]
\label{jko}
In the above definition, if we take $T=1$, $\mathcal{F}=0$ and $\mathcal{G}(h) = \Delta t \mathcal{E}(h)$ as in \eqref{Energy}, 
and set parameter $\beta = 0$, we obtain 
a dynamic formulation of the celebrated JKO temporal discretization scheme \cite{JKO} for the gradient flow \eqref{eq:model}, which is a first-order variational time-implicit discretization with stepsize $\Delta t>0$. See \cite{FuOsherLi,LiWang22,LiLuWang20} for a related discussion on JKO-type discretizations for gradient flows in Wasserstein-type metric spaces.
\end{remark}

\subsection{MFC reformulations}
In this subsection, we focus on reformulations of the MFC problem in Definition \ref{mfc} which will  be suitable for a finite element discretization.

The first reformulation converts the constraint PDE \eqref{mfcA2} to a linear constraint by a change of variables.  
Specifically, introducing the flux function $\bmm(t,x)\colon [0,T]\times\Omega\rightarrow\mathbb{R}^d$ and source function $s(t, x)\colon [0,T]\times\Omega\rightarrow\mathbb{R}$, such that 
\begin{equation}
\label{change}
\bmm=V_1(h)\Big[\bm v_1 -\beta \nabla P(h)\Big], \quad s=V_2(h)\Big[v_2-\beta P(h)\Big], 
\end{equation}
then the MFC problem \blue{in Definition \ref{mfc}}
is equivalent to the following linearly constrained optimization problem:
% \begin{definition}[MFC reformation I]
% \label{mfcR1}
Given a potential functional $\mathcal{F}\colon \mathcal{M}\rightarrow\mathbb{R}$, 
and a terminal functional $\mathcal{G}\colon \mathcal{M}\rightarrow\mathbb{R}$, 
consider 
\begin{equation}\label{mfc1}
\begin{split}
\inf_{h, \bmm, s}&\int_0^T\!\int_\Omega 
\tfrac12|\tfrac{\bmm}{V_1(h)}+\beta\nabla P(h)|^2 V_1(h)dxdt \\
&
\!\!\!\!\!\!+\int_0^T\!\int_\Omega 
\tfrac12|\tfrac{s}{V_2(h)}+\beta P(h)|^2V_2(h)dxdt
-\int_0^T\mathcal{F}(h)dt+\mathcal{G}(h(T,\cdot)), 
\end{split}
\end{equation}
where the infimum is taken among functions $h, \bmm, s$,  such that
\begin{equation}\label{mfcR1}
\partial_t h  +\nabla\cdot \bmm-s=0 \;\text{ on }
[0,T]\times \Omega, \; 
\left.\bmm\cdot\nnu\right|_{[0,T]\times\partial\Omega} =0, \;
h(0,x) = h_0(x). 
\end{equation}
Expanding the product terms in \eqref{mfc1}, we get
\begin{equation}
% \label{mfcA1R2}
\begin{split}
\inf_{h, \bmm, s}&\int_0^T\!\int_\Omega 
\left(
\frac{|\bmm|^2}{2V_1(h)}
+\frac{|s|^2}{2V_2(h)}
\right)+
\beta\left(\bmm\cdot\nabla P(h)+s\cdot P(h)
\right)
dxdt \\
&
+\int_0^T\!\int_\Omega 
\frac{\beta^2}2
\left(|\nabla P(h)|^2V_1(h)
+|P(h)|^2 V_2(h)
\right)
dxdt \\
&
-\int_0^T\mathcal{F}(h)dt+\mathcal{G}(h(T,\cdot)). 
\end{split}
\end{equation}
Using integration by parts and the constraint \eqref{mfcR1}, we have 
\begin{align}
   & \int_0^T\!\int_\Omega 
    \left(\bmm\cdot\nabla P(h)+s\cdot P(h)
\right)
dxdt \nonumber 
\\ = & \;
 \int_0^T\!\int_\Omega 
  P(h)  \left(-\nabla\cdot\bmm+s
\right)
dxdt \nonumber 
\\
=& \;
 \int_0^T\!\int_\Omega 
  P(h)  \partial_t h\;
dxdt= \mathcal{E}(h(T,\cdot))-
  \mathcal{E}(h_0),
\end{align}
where we used the definition of dynamic pressure $P(h)=\tfrac{\delta \mathcal{E}(h)}{\delta h}$ in the last step.
Combining these derivations and noticing that $h_0$ is given, we arrive at  
the following equivalent formulation of the MFC problem \blue{in Definition \ref{mfc}}.
\begin{definition}[MFC \blue{reformulation} I]
\label{mfcR}
Consider 
\begin{equation}\label{mfcR2}
\begin{split}
\inf_{h, \bmm, s}&\quad \int_0^T\int_\Omega 
\left(
\frac{|\bmm|^2}{2V_1(h)}
+\frac{|s|^2}{2V_2(h)}
\right)dxdt\\
&
+\int_0^T\!\int_\Omega 
\frac{\beta^2}2
\left(|\nabla P(h)|^2V_1(h)
+|P(h)|^2 V_2(h)
\right)
dxdt \\
&
-\int_0^T\mathcal{F}(h)dt+\mathcal{G}(h(T,\cdot))
+\beta \mathcal{E}(h(T,\cdot)), 
\end{split}
\end{equation}
where the infimum is taken among $h, \bmm, s$
satisfying \eqref{mfcR1}.
\end{definition}

\begin{proposition}[MFC systems of droplet dynamics]\label{MFCD}
Let $(h, \bmm, s)$ be the critical point system of the MFC problem \eqref{mfcR2}. Then there exists a function $\phi\colon [0, T]\times\Omega\rightarrow \mathbb{R}$, such that 
\begin{equation}
\label{opt-m}
\frac{\bmm(t,x)}{V_1(h(t,x))}=\nabla \phi(t,x),\quad \frac{s(t,x)}{V_2(h(t,x))}=\phi(t,x),
\end{equation}
and 
\begin{equation}\label{SBPI}
\left\{\begin{aligned}
&\partial_th(t,x) +\nabla\cdot(V_1(h(t,x))\nabla \phi(t,x))-V_2(h(t,x))\phi(t,x)=0,\\
&\partial_t\phi(t,x)+\frac{1}{2}\|\nabla \phi(t,x)\|^2 V_1'(h(t,x))+\frac{1}{2}|\phi(t,x)|^2V'_2(h(t,x))\\
&\hspace{5.8cm}+\frac{\delta}{\delta h}\Big[\mathcal{F}(h)-\frac{\beta^2}{2}\mathcal{I}(h)\Big](t,x)=0,
\end{aligned}\right.
\end{equation}
where 
$\mathcal{I}(h)$ is the generalized Fisher information functional given in \eqref{fisher}, such that 
\begin{equation}
\begin{split}
\frac{\delta}{\delta h}\mathcal{I}(h)=&\quad \frac{1}{2}(V_1'(h)|\nabla P(h)|^2+V_2'(h)|P(h)|^2)-\nabla\cdot(V_1(h)\nabla P(h))\Pi'(h)\\    
&+\alpha^2\Delta \nabla\cdot(V_1(h)\nabla P(h))+\Pi'(h)V_2(h)P(h)-\alpha^2\Delta (P(h)V_2(h)), 
\end{split}
\end{equation}
with initial and terminal time conditions 
%{\color{red}[boundary conditions?]}
\begin{equation}
   h(0,x)=h_0(x),\quad \phi(T,x)=-\frac{\delta}{\delta h}\Big(\mathcal{G}(h(T,\cdot))+\beta \mathcal{E}(h(T,\cdot))\Big).  
\end{equation}
%%In addition, 
\end{proposition}
We remark that a similar MFC formulation and system for reaction-diffusion equation was considered in our earlier work \citep{FuOsherLi}. We present the derivation of the MFC system \eqref{SBPI} in Appendix \ref{appA}.

MFC problems and systems are generalizations of Benamou-Brenier formulas in optimal transport \citep{villani2008optimal}. This refers to setting $\beta=0$, $V_1(h)=h$, and $V_2(h)=0$. 
\blue{In the context of mean field control of droplet dynamics, we need to address additional challenges, in which the dynamic pressure $P(h)$ given in \eqref{eq:pressure_condense} involves a second-order Laplacian term. The role of this Laplacian term is well known in lubrication models, and its treatment in both forward simulations \citep{witelski2003adi} and (boundary) optimal control
contexts \citep{BISWAL2024133942} is well understood.
\red{In addition, the treatment of the Laplacian term in controlling free interfaces has been discussed in the context of the KS equation \citep{tomlin2019point} and weighted residual integral boundary layer models \citep{wray2022electrostatic}.}
However, the Laplacian term brings additional difficulties in the computation of the proposed MFC problem using finite element methods. This is from the fact that we need to approximate the forward-backward mean field control system \eqref{SBPI}, in which the Laplacian term is enforced in the formulation of generalized Fisher information functional \eqref{fisher}.} We also comment that the dynamic pressure $P(h)$ is essential in modeling the disjoining pressure and surface tension that govern the droplet dynamics.  In numerical experiments, we demonstrate that the MFC problem with this pressure term exhibits essential patterns of droplets, including droplet spreading, transport, merging, and splitting.

We introduce additional auxiliary variables 
to further reformulate the MFC problem \eqref{mfcR}.
Let $\bmn(t,x): [0,T]\times \Omega\rightarrow \mathbb{R}^d$, 
$p(t,x): [0,T]\times \Omega\rightarrow \mathbb{R}$, 
and 
$\bmq(t,x): [0,T]\times \Omega\rightarrow \mathbb{R}^d$ be defined as follows:
\begin{align}
\label{aux}
\bmn = \alpha\nabla h, \;\;
p = -\alpha\nabla\cdot \bmn, \;\;
\bmq = \alpha \nabla p.
\end{align}
This implies $p = -\alpha\nabla\cdot(\alpha\nabla h) = -\alpha^2\nabla^2h$. \blue{Hence, the dynamic pressure 
$P(h)$ and its gradient $\nabla P(h)$ can be expressed as follows:
\begin{align}
    \label{dPh}
    P(h) = U'(h)+p, \quad 
    \nabla P(h) = U''(h)\nabla h + \nabla p=
\frac1{\alpha}(U''(h)\bmn+\bmq).
\end{align}
}
% satisfies 
% $P(h) = U'(h) + p$, and 
% its gradient follows  
% \begin{equation}
% \label{dPh}
% \nabla P(h) = U''(h)\nabla h + \nabla p=
% \frac1{\alpha}(U''(h)\bmn+\bmq).
% \end{equation}
Plugging these relations back to the MFC problem \eqref{mfcR2}, we obtain the following equivalent reformulation.
\begin{definition}[MFC reformulation I\!I]
\label{mfcZ}
Consider 
\begin{equation}\label{mfcZ2}
\begin{split}
\inf_{h, \bmm, s, \bmn, p, \bmq}&\int_0^T\!\int_\Omega 
\left(
\frac{|\bmm|^2}{2V_1(h)}
+\frac{|s|^2}{2V_2(h)}
\right)dxdt\\
&
\!\!\!\!\!\!\!\!\!\!\!\!\!\!\!\!+\int_0^T\!\int_\Omega 
\frac{\beta^2}{2}
\left(\tfrac{|U''(h)\bmn+\bmq|^2}{\alpha^2}V_1(h)
+|U'(h)+p|^2 V_2(h)
\right)
dxdt \\
&
\!\!\!\!\!\!\!\!\!\!\!\!\!\!\!\!-\int_0^T\mathcal{F}(h)dt+\mathcal{G}(h(T,\cdot))
+\beta 
\int_{\Omega}\left(U(h(T,x)) + \frac{|\bmn(T,x)|^2}{2}\right)dx
, 
\end{split}
\end{equation}
where the infimum is taken among $h, \bmm, s$
satisfying \eqref{mfcR1}, and $\bmn, p, \bmq$ satisfying 
\eqref{aux}.
\end{definition}
To simplify the notation, we collect the variables 
into a 
% big 
vector
\begin{align}
\label{uu}
    \bmu := (h, \bmm, s, \bmn, p, \bmq), 
\end{align}
and introduce $\bmu_T:= (h_T, \bmn_T)$
where $h_T:\Omega\rightarrow \mathbb{R}_+$ is the terminal surface height, and
$\bmn_T:\Omega\rightarrow \mathbb{R}^d$
 the scaled surface height gradient at terminal time.
Hence $\bmu(t,x):[0,T]\times \Omega\rightarrow\mathbb{R}^{3d+3}$ is a space-time function with $3d+3$ components, and 
$\bmu_T(x):\Omega\rightarrow\mathbb{R}^{d+1}$ is a spatial function with 
$d+1$ components.
 We further denote the functionals $H(\bmu)$ 
 and $H_T(\bmu_T)$, such that 
 \begin{subequations}
\label{Hu}     
 \begin{align}
\label{Hu1}
H(\bmu) := &\;
\frac{|\bmm|^2}{2V_1(h)}
+\frac{|s|^2}{2V_2(h)}
+\frac{\beta^2}{2\alpha^2}
|U''(h)\bmn+\bmq|^2V_1(h)
\nonumber \\
&\;+\frac{\beta^2}{2}|U'(h)+p|^2 V_2(h) - F(h),\\
\label{Hu2}
H_T(\bmu_T):= &\;
\beta\left(U(h_T)
+\frac{|n_T|^2}{2}\right)
+G(h_T),
 \end{align}
 \end{subequations}
where $F(h)$ and $G(h)$ are density functions for functionals $\mathcal{F}$ and $\mathcal{G}$, i.e., 
\begin{equation}
\mathcal{F}(h) = \int_\Omega F(h)dx, \quad
\mathcal{G}(h) = \int_\Omega G(h)dx.   
\end{equation}
Using these notations, the MFC problem \ref{mfcZ}
takes the following compact form:
\begin{definition}[MFC problem: compact form]
\label{mfcC}
Consider 
\begin{subequations}
\label{mfcC0}
\begin{equation}\label{mfcC1}
\begin{split}
\inf_{\bmu, \bmu_T}&\int_0^T\!\int_\Omega H(\bmu)
dxdt +\int_{\Omega}H_T(\bmu_T)dx, 
\end{split}
\end{equation}
subject to the constraints on the space-time domain
\begin{align}
    \label{mfcC2}
    \begin{split}
    \partial_t h + \nabla\cdot\bmm - s = &\;0,\\
    -\bmn + \alpha \nabla h = &\;0,\\
    p + \alpha\nabla\cdot\bmn = &\;0,\\
    \bmq - \alpha\nabla p = &\;0,
\end{split}
\quad\quad\quad\text{ on } [0,T]\times\Omega,
\end{align}
and the constraints at terminal time
\begin{align}
    \label{mfcC3}
    \begin{split}
    h(0,\cdot) = h_0,\\
    h(T,\cdot) = h_T,\\
    -\bmn_T + \alpha \nabla h_T = &\;0,\\
\end{split}
\quad\quad\quad\text{ on } \Omega,
\end{align}
with Neumann boundary condition 
\begin{align}
    \label{mfcC4}
    \bmm\cdot\nnu = 0 \quad \text{ on }[0,T]\times\partial\Omega.
\end{align}
\end{subequations}
\end{definition}
\begin{remark}
We note that all the MFC formulations above are mathematically equivalent. Our numerical discretization, however, will be constructed based on the last formulation in Definition \ref{mfcZ} or Definition \ref{mfcC}. It has a form that the constraints are linear PDEs, and the objective function does not involve spatial derivatives. These two properties are crucial for the efficient implementation of the finite element scheme that we will develop in Section \ref{sec3}.   
\end{remark}

\blue{
\begin{remark}[On recovering the physics control variable $\zeta$]
\label{rk:physics}
Combining the definition in \eqref{change} with the optimality conditions \eqref{opt-m} in Proposition \ref{MFCD}, we get 
    \[
    \bm v_1 - \beta \nabla P(h) = \nabla \phi, \quad 
    v_2-\beta P(h) = \phi.
    \]
    In particular, this implies that $\bm v_1$ is a gradient field, and 
    the relation $\bm v_1 = \nabla v_2$ holds.
Using the equations \eqref{v1v2} in Section \ref{sec:model} that
relate the control variables $\bm v_1$ and $v_2$ 
to the activity field $\zeta$,
we get  
\begin{align}
\label{zeta}
    \zeta = \frac{\beta P(h)+\phi}{h} = \frac{\beta(U'(h)+p)+\phi}{h},
\end{align}
where we used the relation \eqref{dPh} in the last equality.
\end{remark}
}

\subsection{Saddle-point problem}
Finally, we reformulate the constrained optimization problem in Definition \ref{mfcC} into a saddle-point problem using Lagrange multipliers, for which a finite element discretization will be developed in Section \ref{sec3}. 
We introduce the following four Lagrange multipliers on the space-time domain $[0,T]\times \Omega$ 
for the four equations in \eqref{mfcC2}. They are
scalar functions
$\phi(t,x):[0,T]\times\Omega\rightarrow\mathbb{R}$, $\xi(t,x):[0,T]\times\Omega\rightarrow\mathbb{R}$, 
vectorial functions 
$\bm{\sigma}(t,x):[0,T]\times\Omega\rightarrow\mathbb{R}^d$, $\bm{\theta}(t,x):[0,T]\times\Omega\rightarrow\mathbb{R}^d$, 
and a Lagrange multiplier $\bm{\sigma}_T(x):\Omega\rightarrow\mathbb{R}^d$
on the spatial domain (at terminal time):
\begin{alignat*}{2}
    \partial_t h + \nabla\cdot\bmm - s = &\;0 &&\;\;\longleftrightarrow \phi,\\
    -\bmn + \alpha \nabla h = &\;0&&\;\;\longleftrightarrow \bm{\sigma},\\
    p + \alpha\nabla\cdot\bmn = &\;0&&\;\;\longleftrightarrow \xi,\\
    \bmq - \alpha\nabla p = &\;0&&\;\;\longleftrightarrow \bm{\theta},\\
        -\bmn_T + \alpha \nabla h_T = &\;0&&\;\;\longleftrightarrow \bm{\sigma}_T.
\end{alignat*}
Then the MFC problem \ref{mfcC} can be formulated as the following saddle-point problem: 
\begin{equation}
\begin{split}
\inf_{\bmu, \bmu_T}
\sup_{\bPphi, \bm{\sigma}_T}&\quad\int_0^T\!\int_\Omega H(\bmu)
dxdt +\int_{\Omega}H_T(\bmu_T)dx\\ 
&+\int_0^T\!\int_\Omega \Big[(\partial_t h + \nabla\cdot\bmm - s)\phi 
+(-\bmn + \alpha\nabla h)\cdot\bm{\sigma}\Big]dxdt\\
&+\int_0^T\!\int_\Omega \Big[(p + \alpha\nabla\cdot\bmn)\xi+
(\bmq-\alpha\nabla p)\cdot\bm{\theta}
\Big]
dxdt\\
&+\int_\Omega (-\bmn_T + \alpha\nabla h_T)\cdot\bm{\sigma}_T
dx,
\end{split}
\end{equation}
with the following boundary and initial/terminal conditions
\begin{equation}
h(0,\cdot) = h_0, \;\; h(T,\cdot) = h_T \quad \text{ in }\Omega, \quad\text{ and } \bmm\cdot\nnu = 0 \text{ on }[0,T]\times\partial\Omega.
\end{equation}
Here $\bPphi = (\phi, \xi, \bm{\sigma}, \bm{\theta})$.

Next, applying integration by parts on the above saddle-point problem to move all derivatives of $\bm u$ to the dual variables $\bPphi$,
% Lagrange multipliers $\phi$, $\bm{\sigma}$, 
and using the initial and boundary conditions, we obtain 
\begin{equation}
\label{saddle}
\begin{split}
\inf_{\bmu, \bmu_T}
\sup_{\bPphi, \bm{\sigma}_T}&\int_0^T\!\int_\Omega H(\bmu)
dxdt +\int_{\Omega}H_T(\bmu_T)dx\\ 
&\!\!\!\!\!\!\!\!\!\!\!\!+\int_0^T\!\int_\Omega \Big[-( h \partial_t\phi+ \bmm\cdot\nabla\phi + s\phi) 
-(\bmn\cdot\bm{\sigma} + \alpha h\nabla\cdot\bm{\sigma})\Big]dxdt\\
&\!\!\!\!\!\!\!\!\!\!\!\!+\int_0^T\!\int_\Omega \Big[(p \xi- \alpha\bmn\cdot\nabla \xi)+
(\bmq\cdot\bm{\theta}+\alpha p\nabla \cdot\bm{\theta})
\Big]
dxdt\\
&\!\!\!\!\!\!\!\!\!\!\!\!+\int_\Omega\Big[-(\bmn_T \cdot\bm{\sigma}_T+ \alpha h_T \nabla\cdot\bm{\sigma}_T)
+h_T\phi(T, x)-h_0\phi(0, x)\Big]
dx.
\end{split}
\end{equation}
In the above formulation, we assume the Lagrange multipliers $\bm{\sigma}$, $\xi$, and 
$\bm{\theta}$ satisfy the following Neumann boundary conditions:
\begin{equation}
 \bm{\sigma}\cdot\nnu = 0, \;\;
   \nabla{\xi}\cdot\nnu = 0, \;\;
  \bm{\theta}\cdot\nnu = 0, 
  \text{ on }[0,T]\times\partial\Omega. 
\end{equation}

The saddle-point problem \eqref{saddle} is the final form of our MFC problem that will be discretized in the next section. The variational structure of this problem makes the finite element method an ideal candidate for such a problem.
We close this section with a discussion on the proper function spaces 
for the primal variables $\bmu$ and $\bmu_T$ and dual variables $\bPphi$
and $\bm{\sigma}_T$ in \eqref{saddle} which makes the integrals in \eqref{saddle} valid.
The spaces are given as follows
\begin{subequations}
    \label{space}
    \begin{alignat}{2}
        \bmu\in &\;\;\Big\{
        \bm{v}\in [L^2([0,T]\times \Omega)]^{3d+3}:\quad
        \int_0^T\!\!\!\!\int_{\Omega}H(\bm{v})dxdt <+\infty,\\
        \nonumber
        &\;\; \quad\quad\quad\quad\quad\quad\quad\quad\quad\text{first component of }\bm{v} \text{ is non-negative}
        \Big\}, \\
        \bmu_T\in &\;\;\Big\{
        \bm{v}_T\in [L^2(\Omega)]^{d+1}:\quad
        \int_{\Omega}H_T(\bm{v}_T)dx <+\infty,
        \\
        \nonumber
        &\;\; \quad\quad\quad\quad\quad\quad\quad\quad\quad\text{first component of }\bm{v}_T \text{ is non-negative}
        \Big\}, \\
        \bPphi\in &\;\;V_{\phi}\times \bm{V}_{\sigma}\times 
        {V}_{\xi}\times \bm{V}_{\theta}, \quad\quad
        \bm{\sigma}_T\in \;\;H_0(\mathrm{div}; \Omega),
    \end{alignat}
where 
\begin{align}
    V_{\phi}=& H^1([0,T]\times\Omega),\\ 
    \bm{V}_{\sigma}=\bm{V}_{\theta}=& L^2([0,T])\otimes H_0(\mathrm{div};\Omega), \\
    {V}_{\xi}=& L^2([0,T])\otimes H_0^1(\Omega).
\end{align}
\end{subequations}
Here we use the usual definition of Sobolev spaces
\begin{align}
    L^2(\Omega):=&\{v:\Omega\rightarrow \mathbb{R}: \;\;\int_{\Omega}|v|^2dx <+\infty\},\\
    H^1(\Omega):=&\{v\in L^2(\Omega): \nabla v \in [L^2(\Omega)]^d\},\\
    H(\mathrm{div}; \Omega):=&\{\bm{v}\in [L^2(\Omega)]^d: \nabla \cdot \bm{v} \in L^2(\Omega)\}.
\end{align}
Moreover, $H^1_0(\Omega)$ is the subspace of $H^1(\Omega)$ with a {\it zero} boundary condition, 
and $H_0(\mathrm{div}; \Omega)$ is the subspace of $H(\mathrm{div}; \Omega)$ with a {\it zero} boundary condition on the normal direction.

\section{High order discretizations and optimization algorithms}\label{sec3}
This section presents the high-order spatial-time finite element discretization and its associated primal-dual hybrid gradient (PDHG) 
optimization solver for the proposed MFC saddle-point problem \eqref{saddle}.
\subsection{The high-order finite element scheme}
We first partition the spatial domain $\Omega$ into a spatial mesh $\Omega_h=\{K_\ell\}_{\ell=1}^{N_S}$ with $N_S$ elements where each element $K_{\ell}$ is assumed to be a mapped hypercube in $\mathbb{R}^d$, 
and the temporal domain $[0,T]$ into a temporal mesh $I_h=\{I_j\}_{j=1}^{N_T}$ with $N_T$ segments. Denote the space-time mesh as 
$\Omega_{T,h}=I_h\otimes \Omega_h$.
The function spaces in \eqref{space} for the saddle-point problem \eqref{saddle} indicate natural discretization spaces for the primal and dual variables. 
In particular, we use the following {\it conforming} finite element spaces to discretize  the dual variables $\bPphi$ and $\bm{\sigma}_T$: 
\begin{subequations}
    \label{d-space}
    \begin{align}
        V_{\phi,h}^{k+1} =&\; \{\psi\in V_{\phi}: \;\; 
        v|_{I_j\times K_\ell}
        \in Q^{k+1}(I_j)\otimes Q^{k+1}(K_\ell)\;\;
\forall j, \ell\},\\
        \bm{V}_{\sigma,h}^k =&\; \{\bm{\tau}\in \bm{V}_{\sigma}: \;\; 
        \bm{\tau}|_{I_j\times K_\ell}
        \in Q^k(I_j)\otimes RT^k(K_\ell)\;\;
\forall j, \ell\},\\
        {V}_{\xi,h}^{k,k+1} = &\;\{q\in \bm{V}_{\xi}: \;\; 
        q|_{I_j\times K_\ell}
        \in Q^k(I_j)\otimes Q^{k+1}(K_\ell)\;\;
\forall j, \ell\},\\
        \bm{M}_{\sigma,h}^k = &\;\{\bm{\tau}\in H_0(\mathrm{div};\Omega): \;\; 
        \bm{\tau}|_{K_\ell}
        \in RT^k(K_\ell)\;\;\forall j, \ell\},
    \end{align}
    where $Q^k(K_\ell)$ is the tensor-produce polynomial space of degree no greater than $k$ in each direction, and 
    $RT^k(K_\ell)$ is the local Raviart-Thomas finite element space 
    \cite{Boffi13}
    on the mapped hypercube $K_\ell$, for $k\ge 0$.
    For the primal variables $\bmu$ and $\bmu_T$, it is natural to use 
   an integration rule space such that they are 
    defined {\it only } on the (high-order) numerical integration points, 
    since no derivative calculation is needed for these variables. 
    Let $X_{t,k}:=\{\chi_{t,i}\}_{i=1}^{N_{T,q}^k}$ be the quadrature points 
    and $\{\omega_{t,i}\}_{i=1}^{N_{T,q}^k}$ the corresponding quadrature weights on the temporal mesh $I_h$ using $(k+1)$ Gauss-Legendre (GL) 
    integration points per line segment, and denote 
    $X_{s,k}:=\{\chi_{s,j}\}_{j=1}^{N_{S,q}^k}$ be the quadrature points 
    and $\{\omega_{s,j}\}_{j=1}^{N_{S,q}^k}$ the corresponding quadrature weights on the spatial mesh $\Omega_h$ using $(k+1)$ Gauss-Legendre (GL)  integration points per coordinate direction in each element. 
    We approximate each component of $\bmu$ and $\bmu_T$ using the following space-time and spatial integration rule spaces, respectively:
    \begin{align}
        \label{u-space}
    W_h^k :=\Big\{v: X_{t,k}\times X_{s,k}\rightarrow \mathbb{R}\Big\},
    \quad 
        M_h^k :=\Big\{v: X_{s,k}\rightarrow \mathbb{R}\Big\}.
    \end{align}
    Note that a function in the quadrature space $W_h^k$
    can be interpreted as a vector of size $N_{T,q}^k\times N_{S,q}^k$, and 
    a function in $M_h^k$ can be interpreted as a vector of size $N_{S,q}^k$.     
\end{subequations}

Using the above finite element spaces, we define the following discrete saddle-point problem:
Find the critical point of the discrete system
\begin{equation}
\label{saddleH}
\begin{split}
\inf_{\bmu_h, \bmu_{T,h}}
\sup_{\bPphi_h, \bm{\sigma}_{T,h}}& 
\intst{H(\bmu_h)}
+
\ints{H_T(\bmu_{T,h})}
\\ 
&\!\!\!\!\!\!\!\!\!\!\!\!\!\!\!\!\!\!\!\!\!\!\!\!- \intst{h_h \partial_t\phi_h+ \bmm_h\cdot\nabla\phi_h + s_h\phi_h}-\intst{\bmn_h\cdot\bm{\sigma}_h + \alpha h_h\nabla\cdot\bm{\sigma}_h}\\
&\!\!\!\!\!\!\!\!\!\!\!\!\!\!\!\!\!\!\!\!\!\!\!\!+\intst{p_h \xi_h- \alpha\bmn_h\cdot\nabla \xi_h}+
\intst{\bmq_h\cdot\bm{\theta}_h+\alpha p_h\nabla \cdot\bm{\theta}_h}\\
&\!\!\!\!\!\!\!\!\!\!\!\!\!\!\!\!\!\!\!\!\!\!\!\!-\ints{\bmn_{T,h} \cdot\bm{\sigma}_{T,h}+ \alpha h_{T,h} \nabla\cdot\bm{\sigma}_{T,h}}
+\ints{h_{T,h}\phi_h(T, x)-h_0\phi_h(0, x)}
\end{split}
\end{equation}
where the variables
$\uu_h:=(h_h, \mm_h, \so_h, \nn_h, p_h, q_h)\in [W_h^k]^{3d+3}$
with $h_h\ge0$, $\uu_{T,h}=(h_{T,h}, \nn_{T,h})\in [M_h^{k}]^{d+1}$
with $h_{T,h}\ge 0$, 
$\Pphi_h:=(\pphi_h,\ssigma_h, \xi_h, \bm{\theta}_h)
\in V_{\phi,h}^{k+1}\times \bm{V}_{\sigma}^k
\times V_{\xi, h}^{k,k+1}\times 
\bm{V}_{\sigma}^k$, and $\ssigma_{T,h}\in \bm{M}_{\sigma,h}^k$.
Here the double bracket is the numerical integration on the space-time domain 
and single bracket is the numerical integration on the spatial domain 
defined as follows:
\begin{align}
    \label{int}
    \ints{f(x)} = \sum_{j=1}^{N_{s,j}^k}f(\chi_{s,j})\omega_{s,j}, 
    \quad 
        \intst{f(t, x)} = \sum_{i=1}^{N_{t,i}^k}\sum_{j=1}^{N_{s,j}^k}f(\chi_{t,i}, \chi_{s,j})\omega_{t,i}\omega_{s,j}. 
\end{align}
\blue{After solving for the discrete saddle-point problem \eqref{saddleH}, we can recover the 
dynamic pressure $P_h = U'(h_h)+p_h$, 
and the activity field $\zeta(t, \bm x) = \frac{\beta P_h + \phi_h}{h_h}$; see Remark \ref{rk:physics}. 
% $\bm{v}_1$ and $v_2$ using the formulas \eqref{change} and \eqref{dPh}:
% \begin{align}
%     \bm{v}_1 = \frac{\bmm_h}{V_1(h_h)}+\frac{\beta}{\alpha}(U''(h_h)\bmn_h+\bmq_h), \quad 
%     {v}_2 = \frac{s_h}{V_2(h_h)}+{\beta}(U'(h_h)+p_h).
% \end{align}
}

\subsection{A generalized PDHG algorithm}
We solve the discrete saddle point problem \eqref{saddleH}
using a generalized preconditioned PDHG algorithm, which is a splitting algorithm that solve for the primal variables $\bmu_h$ and $\bmu_{T,h}$, 
and each component of the dual variables $\Pphi_h$ and $\ssigma_{T,h}$
sequentially. The following algorithm is a generalization of the G-prox PDHG algorithm developed in \cite{jacobs2019solving}. 

\begin{algorithm}
\caption{Generalized PDHG  for \eqref{saddleH}.}
\label{alg:1}
\begin{algorithmic}[1]
\STATE Choose initial guesses $\bPphi_h^0, \ssigma_{T,h}^0,
\uu_h^0, \uu_{T,h}^0$, and parameters $\sigma_{\phi}, \sigma_{u}>0$.
\STATE \textbf{for} $\ell=0,1,\cdots$ \textbf{do}
\STATE \quad\quad Compute $\phi_h^{\ell+1}
\in V_{\phi,h}^{k+1}$, 
$\ssigma_h^{\ell+1}\in \bm{V}_{\sigma,h}^{k}$,
$\xi_h^{\ell+1}\in {V}_{\xi,h}^{k,k+1}$, 
$\bm{\theta}_h^{\ell+1}\in \bm{V}_{\sigma,h}^{k}$, and 
 $\ssigma_{T,h}^{\ell+1}\in \bm{M}_{\sigma,h}^k$ 
such that they are the solutions to the following minimization problems:
\begin{subequations}
    \label{pdhg}
\begin{align}
\label{phi-solve}
&&\begin{split}
 \argmin_{\phi_h\in V_{\phi,h}^{k+1}}~& \frac{1}{2\sigma_{\phi}} \stint{|\partial_t(\phi_h-\phi_h^\ell)|^2+ |\nabla(\phi_h-\phi_h^\ell)|^2 + |\phi_h-\phi_h^\ell|^2}\\
 &
 % \!\!\!\!\!\!\!\!\!\!\!\!\!\!\!
 +\frac{1}{2\sigma_{\phi}} \sint{|\phi_h(T,\cdot)-\phi_h^\ell(T,\cdot)|^2}\\
 &
 +\stint{
 h_h^\ell \partial_t\phi_h+ \bmm_h^\ell\cdot\nabla\phi_h + s_h^\ell\phi_h}\\
 % -\sint{h_{T,h}^\ell \phi_h(T,\cdot) -h^0\phi_h(0,\cdot)},\\
 &
 -\sint{h_{T,h}^\ell \phi_h(T,\cdot) -h^0\phi_h(0,\cdot)},
\end{split}\\
% \end{align}
% \begin{align}
\label{sigma-solve}
&&\begin{split}
 \argmin_{\ssigma_h\in \bm{V}_{\sigma,h}^{k}}~& \frac{1}{2\sigma_{\phi}} \stint{|\ssigma_h-\ssigma_h^\ell|^2+ |\partial_t(\phi_h^{\ell+1}-\phi_h^{\ell})+\alpha\nabla\cdot(\ssigma_h-\ssigma_h^\ell)|^2}\\
 &+\stint{
\nn_h^\ell\cdot\ssigma_h+\alpha h_h^\ell\nabla\cdot\ssigma_h},
\end{split}\\
% \end{align}
% \begin{align}
\label{xi-solve}
&&\begin{split}
 \argmin_{\xi_h\in V_{\xi,h}^{k,k+1}}~& \frac{1}{2\sigma_{\phi}} \stint{|\xi_h-\xi_h^\ell|^2+ |(\ssigma_h^{\ell+1}-\ssigma_h^\ell)
 +\alpha\nabla(\xi_h-\xi_h^\ell)|^2}\\
 &-\stint{
 p_h^\ell\xi_h-\alpha\nn_h^\ell\cdot\nabla\xi_h},
\end{split}\\
% \end{align}
% \begin{align}
\label{theta-solve}
&&\begin{split}
 \argmin_{\bm{\theta}_h^{\ell+1}\in \bm{V}_{\sigma,h}^{k}}~& \frac{1}{2\sigma_{\phi}} \stint{|\bm{\theta}_h-\bm{\theta}_h^\ell|^2+ |\xi_h-\xi_h^\ell
 +\alpha\nabla\cdot(\bm{\theta}_h-\bm{\theta}_h^\ell)|^2}\\
 &
 -
 \intst{\bmq_h^\ell\cdot\bm{\theta}_h+\alpha p_h^\ell\nabla \cdot\bm{\theta}_h},
\end{split}\\
% \end{align}
% \STATE \quad Compute $\ssigma_{T,h}^{\ell+1}\in \bm{M}_{\sigma,h}^k$ such that  it is the following minimizer
% \begin{align}
\label{sigmaT-solve}
&&\begin{split}
 \argmin_{\ssigma_{T,h}\in \bm{M}_{\sigma,h}^{k}}~& \frac{1}{2\sigma_{\phi}} \sint{|\alpha\nabla\cdot(\ssigma_{T,h}-\ssigma_{T,h}^\ell)
 -(\phi_h^{\ell+1}(T,\cdot)-\phi_h^{\ell}(T,\cdot))
 |^2}\\
 &
\!\!\!\!\!\!+ \frac{1}{2\sigma_{\phi}}\sint{|\ssigma_{T,h}-\ssigma_{T,h}^\ell|^2}
+ \ints{\bmn_{T,h}^\ell \cdot\bm{\sigma}_{T,h}+ \alpha h_{T,h}^\ell \nabla\cdot\bm{\sigma}_{T,h}}.
\end{split}
\end{align}

\STATE \quad \quad Extrapolate $\tilde{\bPphi}_h^{\ell+1}= 2{\bPphi}_h^{\ell+1}-{\bPphi}_h^{\ell}$, and 
$\tilde{\ssigma}_{T,h}^{\ell+1}= 2{\ssigma}_{T,h}^{\ell+1}-{\ssigma}_{T,h}^{\ell}$.
\STATE \quad\quad Compute $\uu_h^{\ell+1}\in [W_h^k]^{3d+3}$ 
and $\uu_{T,h}^{\ell+1}\in [M_h^k]^{d+1}$ 
such that they are the following minimizers:
\begin{align}
\label{u-solve}
&&\begin{split}
     \argmin_{\uu_h\in [W_h^k]^{3d+3}, h_h\ge0}&~ \frac{1}{2\sigma_{u}} \stint{|\uu_h-\uu_h^\ell|^2}
   +\stint{H(\uu_h)}  \\
& - \intst{h_h \partial_t\tilde{\phi}_h^{\ell+1}+ \bmm_h\cdot\nabla\tilde{\phi}_h^{\ell+1} + s_h\tilde{\phi}_h^{\ell+1}}\\
&-\intst{\bmn_h\cdot\tilde{\bm{\sigma}}_h^{\ell+1} + \alpha h_h\nabla\cdot\tilde{\bm{\sigma}}_h^{\ell+1}}\\
&+\intst{p_h \tilde{\xi}_h^{\ell+1}- \alpha\bmn_h\cdot\nabla \tilde{\xi}_h^{\ell+1}}\\
&+
\intst{\bmq_h\cdot\tilde{\bm{\theta}}_h^{\ell+1}+\alpha p_h\nabla \cdot\tilde{\bm{\theta}}_h^{\ell+1}}.
\end{split}\\
\label{ut-solve}
&&\begin{split}
     \argmin_{\uu_{T,h}\in [M_h^k]^{d+1}, h_{T,h}\ge0}&~ \frac{1}{2\sigma_{u}} \sint{|\uu_{T,h}-\uu_{T,h}^\ell|^2}
   +\sint{H_T(\uu_{T,h})}  \\
&
-\ints{\bmn_{T,h} \cdot\widetilde{\bm{\sigma}}_{T,h}^{\ell+1}+ \alpha h_{T,h} \nabla\cdot\tilde{\bm{\sigma}}_{T,h}^{\ell+1}}
+\ints{h_T\widetilde\phi_{T,h}^{\ell+1}}.
\end{split}
\end{align}
\end{subequations}
\end{algorithmic}
\end{algorithm}

\begin{remark}
    The dual variable updates in Algorithm \ref{alg:1}
    are constant-coefficient linear elliptic problems, for which scalable solvers have been well-developed in the literature. 
    Preconditioned conjugate gradient methods are used to solve these coupled elliptic problems with a geometric multigrid preconditioner for the diffusion-type problems \eqref{phi-solve} and \eqref{xi-solve},
    and a low-order preconditioner developed in \cite{Pazner23}
    for the $H(\mathrm{div})$-elliptic problems \eqref{sigma-solve} and \eqref{theta-solve}.
    % We use preconditioned conjugate gradient method with a geometric multigrid preconditioner for the space-time diffusion operator 
    % for $\phi_h$ in \eqref{phi-solve} and for 
    % the (spatial) diffusion operator for $\xi_h$ in \eqref{xi-solve}, and with a simple Jacobi preconditioner
    % for the $H(\mathrm{div})$-elliptic operators in \eqref{sigma-solve} and \eqref{theta-solve}.
    Meanwhile, the primal variable updates in \eqref{u-solve} and 
    \eqref{ut-solve} are nonlinear but decoupled for each degree of freedom on the quadrature point, hence they can be solved efficiently in parallel.
\end{remark} 

\subsection{A fully discrete JKO scheme to the PDE \eqref{eq:model}}
% Finally, we conclude this section by presenting 
\blue{Next we present
a fully discrete JKO scheme and its simplified version for solving the PDE \eqref{eq:model}}.
As mentioned in Remark \ref{jko}, by taking the functionals 
$\mathcal{F}=0$, and $\mathcal{G}(h) = \Delta t \mathcal{E}(h)$, 
and setting terminal time $T=1$ and the parameter $\beta = 0$, 
the MFC optimization problem \blue{in Definition \ref{mfc}} becomes the dynamic formulation of a JKO temporal discretization scheme which advances solution in time with step size $\Delta t$. Since the parameter $\beta=0$, we do not need the auxiliary variables $\bm n, p$, and $\bm q$ in the definition of the functional $H(\bm u)$  in \eqref{Hu1}. Hence the fully discrete scheme \eqref{saddleH} is reduced to the following:
\begin{equation}
\label{saddleH-JKO}
\begin{split}
\inf_{\bmu_h, \bmu_{T,h}}
\sup_{\phi_h, \bm{\sigma}_{T,h}}& 
\intst{\frac{|\bmm_h|^2}{2V_1(h_h)}
+\frac{|s_h|^2}{2V_2(h_h)}}
+
\Delta t\ints{U(h_{T,h})
+\frac{|\bm n_{T,h}|^2}{2}}
\\ 
&- \intst{h_h \partial_t\phi_h+ \bmm_h\cdot\nabla\phi_h + s_h\phi_h}
\\
&-\ints{\bmn_{T,h} \cdot\bm{\sigma}_{T,h}+ \alpha h_{T,h} \nabla\cdot\bm{\sigma}_{T,h}}\\
&+\ints{h_{T,h}\phi_h(T, x)-h_0\phi_h(0, x)},
\end{split}
\end{equation}
where $\bm u_h = (h_h, \bmm_h, s_h)\in [W_h^k]^{d+2}$.
Moreover, the corresponding PDHG Algorithm \ref{alg:1}
will be further simplified where the three elliptic solves in \eqref{sigma-solve}, \eqref{xi-solve}, and \eqref{theta-solve} are not needed, and the pointwise optimization problem \eqref{u-solve}
does not have the $\bm n, p, \bm q$ contributions.

\blue{
Since the fully discrete scheme \eqref{saddleH-JKO} is first-order accurate in time, we may use a one-step approximation 
of the time integrals in \eqref{saddleH-JKO} to reduce the computational cost without sacrificing the first-order temporal accuracy. 
This leads to the following {\it approximated JKO} scheme: given time step size $\Delta t >0$ and approximate solution $h_h^{n-1}\in M_h^k$
at time $t^{n-1}$, find the approximation $h_h^n\in M_h^k$ at next time level $t^n = t^{n-1}+\Delta t$ by solving the saddle-point problem
\begin{equation}
\label{saddleH-JKOa}
\begin{split}
h_h^{n}= \arg\inf_{h_h}\;\inf_{\bmu_h}
\sup_{\phi_h, \bm{\sigma}_{h}}& 
\ints{\frac{|\bmm_h|^2}{2V_1(h_h)}
+\frac{|s_h|^2}{2V_2(h_h)}}
+
\Delta t\ints{U(h_{h})
+\frac{|\bm n_{h}|^2}{2}}
\\ 
&+ \ints{(h_h-h_h^{n-1})\phi_h - \bmm_h\cdot\nabla\phi_h - s_h\phi_h}
\\
&-\ints{\bmn_{h} \cdot\bm{\sigma}_{h}+ \alpha h_{h} \nabla\cdot\bm{\sigma}_{h}},%\\
% &+\ints{(h_{h}-h_0)\phi_h(x)}.
\end{split}
\end{equation}
Here $\bm u_h = (h_h, \bmm_h, s_h, \bmn_h)\in [M_h^k]^{2d+2}$, 
$\bm\sigma_h\in \bm{M}_{\sigma,h}^k $,
and 
$
    \phi_h \in V_h^{k+1}$, where $V_h^{k+1}$ is the following $H^1$-conforming finite element space on the spatial domain $\Omega$:
$V_h^{k+1}:=\{v\in H^1(\Omega): v|_{K_\ell}
\in Q^{k+1}(K_\ell) \quad \forall \ell\}.
$
}

\blue{
The saddle-point problem \eqref{saddleH-JKOa} can be solved using a similar optimization solver as Algorithm \ref{alg:1}.
However, our preliminary numerical results indicate a large number of iterations (see the numerical results in Section 5.1) is needed for the convergence of Algorithm \ref{alg:1} for solving the PDE \eqref{eq:model} when the default optimization parameters $\sigma_\phi=\sigma_u=1$ are used, especially during the time when the droplets start to merge.
Here we present a direct solution strategy for \eqref{saddleH-JKOa} by solving the critical point system using the Newton-Raphson  method. 
Taking the first-order variation with respect to the variables $\bmu_h$, $\phi_h$, and $\bm\sigma_h$, we get the following nonlinear system of equations:
find 
$(h_h, \bmm_h, s_h, \bmn_h)\in [M_h^k]^{2d+2}$, 
$\bm\sigma_h\in \bm{M}_{\sigma,h}^k$
and $\phi_h\in V_h^{k+1}$ such that
\begin{subequations}
    \label{newton}
    \begin{align}
        \ints{(\frac{\bmm_h}{V_1(h_h)}-\nabla\phi_h)\cdot\delta\bmm_h}=&\;0,\quad \forall \delta\bmm_h\in [M_h^k]^d,\\
        \ints{(\frac{s_h}{V_2(h_h)}-\phi_h)\cdot\delta s_h}=&\;0,\quad \forall \delta s_h\in M_h^k,\\
        \ints{\left(\Delta t\bmn_h-\bm\sigma_h\right)\cdot\delta\bmn_h}=&\;0,\quad \forall \delta\bmn_h\in [M_h^k]^d,\\
        \ints{
(-\frac{|\bmm_h|^2}{2V_1^2}V_1'
-\frac{|s_h|^2}{2V_2^2}V_2'
+\Delta t U'(h_h))\cdot\delta h_h} \quad\;\;\nonumber\\
+\ints{(\phi_h-\alpha\nabla\cdot \bm\sigma_h)\cdot\delta h_h
        }=&\;0,\quad \forall \delta h_h\in M_h^k,\\
  \ints{\bmn_{h} \cdot\delta \bm{\sigma}_{h}+ \alpha h_{h} \nabla\cdot\delta\bm{\sigma}_{h}}
  =&\;0,\quad \forall \delta \bm\sigma_h\in \bm M_{\sigma,h}^k,\\
\ints{(h_h-h_h^{n-1})\delta\phi_h - \bmm_h\cdot\nabla\delta\phi_h - s_h\delta\phi_h}
=&\;0,\quad \forall \delta \phi_h\in V_{h}^{k+1}.
    \end{align}
\end{subequations}
The Newton-Raphson method is then used to solve this coupled nonlinear system. 
}

\blue{
\subsection{Discussions}
We conclude this section with a discussion on the novelty and challenges of our modeling and simulation approaches for droplet dynamics.
The major novelty is the introduction of a general MFC model for droplet dynamics in Definition \ref{mfc}. It gives a general mathematical framework for the optimal control and manipulation of the evolution of thin-film droplets. 
% The control variables $\bm v_1$ and $v_2$ is related to the physical active field $\zeta$ and temperature $\Theta$
% via the equation \eqref{v1v2}, which could potentially be used to guide physical control of droplets in real experiments.
However, this generality introduces new challenges.
}

\blue{The first challenge is how to choose appropriate potential and terminal functionals $\mathcal{F}$ and $\mathcal{G}$ in the objective function in \eqref{mfcA1}. 
Ultimately, the choice of these functionals shall depend on the system under consideration. We will study the effect of different choices in our future work.}

\blue{The second challenge is how to link the mathematical MFC formulation in \eqref{mfcA} to actual physical processes and guide physical experiments.
At the current stage, there is still a gap between our formulation and physical experiments. 
Our MFC formulation generates an optimal path of the surface height $h$ in the sense that the objective functional in \eqref{mfcA1}
is minimized, along with two control variables $\bm v_1$ and $v_2$. 
How to connect these control variables with physical processes is an interesting question. 
Here we have provided an initial discussion to partially answer this question in Section \ref{sec:model}, where 
we argue that the control variables $\bm v_1$ and $v_2$ in the MFC model \ref{mfcA} are related to a physical active field $\zeta$ via the equation \eqref{v1v2} for the modeling of viscous volatile thin films with an active suspension on a partially wetting substrate. The physical active field $\zeta$  may be used to guide physical experiments. A comprehensive discussion on the linkage of our MFC
formulation with physical processes will be explored in future work.
}

\blue{
The third challenge is solving the resulting (nonconvex) optimization problem efficiently.
 We have presented a high-order space-time finite element method to discretize the optimization problem 
\eqref{mfcA}. It leads to a discrete saddle-point problem \eqref{saddleH}, and we introduced a first-order optimization solver, Algorithm \ref{alg:1},
to solve this discrete saddle-point problem. However, the issue of numerical convergence of Algorithm \ref{alg:1} is not addressed in this work.
}

\section{Numerical results}\label{sec4}
In this section, we present numerical results for both the PDE \eqref{eq:model} and the mean field control (MFC) problem \blue{in Definition \ref{mfc}}.

% on a two-dimensional spatial domain. 
\blue{
In Subsection \ref{ex1}, we provide numerical results for the PDE \eqref{eq:model} using the approximated JKO scheme \eqref{saddleH-JKOa}.
Both one- and two-dimensional numerical results are provided. We obtain consistent simulation results when compared  with classical FEM simulations. 
We use both the optimization based solver Algorithm \ref{alg:1} and the Newton-Raphson method to solve the discrete saddle-point problem \ref{saddleH-JKOa}. 
Interestingly, we 
observe that the Newton-Raphson method can be orders of magnitude faster than the optimization solver due to the need of large amount of iteration counts for accuracy considerations in the optimization solver. 
Improving the performance of the optimization based solver will be a focus of our future work.
}

\blue{
In Subsection \ref{ex2}, we illustrate examples of numerically solving the MFC problem \blue{in Definition \ref{mfc}} using the high-order finite element scheme 
\eqref{saddleH}. Specifically, we apply Algorithm \ref{alg:1} to solve the discrete saddle-point problem \eqref{saddleH}, where the optimization parameters are taken as $\sigma_\phi=\sigma_u = 1$, starting with $h_h = h_0$ as the initial surface height and setting all other initial variables to be zero.
}

\blue{
The finite element software packages  MFEM \citep{MFEM} and NGSolve \citep{Schoberl16} 
are used in the implementation.
Reproducible sample code and animation videos are available in the GitHub repository \url{https://github.com/gridfunction/DROPLET_MFC}. 
}
% {\color{red}Do we want to show some comparisons with small and large $\beta$? Also, do we want to plot the optimal control for $\bm v_1, v_2$, perhaps for at least one example?}

% Throughout, we take the spatial domain to be a unit square 
% $\Omega=[0,1]^2$,  $\gamma=0.04$ in $V_2(h)$ from \eqref{mob}, $\alpha=0.01$ in the energy functional \eqref{Energy}, $\epsilon=0.3$ and $\pstar=0.5$
% in the disjoining pressure in \eqref{disj}.
\blue{Throughout, we use the following mobility functions and energy functional for the PDE model \eqref{eq:model}: 
\begin{align}
\label{ppp}
V_1(h) = h^3, \;V_2(h) = \frac{0.04}{h+0.1}, 
\mathcal{E}(h) = \int_{\Omega} \frac{10^{-4}}{2}
|{\nabla{h}}|^2 + 
\left(\frac{0.3^3}{3h^3}-
\frac{0.3^2}{2h^2}-0.5 h\right)
\, dx\,,
\end{align}
which corresponds to taking parameters $\gamma=0.04, K=0.1, \alpha=0.01, \epsilon=0.3, \pstar =0.5$
in \eqref{mob}--\eqref{disj}.
}
\blue{We note that by rescaling the variables as $\hat{x} = x /\alpha, ~\hat{y} = y /\alpha,~ \hat{t} = t/{\alpha^2}$, and $\hat{\gamma} =\alpha^2 \gamma $, the original thin-film model \eqref{eq:model} can be rewritten as
\begin{equation}
\label{eq:modified_model}
\frac{\partial h}{\partial \hat{t}} = \nabla\cdot \left( V_1(h)\nabla \hat{P}(h)\right) - \frac{\hat{\gamma}}{h+K}\hat{P}(h), \quad \hat{P}(h) = \Pi(h)-\pstar- {\nabla^2 h}, \text{ on } [0, T] \times\hat{\Omega},
\end{equation}
where the spatial domain $\hat{\Omega} = [0, 1/\alpha]^2$. The form of \eqref{eq:modified_model} is consistent with the volatile thin-film model studied in \cite{ji2018instability,ji2024coarsening}, where typically droplet dynamics are studied over a large spatial domain. 
% Here, we use the model \eqref{eq:model} with the additional parameter $\alpha$ to reduce the size of the computational domain and hence lower the computational cost.
}

\subsection{JKO scheme for the PDE \eqref{eq:model}.}
\label{ex1}

\begin{figure}
\centering
\subfigure[t=0.05]{
\includegraphics[width=0.48\textwidth]{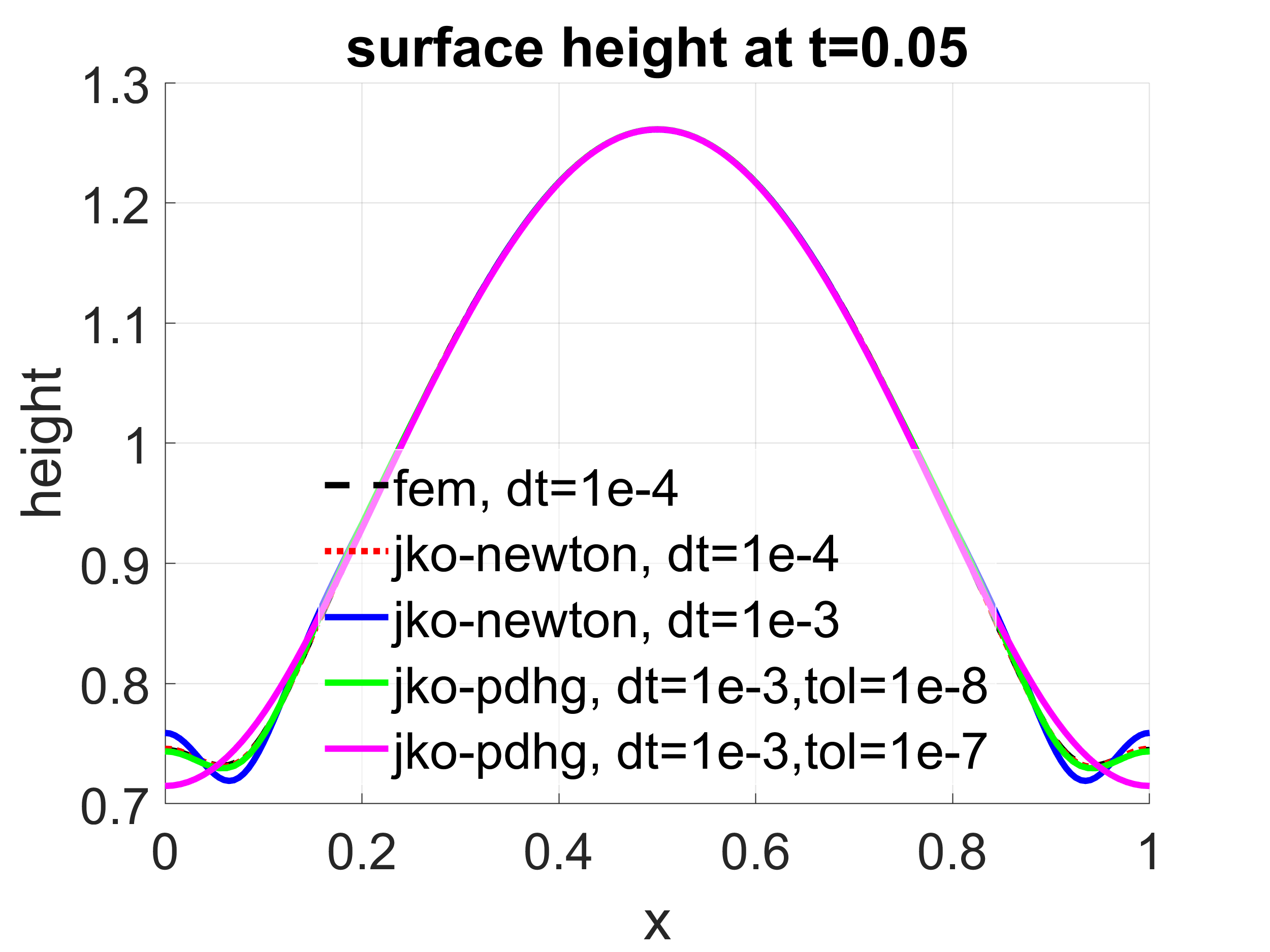}}
\subfigure[t=0.10]{
\includegraphics[width=0.48\textwidth]{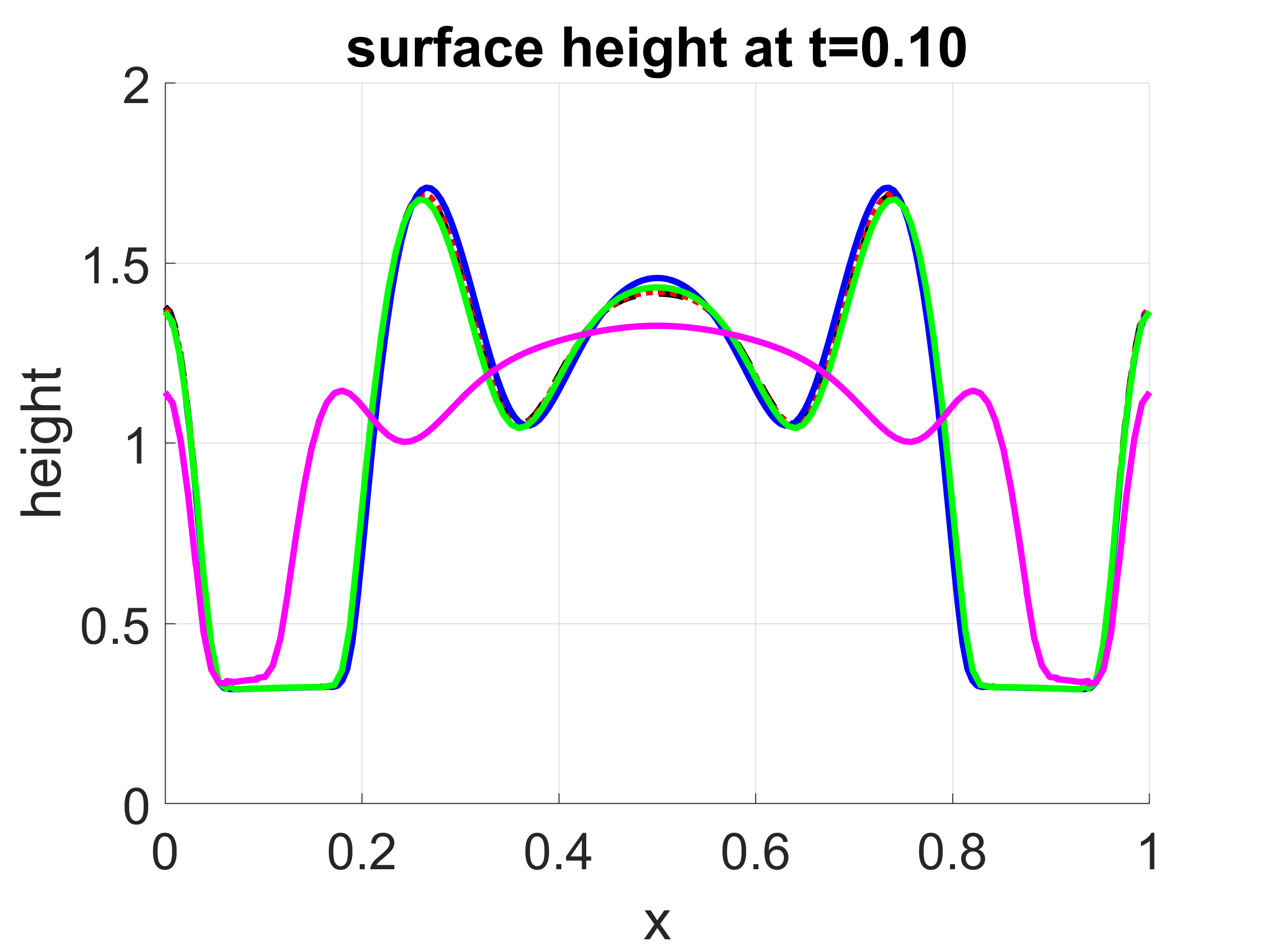}}
\subfigure[t=0.15]{
\includegraphics[width=0.48\textwidth]{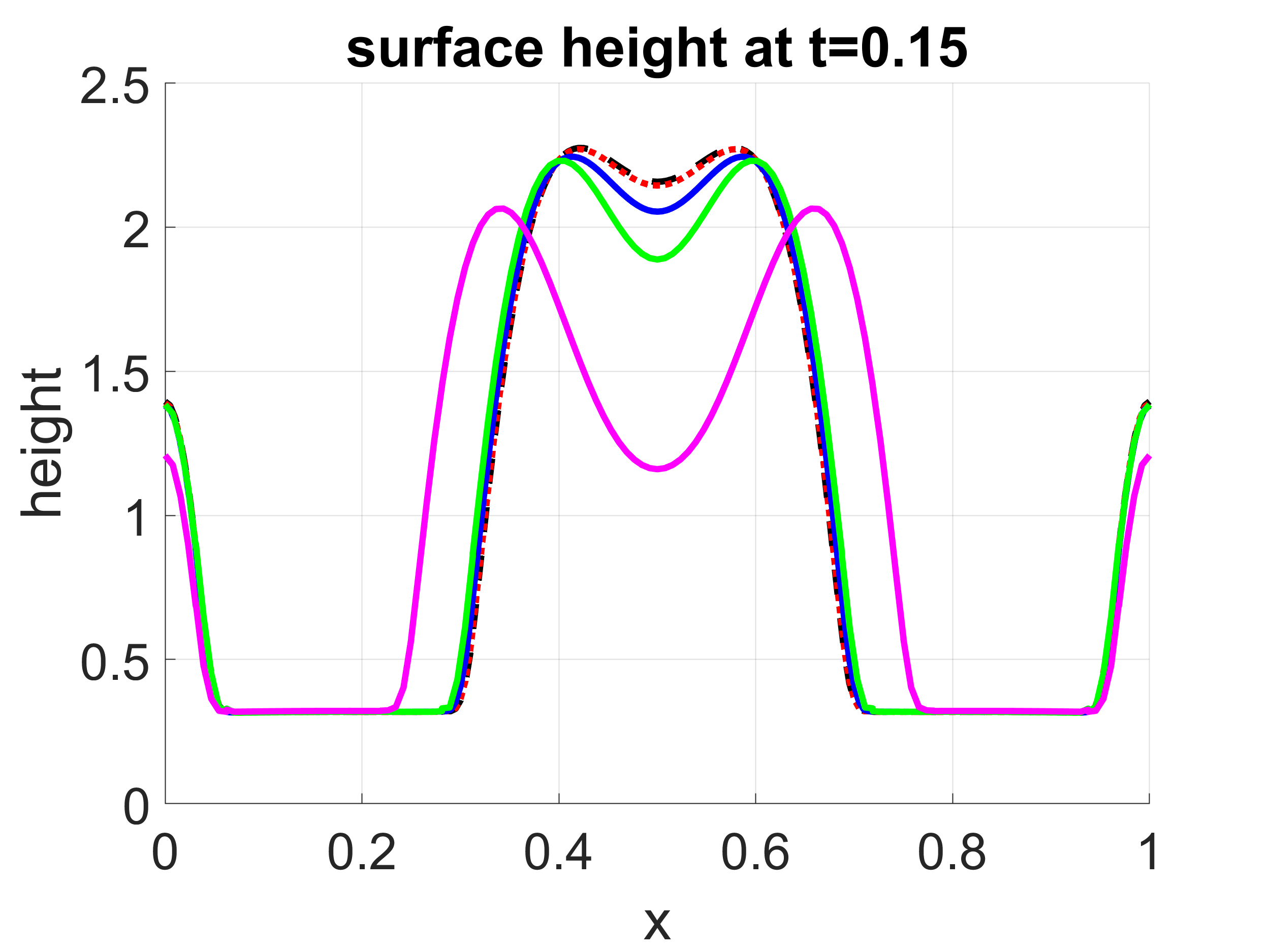}}
\subfigure[t=0.20]{
\includegraphics[width=0.48\textwidth]{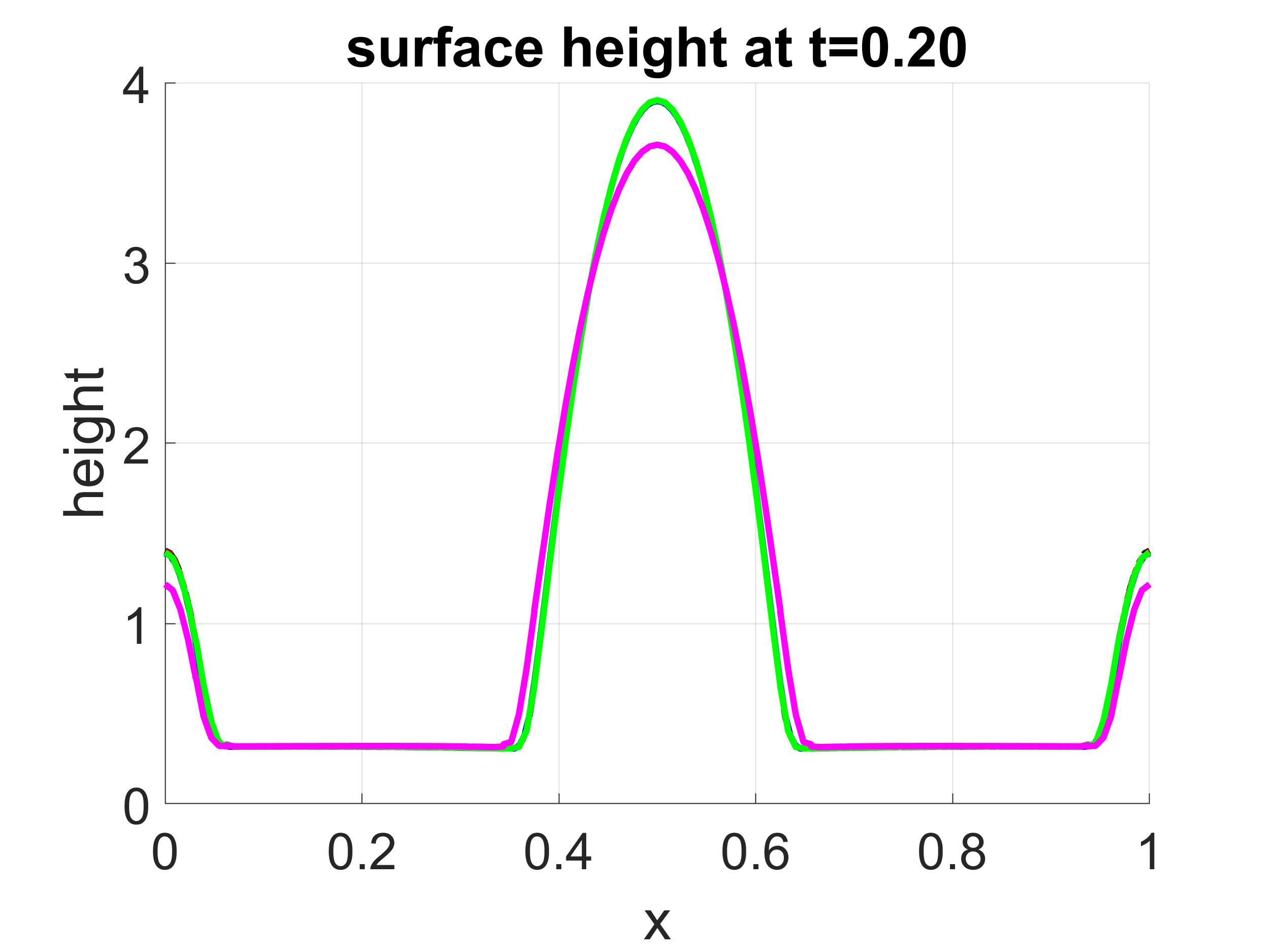}}
\subfigure[t=0.30]{
\includegraphics[width=0.48\textwidth]{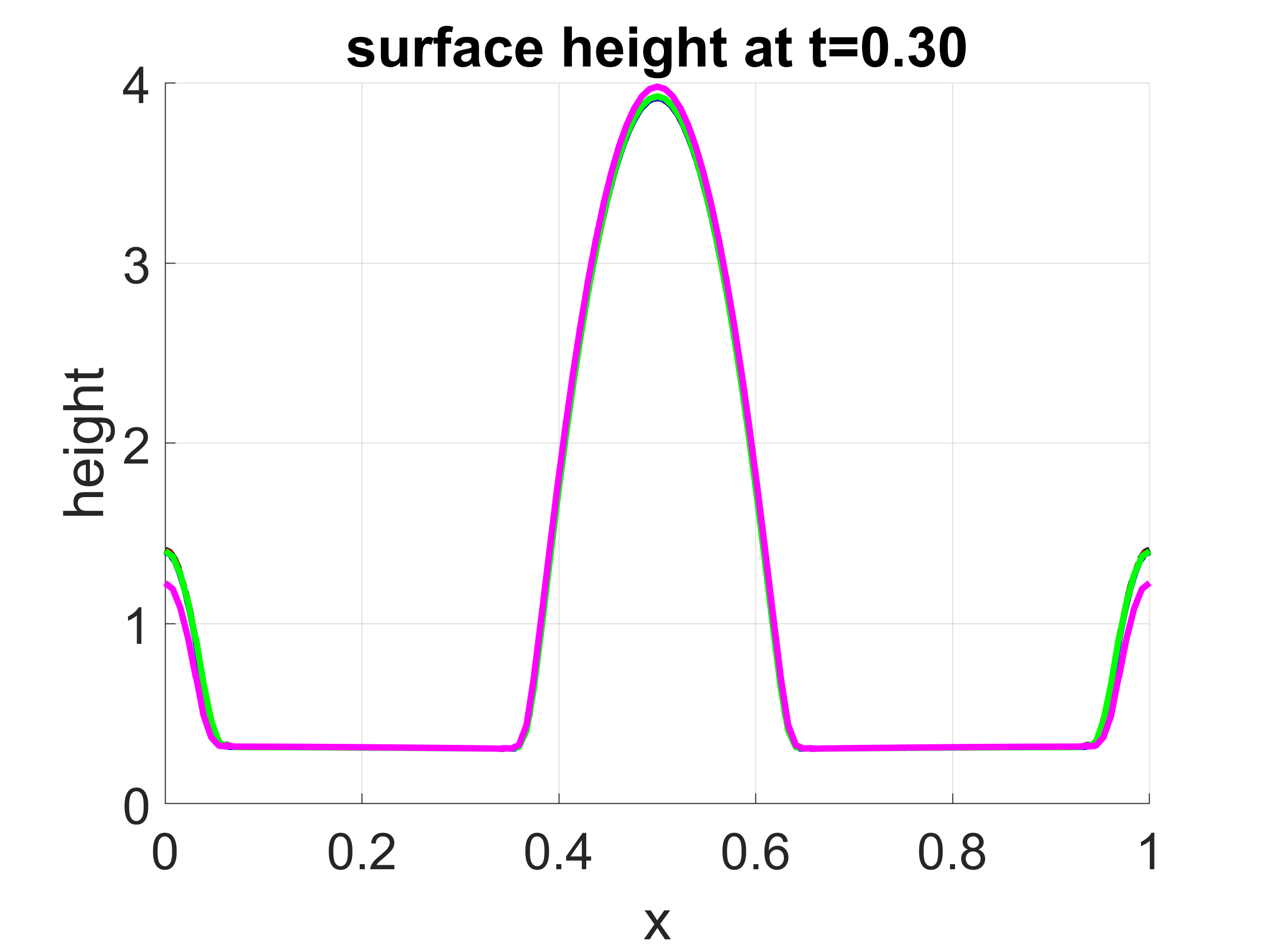}}
\subfigure[t=0.40]{
\includegraphics[width=0.48\textwidth]{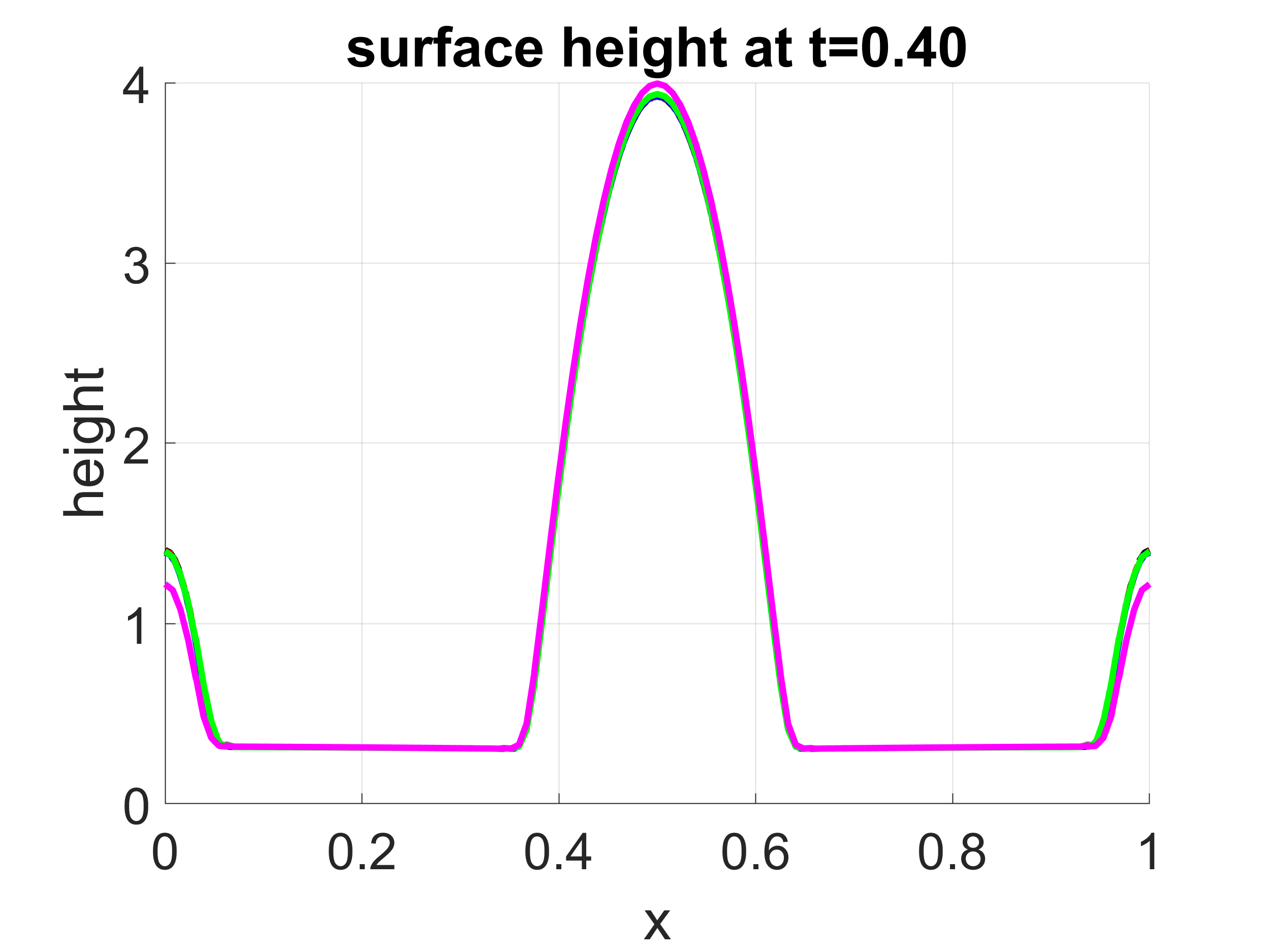}}
\caption{
Snapshots of the surface height for 1D thin-film equation \eqref{eq:model} at different times.
Dashed black line: FEM scheme \eqref{FEM} with time step size $\Delta t = 10^{-4}$;
Dotted red line: Approximated JKO scheme \eqref{saddleH-JKOa} with $\Delta t = 10^{-4}$ using
the Newton-Raphson solver; 
Blue line: Approximated JKO scheme \eqref{saddleH-JKOa} with $\Delta t = 10^{-3}$ using
the Newton-Raphson solver; 
Green line: Approximated JKO scheme \eqref{saddleH-JKOa} with $\Delta t = 10^{-3}$ using
the PDHG solver with tolerance $tol=10^{-8}$; 
Magenta line: Approximated JKO scheme \eqref{saddleH-JKOa} with $\Delta t = 10^{-3}$ using
the PDHG solver with tolerance $tol=10^{-7}$.
% showing stages in the dewetting of a one-dimensional thin film with weak condensation: 
% (b) formation and spreading of dry spots, (c) break-up of ridges, (d) formation of droplets that slowly condense in time.
}
\label{fig:jko1D}
\end{figure}

% \blue{
In this subsection, we numerically solve the PDE \eqref{eq:model}
using the approximated JKO scheme  \eqref{saddleH-JKOa}. 
Both one- and two-dimensional numerical examples are considered.
\red{In Appendix B, we present details of the spatial/temporal mesh convergence studies of the scheme \eqref{saddleH-JKOa}.}

\subsubsection{One-dimensional example}
We take the computational domain to be a periodic line segment $\Omega=[0,1]$.
The initial condition is chosen to be $h_0(x)=1-0.2\cos(2\pi x)$, and terminal time is $T=0.4$.
The approximated JKO scheme \eqref{saddleH-JKOa} is applied on a uniform spatial mesh with $32$ elements
with polynomial degree $k=3$. 
In each JKO step, we use either the Newton-Raphson method to solve the critical point system \eqref{newton}, or the PDHG Algorithm \ref{alg:1} 
to solve the the saddle point problem \eqref{saddleH-JKOa}. 
The PDHG iteration is terminated when the $L_1$-norm of the difference of two consecutive surface heights $h_{h}$ is less than a prescribed tolerance $tol$. The simulation results are compared with the following classical finite element discretization for \eqref{eq:model}
with BDF2 time stepping using the same spatial discretization parameters: find $(h_h^n, P_h^n)\in [V_h^{k+1}]^2$ such that 
\begin{subequations}
    \label{FEM}
    \begin{align}
    \ints{
    \frac{3h_h^{n}-4h_h^{n-1}+h_h^{n-2}}{2\Delta t} q_h
    +
    V_1(h_h^n)\nabla P_h^n\cdot\nabla q_h
    +V_2(h_h^n)P_h^n q_h
    }=&\;0, \; \forall    q_h\in V_h^{k+1},\\
 \ints{
P_h^n r_h -U'(h_h^n)r_h-\alpha^2\nabla h_h^n\cdot\nabla r_h
    } = &\;0, \; \forall    r_h\in V_h^{k+1}.
\end{align}
\end{subequations}
In the FEM scheme \eqref{FEM}, we use a small time step size $\Delta t = 0.0001$, treating these results as reference solutions. For the Newton-Raphson method, we employ both a small time step size $\Delta t = 0.0001$ and a larger time step size $\Delta t = 0.001$. In the optimization approach, the time step size is set to $\Delta t = 0.001$, with varying stopping tolerances for each JKO iteration at $tol = 10^{-7}$ or $tol = 10^{-8}$. Figure~\ref{fig:jko1D} presents snapshots of surface height at different times for all these simulations.
% For the FEM scheme \eqref{FEM}, we take a small time step size $\Delta t = 0.0001$, and the results are treated as the reference solutions. 
% For the Newton-Raphson method, we take both a small time step size $\Delta t = 0.0001$ and a large time step size $\Delta t = 0.001$. 
% For the optimization approach, we take time step size to be $\Delta t = 0.001$, and vary the stopping tolerance for each JKO iteration to be 
% $tol = 10^{-7}$ or $tol=10^{-8}$.
% Snapshots of surface height at different times for all these simulations are shown in Figure~\ref{fig:jko1D}. 
First, we observe that the results for the PDHG algorithm with $tol=10^{-7}$ (in magenta) differ significantly from the other simulation results, particularly at intermediate times $t=0.10$ and $t=0.15$ when the droplet is forming and spreading through dry spots. This suggests that $tol=10^{-7}$ is insufficient for accuracy in the PDHG algorithm. Reducing the tolerance to $tol=10^{-8}$ yields results consistent with the Newton-Raphson method, though the dip at the center point $x=0.5$ at time $t=0.10$ is still not well captured compared to the reference solution. Additionally, the Newton-Raphson method results with $\Delta t = 0.0001$ overlap with the FEM simulation results, validating the proposed approximate JKO scheme \eqref{saddleH-JKOa}.
% }

% \blue{
% Furthermore, we record the number of PDHG iterations required to reach tolerances of $10^{-7}$ and $10^{-8}$ in Figure~\ref{fig:pdhg}. The iteration counts increase significantly during the droplet forming and dry spots spreading phase until about 200 JKO steps (time $=0.2$), with peak iteration counts exceeding $10^4$ for both cases. This figure suggests that the proposed optimization algorithm performs poorly for the approximate JKO scheme \eqref{saddleH-JKOa}, especially compared with the Newton-Raphson method, which typically requires up to 3 iterations for convergence. We will investigate the convergence issues of the optimization solver in future work, focusing on the roles of the PDHG parameters $\sigma_\phi$ and $\sigma_u$, and the time step size $\Delta t$ on overall convergence.
% \begin{figure}
% \centering
% \includegraphics[width=0.64\textwidth]{figs/1DPDHG.png}
% \caption{
% PDHG iteration counts across JKO steps for the 1D thin-film equation \eqref{gradientFlow}.}
% \label{fig:pdhg}
% \end{figure}
% }

\blue{
\subsubsection{Two-dimensional example}
\begin{figure}
\centering
% \subfigure[t=0.05]{
% \includegraphics[width=0.15\textwidth]{figs/zaD2T5F.jpg}
% }
% \subfigure[t=0.10]{
% \includegraphics[width=0.15\textwidth]{figs/zaD2T10F.jpg}}
% \subfigure[t=0.15]{
% \includegraphics[width=0.15\textwidth]{figs/zaD2T15F.jpg}}
% \subfigure[t=0.20]{
% \includegraphics[width=0.15\textwidth]{figs/zaD2T20F.jpg}}
% \subfigure[t=0.30]{
% \includegraphics[width=0.15\textwidth]{figs/zaD2T30F.jpg}}
% \subfigure[t=0.40]{
% \includegraphics[width=0.15\textwidth]{figs/zaD2T40F.jpg}}
 % \hspace{-0.43cm}
\subfigure[t=0.05]{
              \begin{tikzpicture}
        \node [anchor=south west,inner sep=0] (image) at (0,0) {\includegraphics[width=0.14\textwidth]{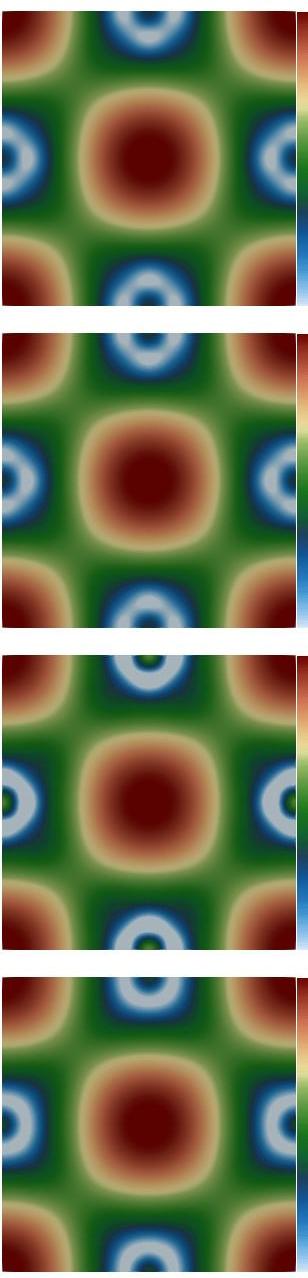}};  % Adjust the width as needed
                \begin{scope}[shift={(image.south east)}]
            % Adjust positioning of numbers vertically and horizontally
            \node at (0.12, 7.78) {\scalebox{0.45}{\textbf{1.3}}}; % First number (on top)
            \node at (0.12, 6.88) {\scalebox{0.45}{\textbf{1.0}}}; % Second number (in the middle)
            \node at (0.12, 5.98) {\scalebox{0.45}{\textbf{0.6}}}; % Third number (on bottom)
            \node at (0.12, 5.79) {\scalebox{0.45}{\textbf{1.3}}}; % First number (on top)
            \node at (0.12, 4.89) {\scalebox{0.45}{\textbf{1.0}}}; % Second number (in the middle)
            \node at (0.12, 3.99) {\scalebox{0.45}{\textbf{0.6}}}; % Third number (on bottom)
            \node at (0.12, 3.82) {\scalebox{0.45}{\textbf{1.3}}}; % First number (on top)
            \node at (0.12, 2.92) {\scalebox{0.45}{\textbf{1.0}}}; % Second number (in the middle)
            \node at (0.12, 2.02) {\scalebox{0.45}{\textbf{0.6}}}; % Third number (on bottom)
            \node at (0.12, 1.85) {\scalebox{0.45}{\textbf{1.3}}}; % First number (on top)
            \node at (0.12, 0.95) {\scalebox{0.45}{\textbf{1.0}}}; % Second number (in the middle)
            \node at (0.12, 0.05) {\scalebox{0.45}{\textbf{0.6}}}; % Third number (on bottom)
        \end{scope}        
    \end{tikzpicture}} 
    \hspace{-0.4cm}
\subfigure[t=0.10]{
              \begin{tikzpicture}
        \node [anchor=south west,inner sep=0] (image) at (0,0) {\includegraphics[width=0.14\textwidth]{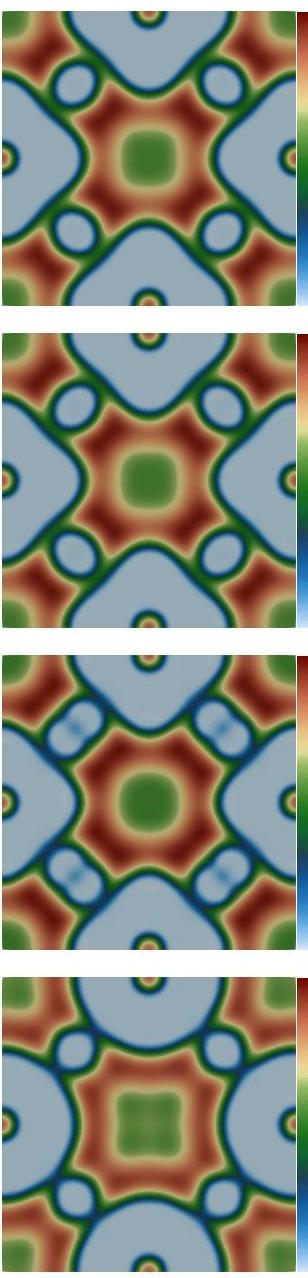}};  % Adjust the width as needed
                \begin{scope}[shift={(image.south east)}]
            % Adjust positioning of numbers vertically and horizontally
            \node at (0.12, 7.78) {\scalebox{0.45}{\textbf{2.0}}}; % First number (on top)
            \node at (0.12, 6.88) {\scalebox{0.45}{\textbf{1.2}}}; % Second number (in the middle)
            \node at (0.12, 5.98) {\scalebox{0.45}{\textbf{0.3}}}; % Third number (on bottom)
            \node at (0.12, 5.79) {\scalebox{0.45}{\textbf{2.0}}}; % First number (on top)
            \node at (0.12, 4.89) {\scalebox{0.45}{\textbf{1.2}}}; % Second number (in the middle)
            \node at (0.12, 3.99) {\scalebox{0.45}{\textbf{0.3}}}; % Third number (on bottom)
            \node at (0.12, 3.82) {\scalebox{0.45}{\textbf{2.0}}}; % First number (on top)
            \node at (0.12, 2.92) {\scalebox{0.45}{\textbf{1.2}}}; % Second number (in the middle)
            \node at (0.12, 2.02) {\scalebox{0.45}{\textbf{0.3}}}; % Third number (on bottom)
            \node at (0.12, 1.85) {\scalebox{0.45}{\textbf{2.0}}}; % First number (on top)
            \node at (0.12, 0.95) {\scalebox{0.45}{\textbf{1.2}}}; % Second number (in the middle)
            \node at (0.12, 0.05) {\scalebox{0.45}{\textbf{0.3}}}; % Third number (on bottom)
        \end{scope}        
    \end{tikzpicture}}
 \hspace{-0.4cm}
\subfigure[t=0.15]{
              \begin{tikzpicture}
        \node [anchor=south west,inner sep=0] (image) at (0,0) {\includegraphics[width=0.14\textwidth]{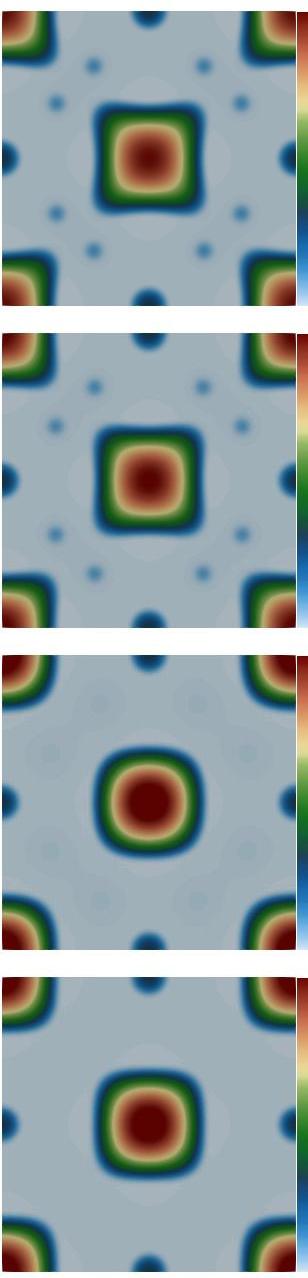}};  % Adjust the width as needed
                \begin{scope}[shift={(image.south east)}]
            % Adjust positioning of numbers vertically and horizontally
            \node at (0.12, 7.78) {\scalebox{0.45}{\textbf{4.8}}}; % First number (on top)
            \node at (0.12, 6.88) {\scalebox{0.45}{\textbf{2.6}}}; % Second number (in the middle)
            \node at (0.12, 5.98) {\scalebox{0.45}{\textbf{0.3}}}; % Third number (on bottom)
            \node at (0.12, 5.79) {\scalebox{0.45}{\textbf{4.8}}}; % First number (on top)
            \node at (0.12, 4.89) {\scalebox{0.45}{\textbf{2.6}}}; % Second number (in the middle)
            \node at (0.12, 3.99) {\scalebox{0.45}{\textbf{0.3}}}; % Third number (on bottom)
            \node at (0.12, 3.82) {\scalebox{0.45}{\textbf{4.8}}}; % First number (on top)
            \node at (0.12, 2.92) {\scalebox{0.45}{\textbf{2.6}}}; % Second number (in the middle)
            \node at (0.12, 2.02) {\scalebox{0.45}{\textbf{0.3}}}; % Third number (on bottom)
            \node at (0.12, 1.85) {\scalebox{0.45}{\textbf{4.8}}}; % First number (on top)
            \node at (0.12, 0.95) {\scalebox{0.45}{\textbf{2.6}}}; % Second number (in the middle)
            \node at (0.12, 0.05) {\scalebox{0.45}{\textbf{0.3}}}; % Third number (on bottom)
        \end{scope}        
    \end{tikzpicture}}
 \hspace{-0.4cm}
\subfigure[t=0.20]{
              \begin{tikzpicture}
        \node [anchor=south west,inner sep=0] (image) at (0,0) {\includegraphics[width=0.14\textwidth]{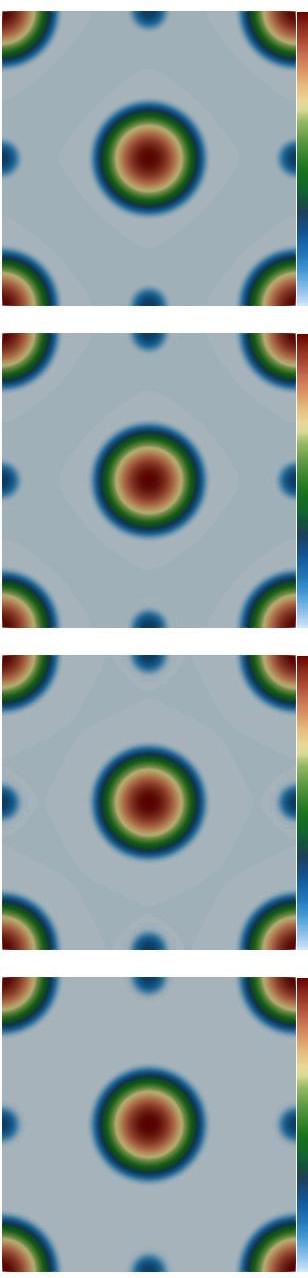}};  % Adjust the width as needed
                \begin{scope}[shift={(image.south east)}]
            % Adjust positioning of numbers vertically and horizontally
            \node at (0.12, 7.78) {\scalebox{0.45}{\textbf{5.7}}}; % First number (on top)
            \node at (0.12, 6.88) {\scalebox{0.45}{\textbf{3.0}}}; % Second number (in the middle)
            \node at (0.12, 5.98) {\scalebox{0.45}{\textbf{0.3}}}; % Third number (on bottom)
            \node at (0.12, 5.79) {\scalebox{0.45}{\textbf{5.7}}}; % First number (on top)
            \node at (0.12, 4.89) {\scalebox{0.45}{\textbf{3.0}}}; % Second number (in the middle)
            \node at (0.12, 3.99) {\scalebox{0.45}{\textbf{0.3}}}; % Third number (on bottom)
            \node at (0.12, 3.82) {\scalebox{0.45}{\textbf{5.7}}}; % First number (on top)
            \node at (0.12, 2.92) {\scalebox{0.45}{\textbf{3.0}}}; % Second number (in the middle)
            \node at (0.12, 2.02) {\scalebox{0.45}{\textbf{0.3}}}; % Third number (on bottom)
            \node at (0.12, 1.85) {\scalebox{0.45}{\textbf{5.7}}}; % First number (on top)
            \node at (0.12, 0.95) {\scalebox{0.45}{\textbf{3.0}}}; % Second number (in the middle)
            \node at (0.12, 0.05) {\scalebox{0.45}{\textbf{0.3}}}; % Third number (on bottom)
        \end{scope}        
    \end{tikzpicture}}
 \hspace{-0.4cm}
\subfigure[t=0.30]{
              \begin{tikzpicture}
        \node [anchor=south west,inner sep=0] (image) at (0,0) {\includegraphics[width=0.14\textwidth]{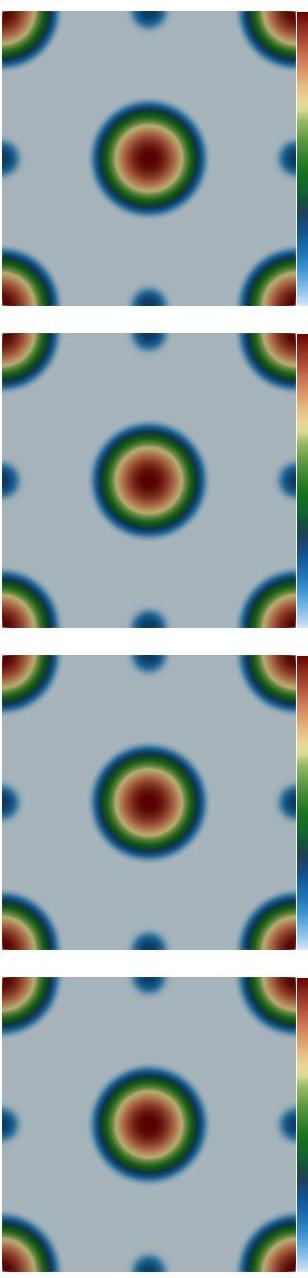}};  % Adjust the width as needed
                \begin{scope}[shift={(image.south east)}]
            % Adjust positioning of numbers vertically and horizontally
            \node at (0.12, 7.78) {\scalebox{0.45}{\textbf{5.7}}}; % First number (on top)
            \node at (0.12, 6.88) {\scalebox{0.45}{\textbf{3.0}}}; % Second number (in the middle)
            \node at (0.12, 5.98) {\scalebox{0.45}{\textbf{0.3}}}; % Third number (on bottom)
            \node at (0.12, 5.79) {\scalebox{0.45}{\textbf{5.7}}}; % First number (on top)
            \node at (0.12, 4.89) {\scalebox{0.45}{\textbf{3.0}}}; % Second number (in the middle)
            \node at (0.12, 3.99) {\scalebox{0.45}{\textbf{0.3}}}; % Third number (on bottom)
            \node at (0.12, 3.82) {\scalebox{0.45}{\textbf{5.7}}}; % First number (on top)
            \node at (0.12, 2.92) {\scalebox{0.45}{\textbf{3.0}}}; % Second number (in the middle)
            \node at (0.12, 2.02) {\scalebox{0.45}{\textbf{0.3}}}; % Third number (on bottom)
            \node at (0.12, 1.85) {\scalebox{0.45}{\textbf{5.7}}}; % First number (on top)
            \node at (0.12, 0.95) {\scalebox{0.45}{\textbf{3.0}}}; % Second number (in the middle)
            \node at (0.12, 0.05) {\scalebox{0.45}{\textbf{0.3}}}; % Third number (on bottom)
        \end{scope}        
    \end{tikzpicture}}
        \hspace{-0.4cm}
\subfigure[t=0.4]{
              \begin{tikzpicture}
        \node [anchor=south west,inner sep=0] (image) at (0,0) {\includegraphics[width=0.14\textwidth]{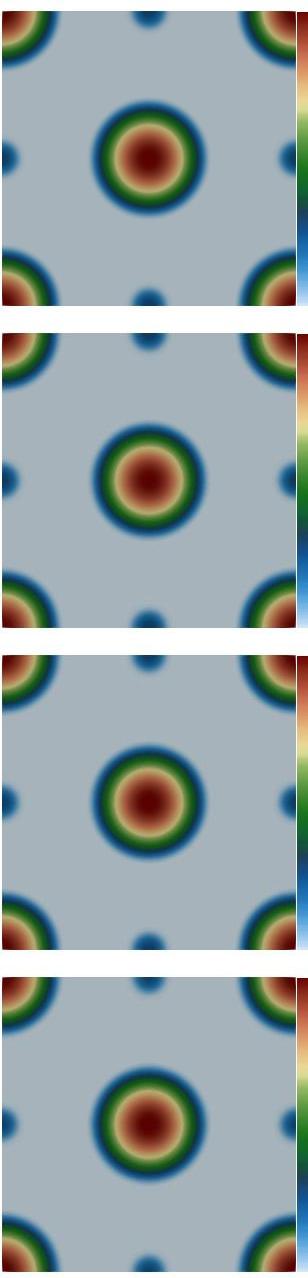}};  % Adjust the width as needed
                \begin{scope}[shift={(image.south east)}]
            % Adjust positioning of numbers vertically and horizontally
            \node at (0.12, 7.78) {\scalebox{0.45}{\textbf{5.7}}}; % First number (on top)
            \node at (0.12, 6.88) {\scalebox{0.45}{\textbf{3.0}}}; % Second number (in the middle)
            \node at (0.12, 5.98) {\scalebox{0.45}{\textbf{0.3}}}; % Third number (on bottom)
            \node at (0.12, 5.79) {\scalebox{0.45}{\textbf{5.7}}}; % First number (on top)
            \node at (0.12, 4.89) {\scalebox{0.45}{\textbf{3.0}}}; % Second number (in the middle)
            \node at (0.12, 3.99) {\scalebox{0.45}{\textbf{0.3}}}; % Third number (on bottom)
            \node at (0.12, 3.82) {\scalebox{0.45}{\textbf{5.7}}}; % First number (on top)
            \node at (0.12, 2.92) {\scalebox{0.45}{\textbf{3.0}}}; % Second number (in the middle)
            \node at (0.12, 2.02) {\scalebox{0.45}{\textbf{0.3}}}; % Third number (on bottom)
            \node at (0.12, 1.85) {\scalebox{0.45}{\textbf{5.7}}}; % First number (on top)
            \node at (0.12, 0.95) {\scalebox{0.45}{\textbf{3.0}}}; % Second number (in the middle)
            \node at (0.12, 0.05) {\scalebox{0.45}{\textbf{0.3}}}; % Third number (on bottom)
        \end{scope}        
    \end{tikzpicture}}
\caption{
Snapshots of the surface height contour for 2D thin-film equation \eqref{eq:model} at different times.
First row: numerical solutions for FEM \eqref{FEM}, $\Delta t=0.0001$; Second row: numerical solutions for the approximated JKO scheme \eqref{saddleH-JKOa}
with the Newton-Raphson solver for \eqref{newton} in each JKO step, $\Delta t=0.0001$; 
Third row: numerical solutions for the approximated JKO scheme \eqref{saddleH-JKOa}
with the Newton-Raphson solver for \eqref{newton} in each JKO step, $\Delta t=0.001$; Last row: numerical solutions for the approximated JKO scheme \eqref{saddleH-JKOa}
with PDHG solver in Algorithm \ref{alg:1} for each JKO step, $\Delta t=0.001$, $tol = 10^{-8}$.
}
\label{fig:jko2D}
\end{figure}
We consider the computational domain as a periodic unit square $\Omega=[0,1]^2$. The initial condition is chosen as $h_0(x,y)=1+0.2\cos(2\pi x)\cos(2\pi y)$, and the terminal time is $T=0.4$.
The approximated JKO scheme \eqref{saddleH-JKOa} is applied on a uniform rectangular mesh with $32\times 32$ elements and polynomial degree $k=3$. 
Similar to the 1D case, we use either the Newton-Raphson method to solve the critical point system \eqref{newton}, or the PDHG Algorithm \ref{alg:1} to solve the the saddle point problem \eqref{saddleH-JKOa} for each JKO step. 
% The PDHG iteration is terminated when the $L_1$-norm of the difference of two consecutive surface heights $h_{h}$ is less than a  tolerance $tol=10^{-8}$.
}

\blue{
We compare the simulation results with the reference solution using the FEM \eqref{FEM} with the same spatial discretization parameters.
For the FEM scheme \eqref{FEM}, we use a small time step size $\Delta t = 0.0001$, and these results are treated as reference solutions. For the Newton-Raphson method, we use both a small time step size $\Delta t = 0.0001$ and a large time step size $\Delta t = 0.001$. For the optimization approach, we set the time step size to be $\Delta t = 0.001$, and terminate the iteration when the $L_1$-norm of the difference of two consecutive surface heights $h_{h}$ is less than a  tolerance $tol=10^{-8}$.}

\blue{Snapshots of surface height contours at different times are shown in Figure~\ref{fig:jko2D}.
We observe a similar pattern for the surface height evolution as in the 1D case.
Driven by the interfacial instabilities described in \eqref{eq:model}, the early stage (see panel (a) of Figure \ref{fig:jko2D}) shows the spatial variations in the solution profile growing until the minimum height approaches $h_{\min} = O(\epsilon)$. 
In the later stage, the minimum height spreads to form dry spots, leading to droplet formation (see panel (b) of Figure \ref{fig:jko2D}), followed by a slow growth in droplet height driven by weak condensation effects (see panels (c-f) of Figure \ref{fig:jko2D}). This example captures the morphological changes previously observed in 1D dewetting thin-film dynamics with weak non-mass-conserving effects \citep{ji2018instability}.
Moreover, results for all four simulations are qualitatively similar to each other. In particular, the results for the first two rows with the small time step size $\Delta t=0.0001$ are almost identical to each other, validating the proposed approximated JKO scheme \eqref{saddleH-JKOa} in the 2D setting.
The results for the last two rows are consistent with each other, which indicates both the Newton-Raphson and the PDHG optimization approaches are able to solve the discrete saddle-point problem \eqref{saddleH-JKOa} accurately.}

\subsection{MFC for droplet dynamics.}
\label{ex2}
\blue{Next, we demonstrate that the developed MFC system can drastically control droplet shapes, motions, and drive morphological changes in droplet configurations.  With different choices of initial and target surface heights, we will present numerical results showcasing the application of the developed MFC scheme for fundamental droplet actuation techniques, including droplet transport, bead-up (i.e., dewetting), spreading, merging, and splitting.
To achieve this, we apply the finite element discretization \eqref{saddleH} and the PDHG Algorithm \ref{alg:1} to solve the full MFC system \blue{in Definition \ref{mfc}}. 
We use the same mobility functions and energy functional \eqref{ppp} as in the previous subsection.
There is additional freedom in choosing the potential and terminal functionals in \eqref{mfcA1} and the scaling constant $\beta$ in \eqref{mfcA2}
to determine the MFC problem \eqref{mfcA}.
Throughout this subsection, we take the terminal time $T=1$ and $\beta=0.01$, and use the following potential and terminal functionals:
\begin{equation}
\mathcal{F}(h) = -0.02\int_0^T\int_{\Omega}h\log(h) dx\,dt,\quad
\mathcal{G}(h) = 0.05\int_{\Omega}h(\log(h/\blue{h_{T}})-1) dx,
\end{equation}
where $\blue{h_T}$ is a given target terminal surface height function. The potential functional $\mathcal{F}(h)$ serves as a regularization term, while the terminal function $\mathcal{G}(h)$ drives the surface height towards the target surface height as 
$h_T=\arg\min_h \mathcal{G}(h)$. The choice of these functional forms for different control constraints is beyond the scope of this study and will serve as an interesting future direction.
}
% and present numerical results for different choices of initial and target surface heights.
% \begin{alignat}{2}
% \label{icbc}
% \begin{cases}
% \text{Case 1}:& h_0(x, y) = \epsilon+
% \frac{10}{3}\left(1-75((x-0.3)^2+(y-0.3)^2)\right)_+, \\
% &h_{\mathrm{trg}}(x,y)=\epsilon+
% \frac{10}{3}\left(1-75((x-0.7)^2+(y-0.7)^2)\right)_+,\\[.4ex]
% \text{Case 2}:& h_0(x, y) = \epsilon+
% \frac{10}{3}\left(1-75((x-0.5)^2+(y-0.5)^2)\right)_+, \\
% &h_{\mathrm{trg}}(x,y)=\epsilon+
% \frac{5}{12}\left(1-\frac{75}{8}((x-0.5)^2+(y-0.5)^2)\right)_+,\\[.4ex]
% \text{Case 3}:& h_0(x, y) = 
% \epsilon+
% \frac{5}{12}\left(1-\frac{75}{8}((x-0.5)^2+(y-0.5)^2)\right)_+,\\
% &h_{\mathrm{trg}}(x,y)= \epsilon+
% \frac{10}{3}\left(1-75((x-0.5)^2+(y-0.5)^2)\right)_+\\[.4ex]
% \text{Case 4}:& h_0(x, y) = 
% \epsilon+
% \frac{5}{12}\left(1-\frac{75}{8}((x-0.5)^2+(y-0.5)^2)\right)_+,\\
% &h_{\mathrm{trg}}(x,y)= \epsilon+
% \frac{10}{3}\left(1-75((x-0.5)^2+(y-0.5)^2)\right)_+\\[.4ex]
% \text{Case 5}:& h_0(x, y) = 
% \epsilon+
% \frac{5}{12}\left(1-\frac{75}{8}((x-0.5)^2+(y-0.5)^2)\right)_+,\\
% &h_{\mathrm{trg}}(x,y)= \epsilon+
% \frac{10}{3}\left(1-75((x-0.5)^2+(y-0.5)^2)\right)_+
% \end{cases}
% \end{alignat}

\blue{In the finite element discretization \eqref{saddleH}, 
we use a uniform spatial rectangular mesh of size $64\times 64$, a uniform temporal mesh of size $16$, and take 
polynomial degree \blue{$k=3$}. The total number of degrees of freedom (DOFs)  for the dual variable $\phi_h$
is about 4 million. This choice of discretization parameters ensures enough resolution in the simulation with a manageable computational cost.
A more in-depth mesh resolution study will be carried out elsewhere.
We stop the PDHG iteration when the $L_1$-norm of the difference between the terminal surface heights
in two consecutive iterations 
$\|h_{h,T}^n-h_{h,T}^{n-1}\|_{L_1}$
is less than $10^{-6}$, when a converged solution profile has been reached. 
For all the test cases considered here, the algorithm converged within 2000 iterations. 
As mentioned in the previous subsection, a more in-depth computational study of the optimization solver Algorithm \ref{alg:1} will be carried elsewhere.
}
% The same stopping criteria as the previous example is used for the PDHG Algorithm \ref{alg:1}.

\subsubsection*{Case 1: Droplet transport.}
\begin{figure}
\centering
% \subfigure[t=0.0]{
% \includegraphics[width=0.15\textwidth]{figs/zaC32T0h.jpg}}
% \subfigure[t=0.2]{
% \includegraphics[width=0.15\textwidth]{figs/zaC32T2h.jpg}}
% \subfigure[t=0.4]{
% \includegraphics[width=0.15\textwidth]{figs/zaC32T4h.jpg}}
% \subfigure[t=0.6]{
% \includegraphics[width=0.15\textwidth]{figs/zaC32T6h.jpg}}
% \subfigure[t=0.8]{
% \includegraphics[width=0.15\textwidth]{figs/zaC32T8h.jpg}}
% \subfigure[t=1.0]{
% \includegraphics[width=0.15\textwidth]{figs/zaC32T10h.jpg}}
\subfigure[t=0.0]{
              \begin{tikzpicture}
        \node [anchor=south west,inner sep=0] (image) at (0,0) {\includegraphics[width=0.14\textwidth]{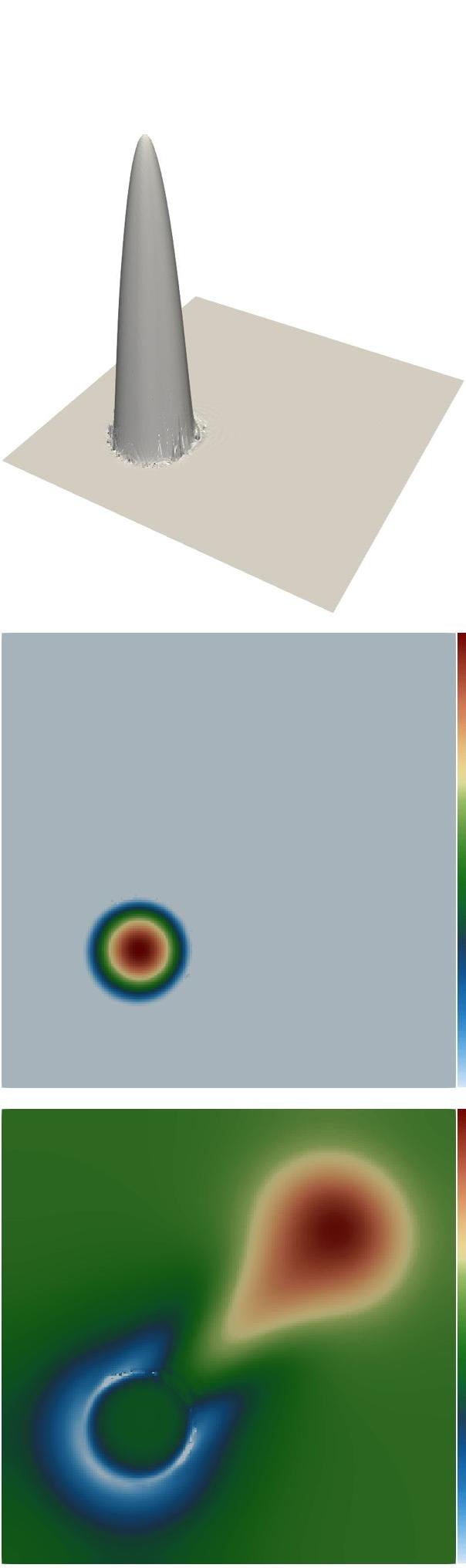}};  % Adjust the width as needed
                \begin{scope}[shift={(image.south east)}]
            % Adjust positioning of numbers vertically and horizontally
            \node at (0.15, 3.8) {\scalebox{0.45}{\textbf{3.72}}}; % First number (on top)
            \node at (0.15, 2.9) {\scalebox{0.45}{\textbf{2.01}}}; % Second number (in the middle)
            \node at (0.15, 2.0) {\scalebox{0.45}{\textbf{0.30}}}; % Third number (on bottom)
            \node at (0.15, 1.85) {\scalebox{0.45}{\textbf{0.06}}}; % First number (on top)
            \node at (0.15, 0.9) {\scalebox{0.45}{\textbf{0.00}}}; % Second number (in the middle)
            \node at (0.14, 0.05) {\scalebox{0.45}{\textbf{-0.05}}}; % Third number (on bottom)
        \end{scope}        
    \end{tikzpicture}}
        \hspace{-0.43cm}
        \subfigure[t=0.2]{
              \begin{tikzpicture}
        \node [anchor=south west,inner sep=0] (image) at (0,0) {\includegraphics[width=0.14\textwidth]{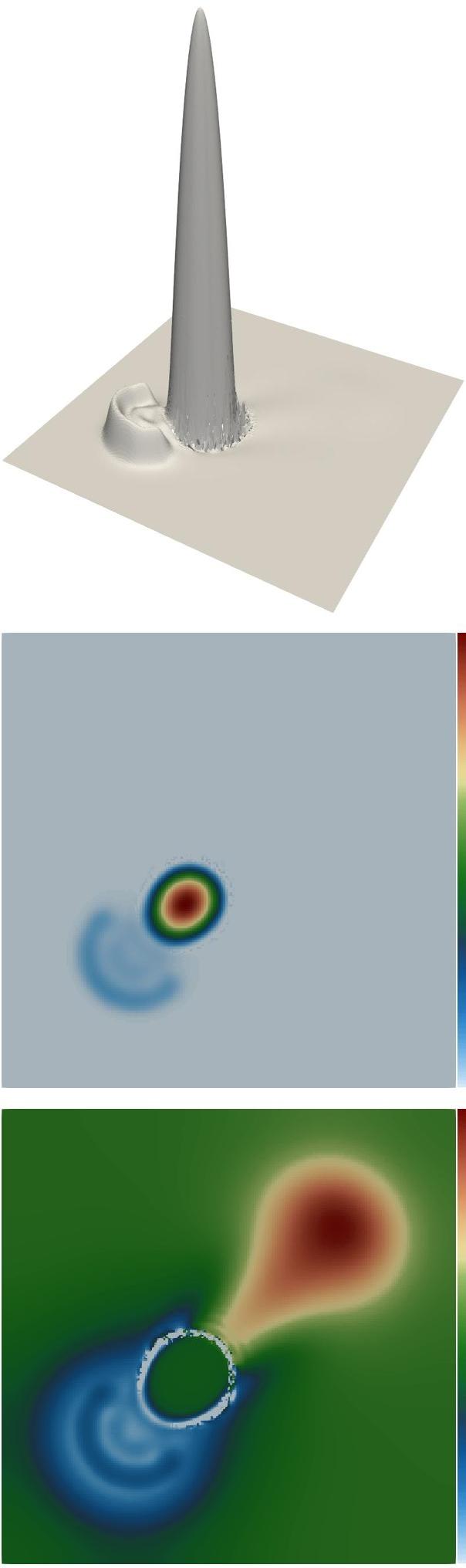}};  % Adjust the width as needed
        % Get the coordinates of the image and add the vertical numbers to the right
                \begin{scope}[shift={(image.south east)}]
            % Adjust positioning of numbers vertically and horizontally
            \node at (0.15, 3.8) {\scalebox{0.45}{\textbf{4.76}}}; % First number (on top)
            \node at (0.15, 2.9) {\scalebox{0.45}{\textbf{2.53}}}; % Second number (in the middle)
            \node at (0.15, 2.0) {\scalebox{0.45}{\textbf{0.30}}}; % Third number (on bottom)
            \node at (0.15, 1.85) {\scalebox{0.45}{\textbf{0.06}}}; % First number (on top)
            \node at (0.15, 0.9) {\scalebox{0.45}{\textbf{0.00}}}; % Second number (in the middle)
            \node at (0.14, 0.05) {\scalebox{0.45}{\textbf{-0.05}}}; % Third number (on bottom)
        \end{scope}        
    \end{tikzpicture}}
        \hspace{-0.43cm}
        \subfigure[t=0.4]{
              \begin{tikzpicture}
        \node [anchor=south west,inner sep=0] (image) at (0,0) {\includegraphics[width=0.14\textwidth]{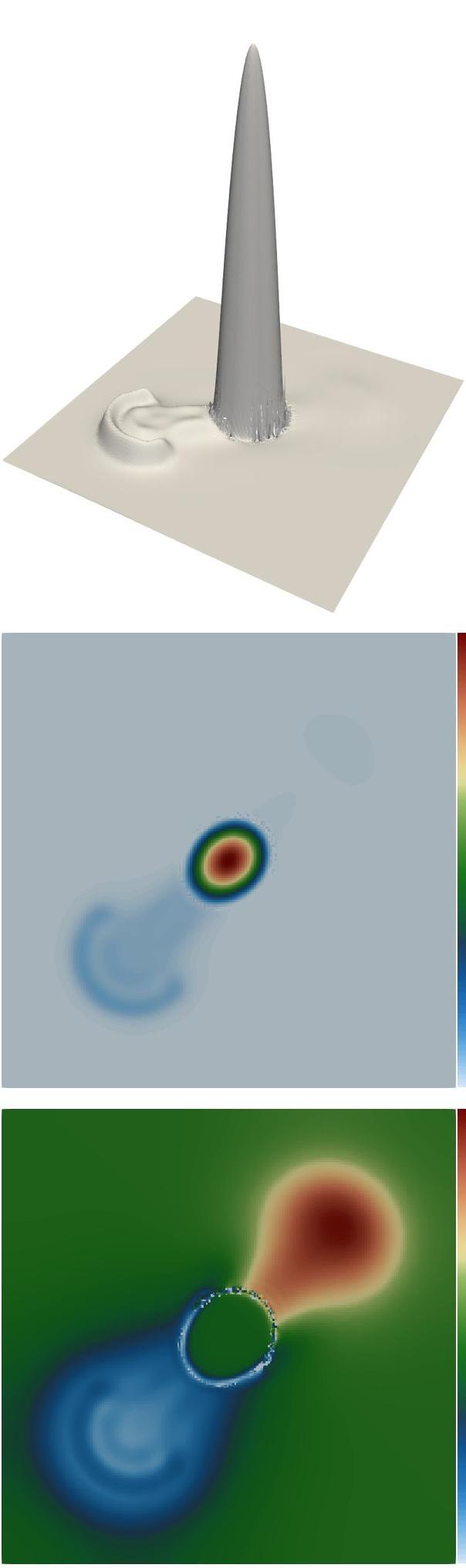}};  % Adjust the width as needed
                \begin{scope}[shift={(image.south east)}]
            % Adjust positioning of numbers vertically and horizontally
            \node at (0.15, 3.8) {\scalebox{0.45}{\textbf{4.47}}}; % First number (on top)
            \node at (0.15, 2.9) {\scalebox{0.45}{\textbf{2.38}}}; % Second number (in the middle)
            \node at (0.15, 2.0) {\scalebox{0.45}{\textbf{0.30}}}; % Third number (on bottom)
            \node at (0.15, 1.85) {\scalebox{0.45}{\textbf{0.06}}}; % First number (on top)
            \node at (0.15, 0.9) {\scalebox{0.45}{\textbf{0.00}}}; % Second number (in the middle)
            \node at (0.14, 0.05) {\scalebox{0.45}{\textbf{-0.05}}}; % Third number (on bottom)
        \end{scope}        
    \end{tikzpicture}}
        \hspace{-0.43cm}
\subfigure[t=0.6]{
              \begin{tikzpicture}
        \node [anchor=south west,inner sep=0] (image) at (0,0) {\includegraphics[width=0.14\textwidth]{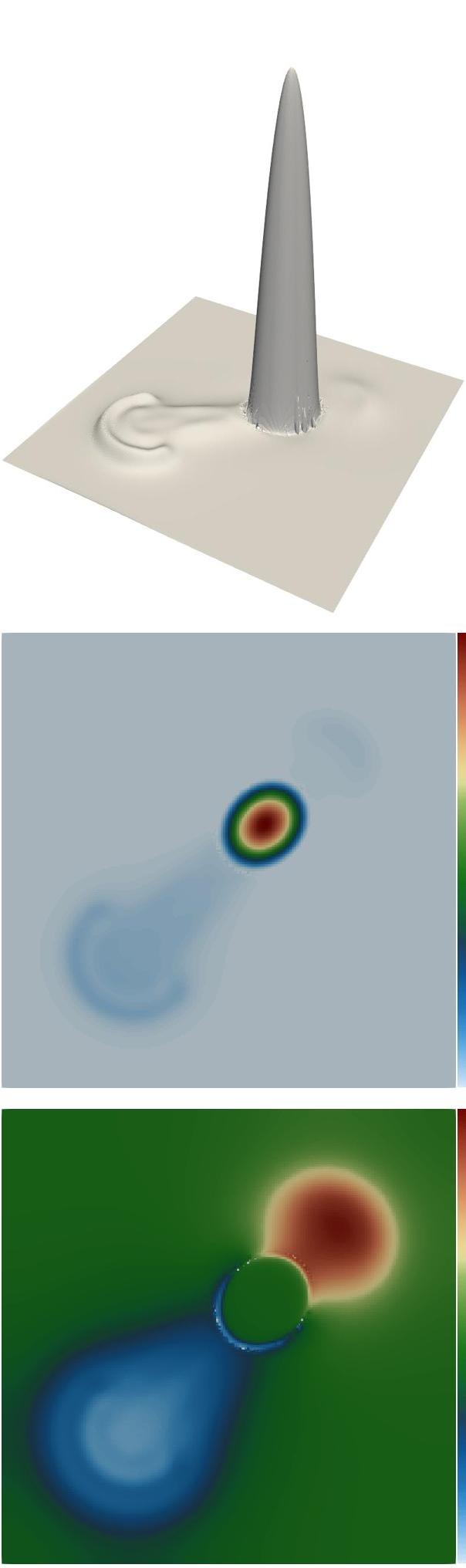}};  % Adjust the width as needed
                        \begin{scope}[shift={(image.south east)}]
            % Adjust positioning of numbers vertically and horizontally
                        % Adjust positioning of numbers vertically and horizontally
            \node at (0.15, 3.8) {\scalebox{0.45}{\textbf{4.19}}}; % First number (on top)
            \node at (0.15, 2.9) {\scalebox{0.45}{\textbf{2.25}}}; % Second number (in the middle)
            \node at (0.15, 2.0) {\scalebox{0.45}{\textbf{0.30}}}; % Third number (on bottom)
            \node at (0.15, 1.85) {\scalebox{0.45}{\textbf{0.06}}}; % First number (on top)
            \node at (0.15, 0.9) {\scalebox{0.45}{\textbf{0.00}}}; % Second number (in the middle)
            \node at (0.14, 0.05) {\scalebox{0.45}{\textbf{-0.05}}}; % Third number (on bottom)
        \end{scope}        
    \end{tikzpicture}}
        \hspace{-0.43cm}
    \subfigure[t=0.8]{
              \begin{tikzpicture}
        \node [anchor=south west,inner sep=0] (image) at (0,0) {\includegraphics[width=0.14\textwidth]{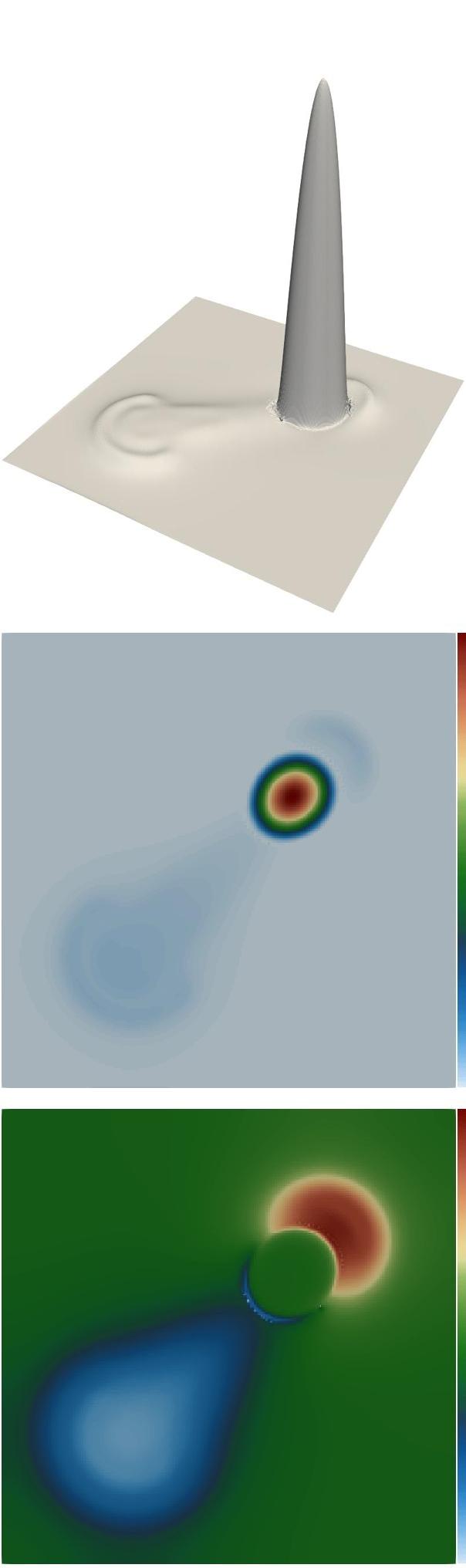}};  % Adjust the width as needed
                \begin{scope}[shift={(image.south east)}]
            % Adjust positioning of numbers vertically and horizontally
            \node at (0.15, 3.8) {\scalebox{0.45}{\textbf{4.06}}}; % First number (on top)
            \node at (0.15, 2.9) {\scalebox{0.45}{\textbf{2.18}}}; % Second number (in the middle)
            \node at (0.15, 2.0) {\scalebox{0.45}{\textbf{0.30}}}; % Third number (on bottom)
            \node at (0.15, 1.85) {\scalebox{0.45}{\textbf{0.06}}}; % First number (on top)
            \node at (0.15, 0.9) {\scalebox{0.45}{\textbf{0.00}}}; % Second number (in the middle)
            \node at (0.14, 0.05) {\scalebox{0.45}{\textbf{-0.05}}}; % Third number (on bottom)
        \end{scope}        
    \end{tikzpicture}}
        \hspace{-0.43cm}
\subfigure[t=1.0]{
              \begin{tikzpicture}
        \node [anchor=south west,inner sep=0] (image) at (0,0) {\includegraphics[width=0.14\textwidth]{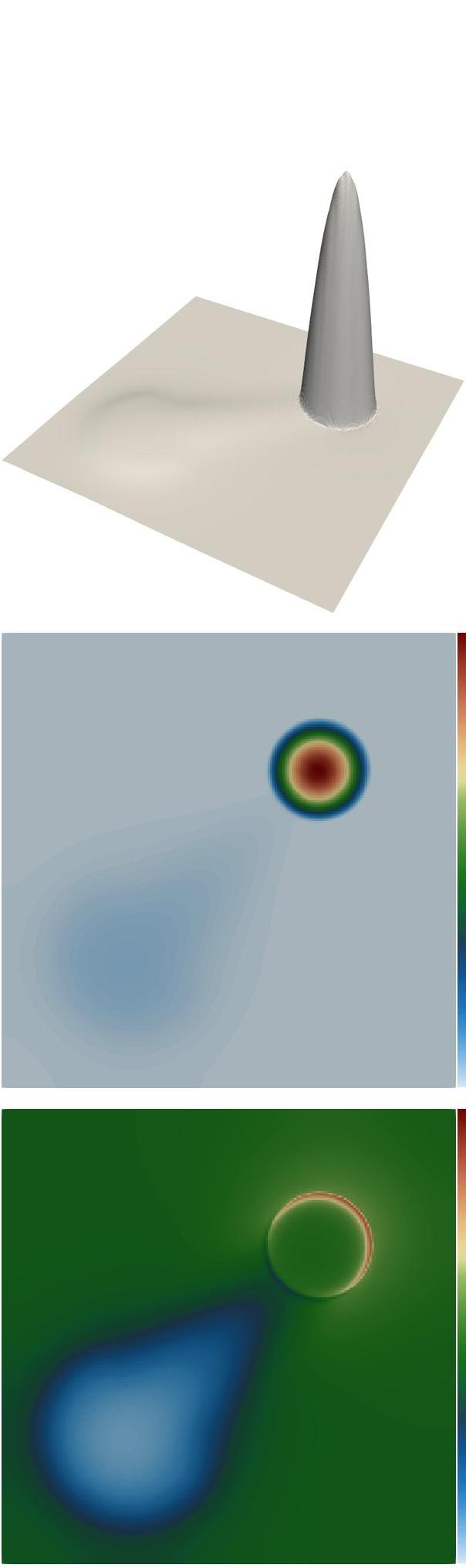}};  % Adjust the width as needed
                \begin{scope}[shift={(image.south east)}]
            % Adjust positioning of numbers vertically and horizontally
            \node at (0.15, 3.8) {\scalebox{0.45}{\textbf{3.05}}}; % First number (on top)
            \node at (0.15, 2.9) {\scalebox{0.45}{\textbf{1.68}}}; % Second number (in the middle)
            \node at (0.15, 2.0) {\scalebox{0.45}{\textbf{0.30}}}; % Third number (on bottom)
            \node at (0.15, 1.85) {\scalebox{0.45}{\textbf{0.06}}}; % First number (on top)
            \node at (0.15, 0.9) {\scalebox{0.45}{\textbf{0.00}}}; % Second number (in the middle)
            \node at (0.14, 0.05) {\scalebox{0.45}{\textbf{-0.05}}}; % Third number (on bottom)
        \end{scope}        
    \end{tikzpicture}}
\caption{
% Example \ref{ex2}: 
\blue{
Case 1 (Droplet transport):
Snapshots of (top row) 3D plots of the controlled surface height $h$, (middle row) contour plots of the controlled surface height, and (bottom row) contour plots of the activity field $\zeta$ at different times. The corresponding animation video can be found in the GitHub repository \citep{github}.
% Top row: snapshots of 3D plots for the controlled surface height $h$ at different times;
% Middle row: snapshots of the controlled surface height contours at different times;
% Bottom row: snapshots of the activity field $\zeta$ at different times.
}
}
\label{fig:T32a}
\end{figure}
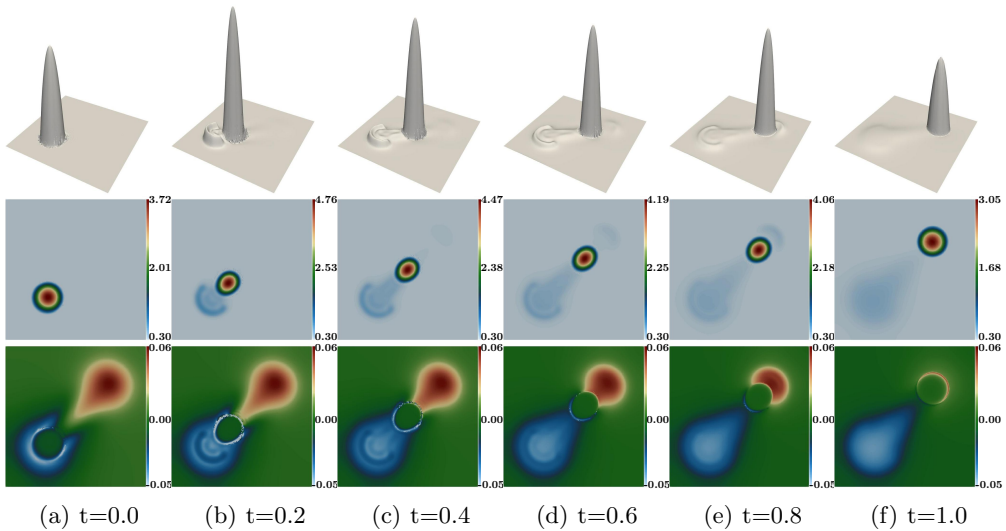
Droplet transport is one of the most important operations in digital microfluidics (DMF) and has been extensively investigated through experimental approaches such as electro-dewetting \citep{li2019ionic}. 
We present results for the MFC problem \blue{\eqref{mfcA}} with the 
initial and target surface heights specified as
\begin{align}
\label{init-D}
h_0(x, y) =&\; \epsilon+
\frac{10}{3}\left(1-75((x-0.3)^2+(y-0.3)^2)\right)_+,\\
\blue{h_{T}(x, y)} =&\; \epsilon+
\frac{10}{3}\left(1-75((x-0.7)^2+(y-0.7)^2)\right)_+.  
\end{align}
Here $(f)_+ = \begin{cases}
    f &\quad \text{if } f\ge0,\\[.2ex]
    0 &\quad \text{otherwise}
\end{cases}$
is the positive part of a function $f$. 
This example models the MFC of an initial parabolic droplet centered at $(0.3, 0.3)$
moving towards a target parabolic droplet centered at $(0.7,0.7)$. 
\blue{Figure~\ref{fig:T32a} shows the evolution of the controlled surface height at different times, alongside the activity field 
$\zeta_h=\frac{\beta(U'(h_h)+p_h)+\phi_h}{h_h}$ in \eqref{zeta}. 
The snapshots in Figure~\ref{fig:T32a} reveal that the droplet quickly breaks the symmetry, beads up, and develops a larger advancing contact angle, leaving a capillary wave as it progresses toward the target position. 
The 3D contour plots in the first row of Figure~\ref{fig:T32a} are reconstructed using piecewise polynomial interpolation of the discrete surface height data in the space-time quadrature space $h_h\in W_h^k$. Specifically,  a third order polynomial interpolation ($k=3$) is employed for the surface reconstruction within each space-time cell, which is then used to evaluate the surface height at specific times.
Due to the non-smoothness of the initial profile \eqref{init-D}, visible oscillations in the reconstructed surface height are observed near the contact line (where strong gradients exist)  at the initial time $t=0$. These oscillations are localized near the contact lines and diminish as time progresses due to the instantaneous smoothing effect of the nonlinear diffusion terms in the MFC problem.
% indicating that the MFC problem exhibits a smoothing effect.
Additionally, the 2D contour plots in the second and third rows of Figure~\ref{fig:T32a} are reconstructed from piecewise constant 
interpolation of the data on a refined $256\times 256$ mesh. Similar to the 3D plots, localized oscillations near the contact lines are observed, which diminish over time. The surface height remains positive throughout the simulation, achieving a smooth profile by the final time $t=1$. Furthermore, the activity field tends to be contractile ($\zeta > 0$) near the target droplet profile and extensile ($\zeta < 0$) near the initial profile, as the active stress competes with passive interfacial forces and pushes the droplet towards the terminal location.
}

\subsubsection*{Case 2: Droplet spreading.}
Controlling the deformation of a droplet is also a crucial aspect of liquid-handling technology. For instance, 
in typical electro-wetting and electro-dewetting experiments, an electric field can induce changes in the contact angles of a slender droplet containing a dilute surfactant \citep{nelson2012droplet}. Here, we demonstrate the MFC mechanism to control the spreading and bead-up of droplets. 

For the droplet spreading example, we take the initial and target surface heights as 
\begin{align}
h_0(x, y) =&\; \epsilon+
\frac{10}{3}\left(1-75((x-0.5)^2+(y-0.5)^2)\right)_+,\\
\blue{h_{T}(x, y)} =&\; \epsilon+
\frac{5}{12}\left(1-\frac{75}{8}((x-0.5)^2+(y-0.5)^2)\right)_+.  
\end{align}
This case models the MFC of an initial parabolic droplet centered at $(0.5, 0.5)$ with half-width $w=\frac{1}{5\sqrt{3}}$
flattening towards the target droplet with half-width $w=\frac{2\sqrt{2}}{5\sqrt{3}}$ . The initial and target droplets have the same total mass.
Snapshots of the simulation results for the scheme \eqref{saddleH} are presented in Figures~\ref{fig:T33a}. The numerical results indicate that the controlled droplet initially evolves into a pancake shape and then gradually converges to the target droplet profile over time. \blue{The radial symmetry in the droplet profile is preserved during the evolution. Compared with the previous droplet transport test case, visible oscillations are only observed at initial time $t=0$, as shown in the top left panel of Figure~\ref{fig:T32a}. In this test case, the activity field remains radially symmetric and contractile ($\zeta > 0$)
% positive
throughout the simulation, reaching its maximum value near the droplet's contact lines at each time.
}

\begin{figure}
\centering
% \subfigure[t=0.0]{
% \includegraphics[width=0.15\textwidth]{figs/zaC33T0h.jpg}}
% \subfigure[t=0.2]{
% \includegraphics[width=0.15\textwidth]{figs/zaC33T2h.jpg}}
% \subfigure[t=0.4]{
% \includegraphics[width=0.15\textwidth]{figs/zaC33T4h.jpg}}
% \subfigure[t=0.6]{
% \includegraphics[width=0.15\textwidth]{figs/zaC33T6h.jpg}}
% \subfigure[t=0.8]{
% \includegraphics[width=0.15\textwidth]{figs/zaC33T8h.jpg}}
% \subfigure[t=1.0]{
% \includegraphics[width=0.15\textwidth]{figs/zaC33T10h.jpg}}
\subfigure[t=0.0]{
              \begin{tikzpicture}
        \node [anchor=south west,inner sep=0] (image) at (0,0) {\includegraphics[width=0.14\textwidth]{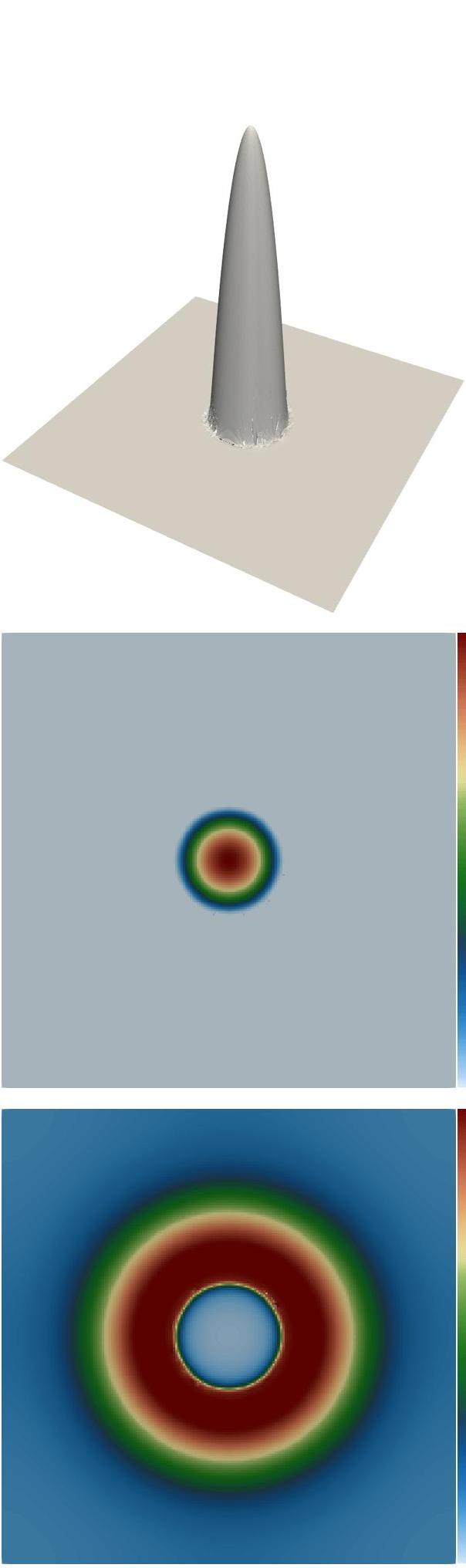}};  % Adjust the width as needed
                \begin{scope}[shift={(image.south east)}]
            % Adjust positioning of numbers vertically and horizontally
            \node at (0.15, 3.8) {\scalebox{0.45}{\textbf{3.52}}}; % First number (on top)
            \node at (0.15, 2.9) {\scalebox{0.45}{\textbf{1.91}}}; % Second number (in the middle)
            \node at (0.15, 2.0) {\scalebox{0.45}{\textbf{0.30}}}; % Third number (on bottom)
            \node at (0.15, 1.85) {\scalebox{0.45}{\textbf{0.06}}}; % First number (on top)
            \node at (0.15, 0.9) {\scalebox{0.45}{\textbf{0.03}}}; % Second number (in the middle)
            \node at (0.15, 0.05) {\scalebox{0.45}{\textbf{0.00}}}; % Third number (on bottom)
        \end{scope}        
    \end{tikzpicture}}
        \hspace{-0.43cm}
        \subfigure[t=0.2]{
              \begin{tikzpicture}
        \node [anchor=south west,inner sep=0] (image) at (0,0) {\includegraphics[width=0.14\textwidth]{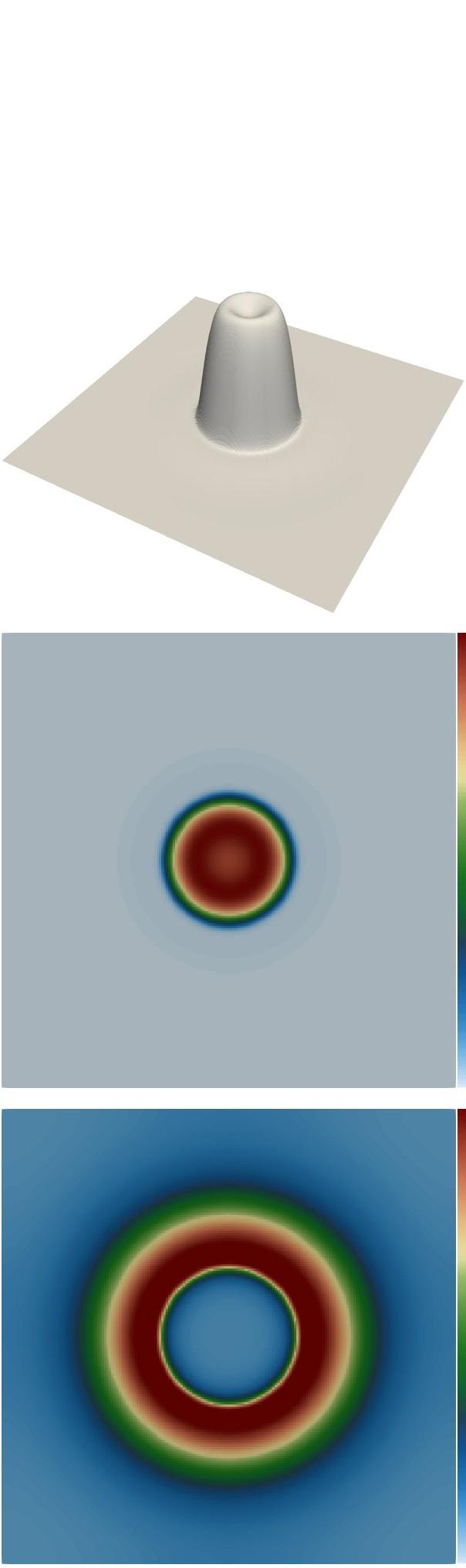}};  % Adjust the width as needed
        % Get the coordinates of the image and add the vertical numbers to the right
                \begin{scope}[shift={(image.south east)}]
            % Adjust positioning of numbers vertically and horizontally
            \node at (0.15, 3.8) {\scalebox{0.45}{\textbf{1.65}}}; % First number (on top)
            \node at (0.15, 2.9) {\scalebox{0.45}{\textbf{0.98}}}; % Second number (in the middle)
            \node at (0.15, 2.0) {\scalebox{0.45}{\textbf{0.30}}}; % Third number (on bottom)
            \node at (0.15, 1.85) {\scalebox{0.45}{\textbf{0.06}}}; % First number (on top)
            \node at (0.15, 0.9) {\scalebox{0.45}{\textbf{0.03}}}; % Second number (in the middle)
            \node at (0.15, 0.05) {\scalebox{0.45}{\textbf{0.00}}}; % Third number (on bottom)
        \end{scope}        
    \end{tikzpicture}}
        \hspace{-0.43cm}
        \subfigure[t=0.4]{
              \begin{tikzpicture}
        \node [anchor=south west,inner sep=0] (image) at (0,0) {\includegraphics[width=0.14\textwidth]{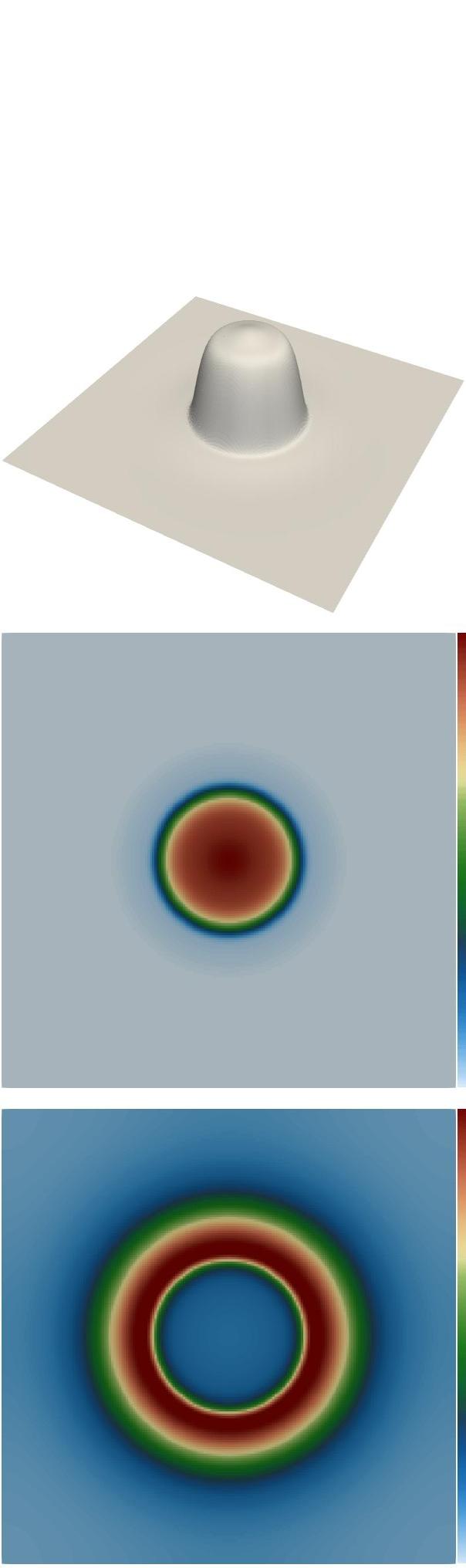}};  % Adjust the width as needed
                \begin{scope}[shift={(image.south east)}]
            % Adjust positioning of numbers vertically and horizontally
            \node at (0.15, 3.8) {\scalebox{0.45}{\textbf{1.38}}}; % First number (on top)
            \node at (0.15, 2.9) {\scalebox{0.45}{\textbf{0.84}}}; % Second number (in the middle)
            \node at (0.15, 2.0) {\scalebox{0.45}{\textbf{0.30}}}; % Third number (on bottom)
            \node at (0.15, 1.85) {\scalebox{0.45}{\textbf{0.06}}}; % First number (on top)
            \node at (0.15, 0.9) {\scalebox{0.45}{\textbf{0.03}}}; % Second number (in the middle)
            \node at (0.15, 0.05) {\scalebox{0.45}{\textbf{0.00}}}; % Third number (on bottom)
        \end{scope}        
    \end{tikzpicture}}
        \hspace{-0.43cm}
\subfigure[t=0.6]{
              \begin{tikzpicture}
        \node [anchor=south west,inner sep=0] (image) at (0,0) {\includegraphics[width=0.14\textwidth]{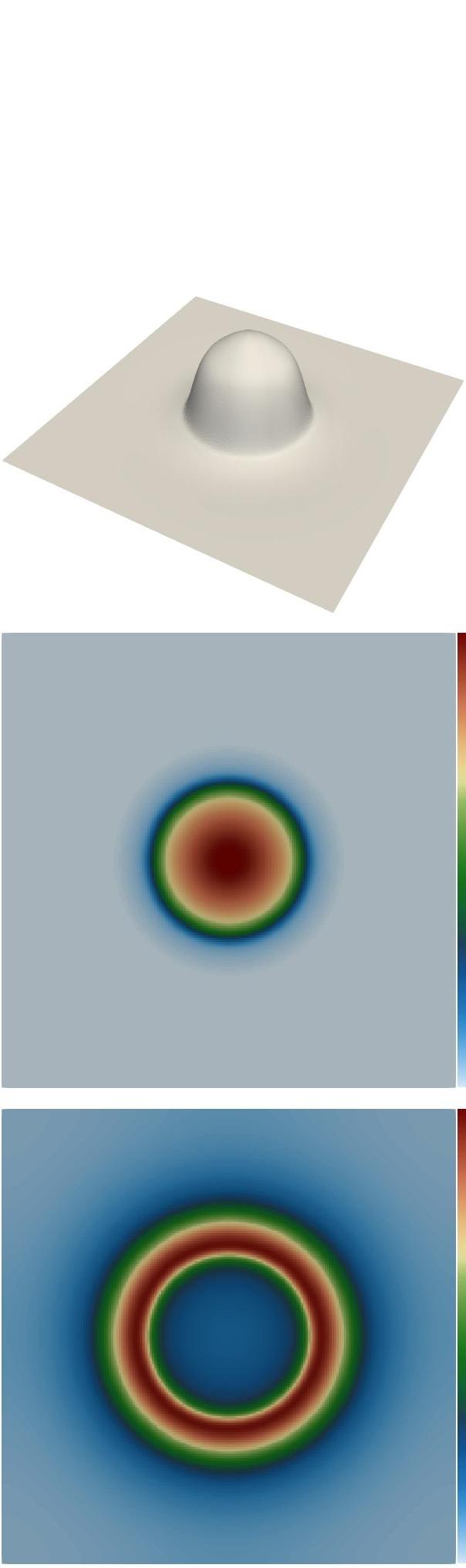}};  % Adjust the width as needed
                        \begin{scope}[shift={(image.south east)}]
            % Adjust positioning of numbers vertically and horizontally
                        % Adjust positioning of numbers vertically and horizontally
            \node at (0.15, 3.8) {\scalebox{0.45}{\textbf{1.30}}}; % First number (on top)
            \node at (0.15, 2.9) {\scalebox{0.45}{\textbf{0.80}}}; % Second number (in the middle)
            \node at (0.15, 2.0) {\scalebox{0.45}{\textbf{0.30}}}; % Third number (on bottom)
            \node at (0.15, 1.85) {\scalebox{0.45}{\textbf{0.06}}}; % First number (on top)
            \node at (0.15, 0.9) {\scalebox{0.45}{\textbf{0.03}}}; % Second number (in the middle)
            \node at (0.15, 0.05) {\scalebox{0.45}{\textbf{0.00}}}; % Third number (on bottom)
        \end{scope}        
    \end{tikzpicture}}
        \hspace{-0.43cm}
    \subfigure[t=0.8]{
              \begin{tikzpicture}
        \node [anchor=south west,inner sep=0] (image) at (0,0) {\includegraphics[width=0.14\textwidth]{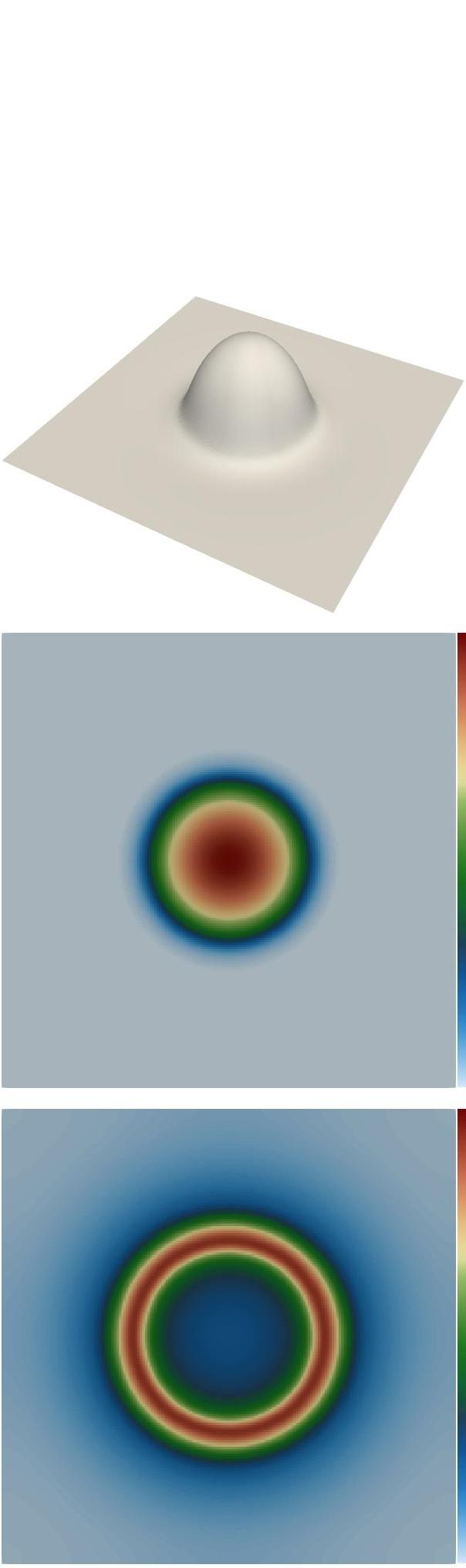}};  % Adjust the width as needed
                \begin{scope}[shift={(image.south east)}]
            % Adjust positioning of numbers vertically and horizontally
            \node at (0.15, 3.8) {\scalebox{0.45}{\textbf{1.30}}}; % First number (on top)
            \node at (0.15, 2.9) {\scalebox{0.45}{\textbf{0.80}}}; % Second number (in the middle)
            \node at (0.15, 2.0) {\scalebox{0.45}{\textbf{0.30}}}; % Third number (on bottom)
            \node at (0.15, 1.85) {\scalebox{0.45}{\textbf{0.06}}}; % First number (on top)
            \node at (0.15, 0.9) {\scalebox{0.45}{\textbf{0.03}}}; % Second number (in the middle)
            \node at (0.15, 0.05) {\scalebox{0.45}{\textbf{0.00}}}; % Third number (on bottom)
        \end{scope}        
    \end{tikzpicture}}
        \hspace{-0.43cm}
\subfigure[t=1.0]{
              \begin{tikzpicture}
        \node [anchor=south west,inner sep=0] (image) at (0,0) {\includegraphics[width=0.14\textwidth]{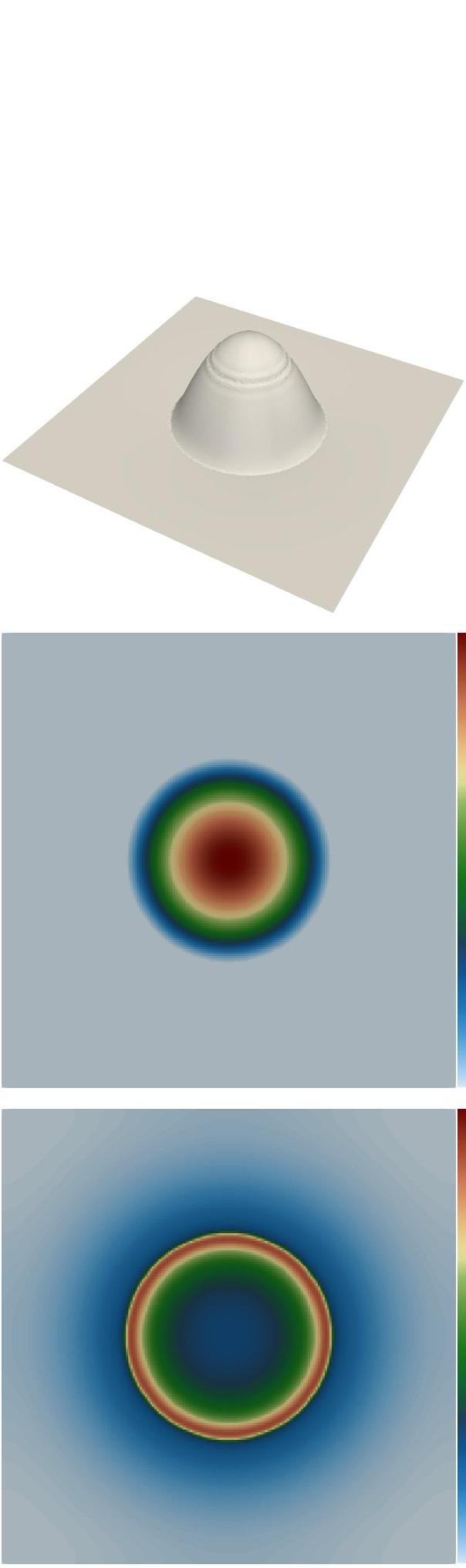}};  % Adjust the width as needed
                \begin{scope}[shift={(image.south east)}]
            % Adjust positioning of numbers vertically and horizontally
            \node at (0.15, 3.8) {\scalebox{0.45}{\textbf{1.30}}}; % First number (on top)
            \node at (0.15, 2.9) {\scalebox{0.45}{\textbf{0.80}}}; % Second number (in the middle)
            \node at (0.15, 2.0) {\scalebox{0.45}{\textbf{0.30}}}; % Third number (on bottom)
            \node at (0.15, 1.85) {\scalebox{0.45}{\textbf{0.06}}}; % First number (on top)
            \node at (0.15, 0.9) {\scalebox{0.45}{\textbf{0.03}}}; % Second number (in the middle)
            \node at (0.15, 0.05) {\scalebox{0.45}{\textbf{0.00}}}; % Third number (on bottom)
        \end{scope}        
    \end{tikzpicture}}
\caption{
% Example \ref{ex2}: 
\blue{
Case 2 (Droplet spreading): Snapshots of (top row) 3D plots of the controlled surface height $h$, (middle row) contour plots of the controlled surface height, and (bottom row) contour plots of the activity field $\zeta$ at different times. The corresponding animation video can be found in the GitHub repository \citep{github}.
% Top row: snapshots of 3D plots for the controlled surface height $h$ at different times;
% Middle row: snapshots of the controlled surface height contours at different times;
% Bottom row: snapshots of the activity field $\zeta$ at different times.
}
}
\label{fig:T33a}
\end{figure}
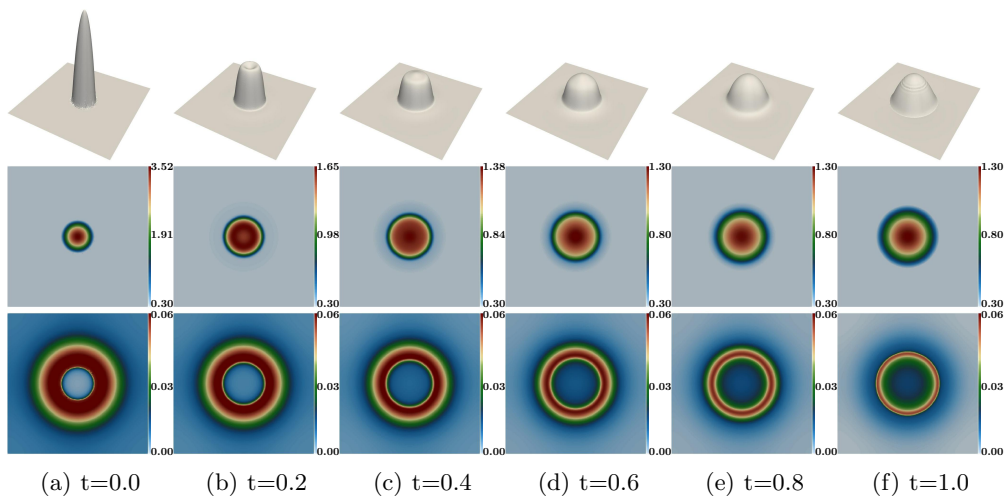

\subsubsection*{Case 3: Droplet bead-up.}
This represents the reverse process of Case 2, where our objective is to induce droplet bead-up (i.e. dewetting). Therefore, we set the initial and target surface heights as 
\begin{align}
h_0(x, y) =&\; \epsilon+
\frac{5}{12}\left(1-\frac{75}{8}((x-0.5)^2+(y-0.5)^2)\right)_+,\\
\blue{h_{T}(x, y)} =&\; \epsilon+
\frac{10}{3}\left(1-75((x-0.5)^2+(y-0.5)^2)\right)_+,
\end{align}
where the initial and target profiles are reversed compared to the example in Case 2. 
This models the MFC of an initial parabolic droplet centered at $(0.5, 0.5)$ with half-width $w=
\frac{2\sqrt{2}}{5\sqrt{3}}$
which beads up and evolves into the target droplet with half-width $w=\frac{1}{5\sqrt{3}}$.
Snapshots of the simulation results for the scheme \eqref{saddleH} are presented in Figures~\ref{fig:T37a}. 
In this case, we observe more pattern formation during the droplet bead-up process, where capillary waves are generated in the early stage before the droplet reaches the target shape. Compared to the target droplet profile, the obtained profile at the terminal time $t = 1$ has a slightly elevated height near the contact line.
\blue{Moreover, the activity field remains radially symmetric and mostly extensile ($\zeta < 0$)
throughout the simulation, reaching its maximum absolute value near the droplet's contact lines at each time.
}

\begin{figure}
\centering
% \subfigure[t=0.0]{
% \includegraphics[width=0.15\textwidth]{figs/zaC37T0h.jpg}}
% \subfigure[t=0.2]{
% \includegraphics[width=0.15\textwidth]{figs/zaC37T2h.jpg}}
% \subfigure[t=0.4]{
% \includegraphics[width=0.15\textwidth]{figs/zaC37T4h.jpg}}
% \subfigure[t=0.6]{
% \includegraphics[width=0.15\textwidth]{figs/zaC37T6h.jpg}}
% \subfigure[t=0.8]{
% \includegraphics[width=0.15\textwidth]{figs/zaC37T8h.jpg}}
% \subfigure[t=1.0]{
% \includegraphics[width=0.15\textwidth]{figs/zaC37T10h.jpg}}
\subfigure[t=0.0]{
              \begin{tikzpicture}
        \node [anchor=south west,inner sep=0] (image) at (0,0) {\includegraphics[width=0.14\textwidth]{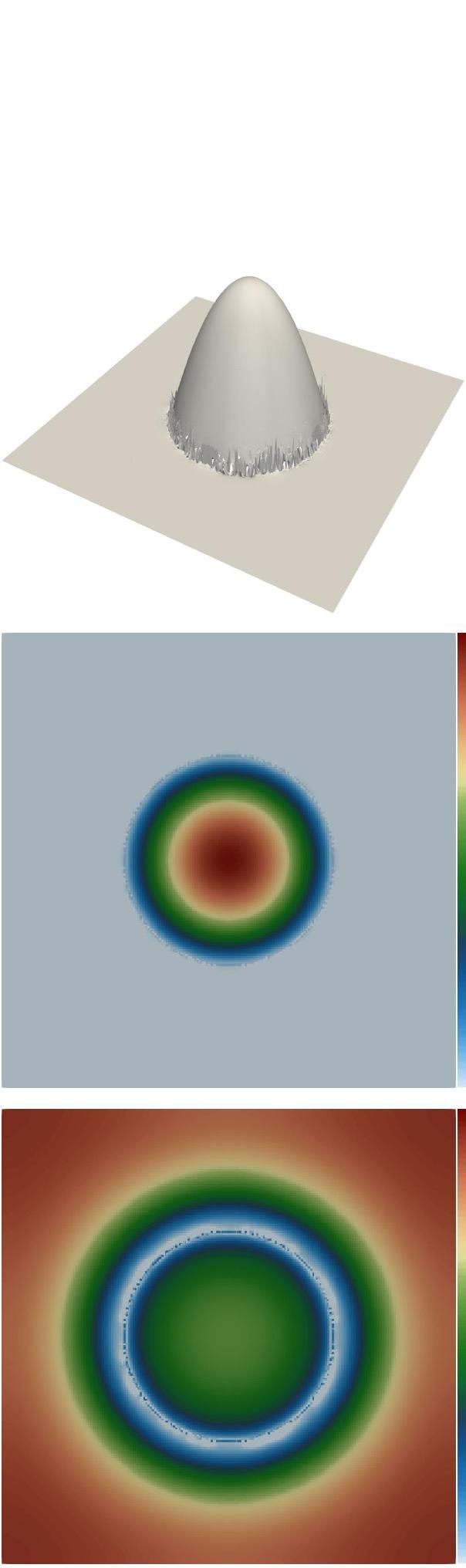}};  % Adjust the width as needed
                \begin{scope}[shift={(image.south east)}]
            % Adjust positioning of numbers vertically and horizontally
            \node at (0.15, 3.8) {\scalebox{0.45}{\textbf{2.00}}}; % First number (on top)
            \node at (0.15, 2.9) {\scalebox{0.45}{\textbf{1.15}}}; % Second number (in the middle)
            \node at (0.15, 2.0) {\scalebox{0.45}{\textbf{0.30}}}; % Third number (on bottom)
            \node at (0.15, 1.85) {\scalebox{0.45}{\textbf{0.01}}}; % First number (on top)
            \node at (0.15, 0.9) {\scalebox{0.45}{\textbf{-0.02}}}; % Second number (in the middle)
            \node at (0.14, 0.05) {\scalebox{0.45}{\textbf{-0.05}}}; % Third number (on bottom)
        \end{scope}        
    \end{tikzpicture}}
        \hspace{-0.43cm}
        \subfigure[t=0.2]{
              \begin{tikzpicture}
        \node [anchor=south west,inner sep=0] (image) at (0,0) {\includegraphics[width=0.14\textwidth]{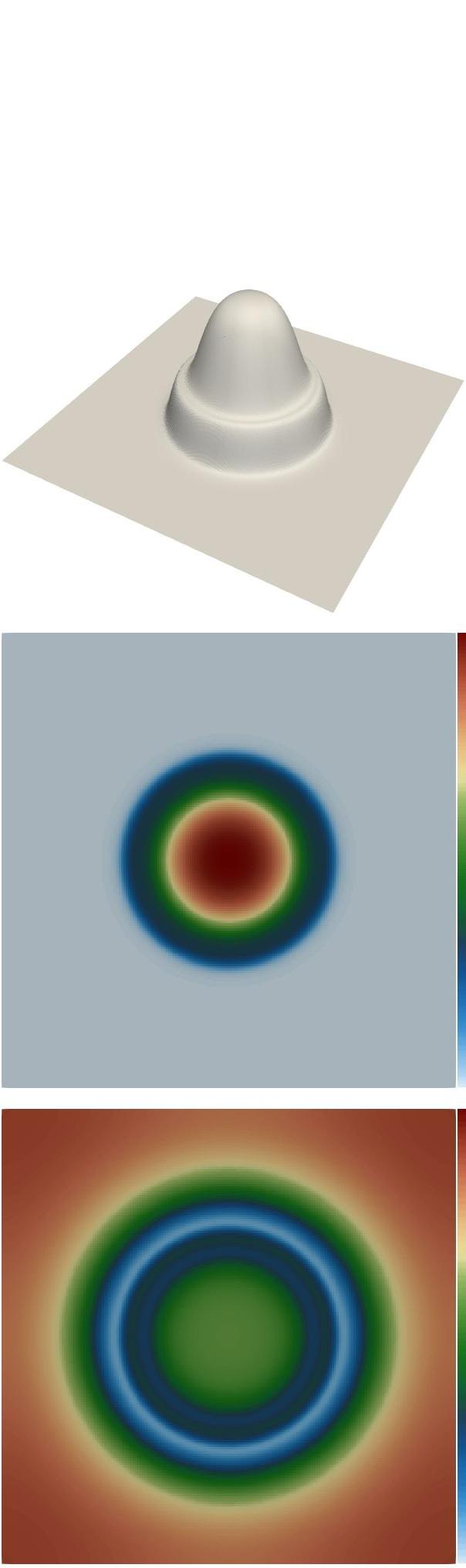}};  % Adjust the width as needed
        % Get the coordinates of the image and add the vertical numbers to the right
                \begin{scope}[shift={(image.south east)}]
            % Adjust positioning of numbers vertically and horizontally
            \node at (0.15, 3.8) {\scalebox{0.45}{\textbf{1.77}}}; % First number (on top)
            \node at (0.15, 2.9) {\scalebox{0.45}{\textbf{1.04}}}; % Second number (in the middle)
            \node at (0.15, 2.0) {\scalebox{0.45}{\textbf{0.30}}}; % Third number (on bottom)
            \node at (0.15, 1.85) {\scalebox{0.45}{\textbf{0.01}}}; % First number (on top)
            \node at (0.15, 0.9) {\scalebox{0.45}{\textbf{-0.02}}}; % Second number (in the middle)
            \node at (0.14, 0.05) {\scalebox{0.45}{\textbf{-0.05}}}; % Third number (on bottom)
        \end{scope}        
    \end{tikzpicture}}
        \hspace{-0.43cm}
        \subfigure[t=0.4]{
              \begin{tikzpicture}
        \node [anchor=south west,inner sep=0] (image) at (0,0) {\includegraphics[width=0.14\textwidth]{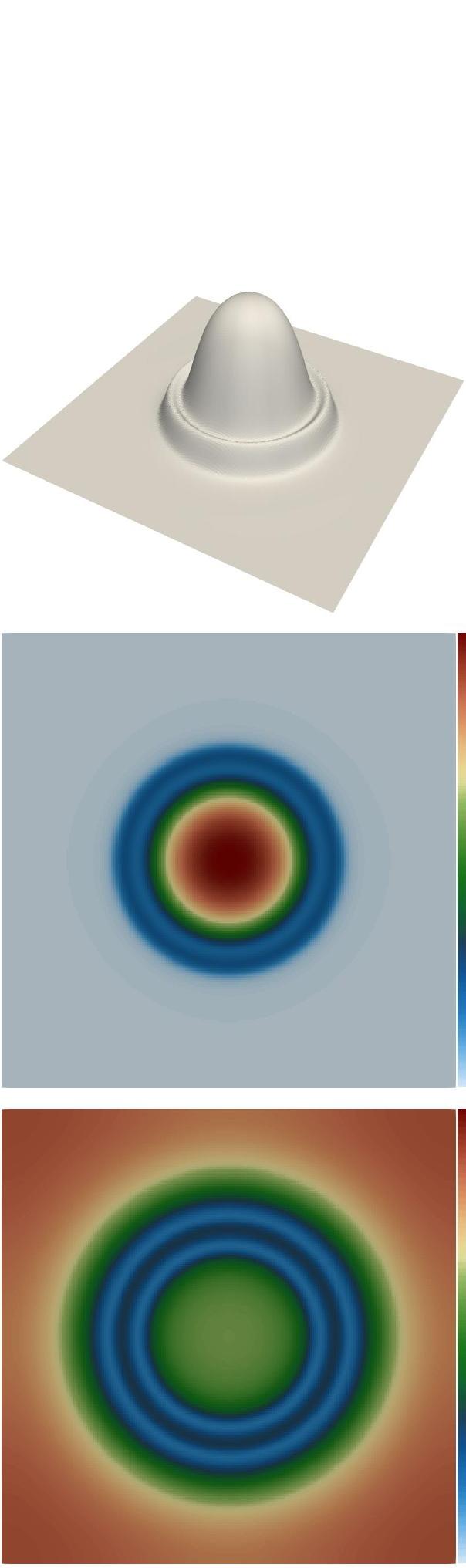}};  % Adjust the width as needed
                \begin{scope}[shift={(image.south east)}]
            % Adjust positioning of numbers vertically and horizontally
            \node at (0.15, 3.8) {\scalebox{0.45}{\textbf{1.75}}}; % First number (on top)
            \node at (0.15, 2.9) {\scalebox{0.45}{\textbf{1.03}}}; % Second number (in the middle)
            \node at (0.15, 2.0) {\scalebox{0.45}{\textbf{0.30}}}; % Third number (on bottom)
            \node at (0.15, 1.85) {\scalebox{0.45}{\textbf{0.01}}}; % First number (on top)
            \node at (0.15, 0.9) {\scalebox{0.45}{\textbf{-0.02}}}; % Second number (in the middle)
            \node at (0.14, 0.05) {\scalebox{0.45}{\textbf{-0.05}}}; % Third number (on bottom)
        \end{scope}        
    \end{tikzpicture}}
        \hspace{-0.43cm}
\subfigure[t=0.6]{
              \begin{tikzpicture}
        \node [anchor=south west,inner sep=0] (image) at (0,0) {\includegraphics[width=0.14\textwidth]{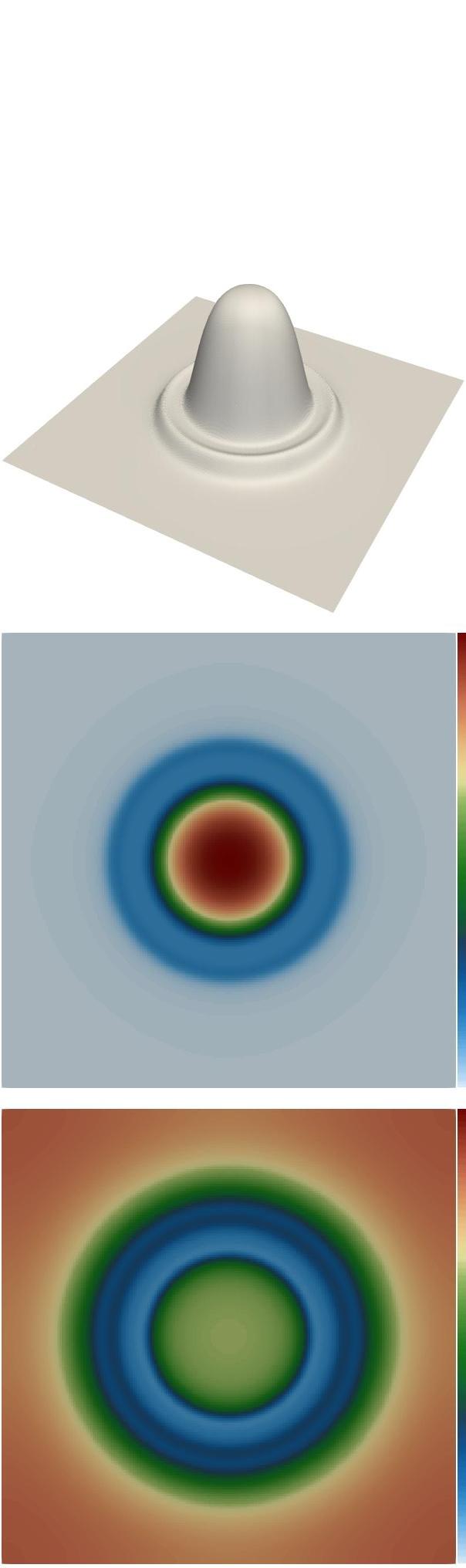}};  % Adjust the width as needed
                        \begin{scope}[shift={(image.south east)}]
            % Adjust positioning of numbers vertically and horizontally
                        % Adjust positioning of numbers vertically and horizontally
            \node at (0.15, 3.8) {\scalebox{0.45}{\textbf{1.85}}}; % First number (on top)
            \node at (0.15, 2.9) {\scalebox{0.45}{\textbf{1.08}}}; % Second number (in the middle)
            \node at (0.15, 2.0) {\scalebox{0.45}{\textbf{0.30}}}; % Third number (on bottom)
            \node at (0.15, 1.85) {\scalebox{0.45}{\textbf{0.01}}}; % First number (on top)
            \node at (0.15, 0.9) {\scalebox{0.45}{\textbf{-0.02}}}; % Second number (in the middle)
            \node at (0.14, 0.05) {\scalebox{0.45}{\textbf{-0.05}}}; % Third number (on bottom)
        \end{scope}        
    \end{tikzpicture}}
        \hspace{-0.43cm}
    \subfigure[t=0.8]{
              \begin{tikzpicture}
        \node [anchor=south west,inner sep=0] (image) at (0,0) {\includegraphics[width=0.14\textwidth]{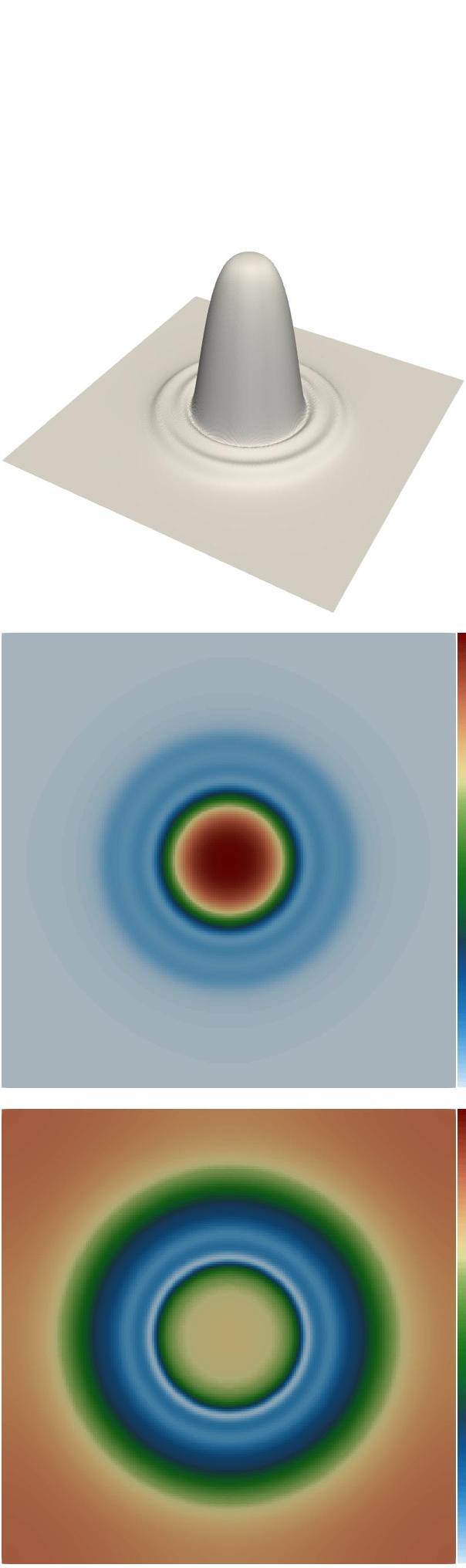}};  % Adjust the width as needed
                \begin{scope}[shift={(image.south east)}]
            % Adjust positioning of numbers vertically and horizontally
            \node at (0.15, 3.8) {\scalebox{0.45}{\textbf{2.20}}}; % First number (on top)
            \node at (0.15, 2.9) {\scalebox{0.45}{\textbf{1.25}}}; % Second number (in the middle)
            \node at (0.15, 2.0) {\scalebox{0.45}{\textbf{0.30}}}; % Third number (on bottom)
            \node at (0.15, 1.85) {\scalebox{0.45}{\textbf{0.01}}}; % First number (on top)
            \node at (0.15, 0.9) {\scalebox{0.45}{\textbf{-0.02}}}; % Second number (in the middle)
            \node at (0.14, 0.05) {\scalebox{0.45}{\textbf{-0.05}}}; % Third number (on bottom)
        \end{scope}        
    \end{tikzpicture}}
        \hspace{-0.43cm}
\subfigure[t=1.0]{
              \begin{tikzpicture}
        \node [anchor=south west,inner sep=0] (image) at (0,0) {\includegraphics[width=0.14\textwidth]{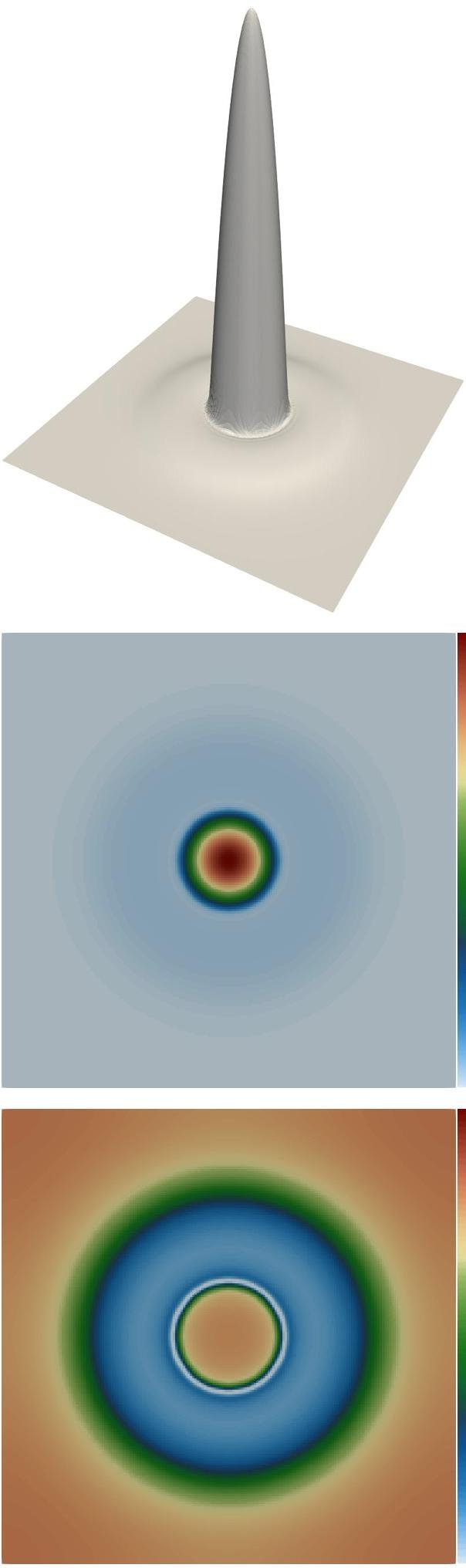}};  % Adjust the width as needed
                \begin{scope}[shift={(image.south east)}]
            % Adjust positioning of numbers vertically and horizontally
            \node at (0.15, 3.8) {\scalebox{0.45}{\textbf{4.80}}}; % First number (on top)
            \node at (0.15, 2.9) {\scalebox{0.45}{\textbf{2.55}}}; % Second number (in the middle)
            \node at (0.15, 2.0) {\scalebox{0.45}{\textbf{0.30}}}; % Third number (on bottom)
            \node at (0.15, 1.85) {\scalebox{0.45}{\textbf{0.01}}}; % First number (on top)
            \node at (0.15, 0.9) {\scalebox{0.45}{\textbf{-0.02}}}; % Second number (in the middle)
            \node at (0.14, 0.05) {\scalebox{0.45}{\textbf{-0.05}}}; % Third number (on bottom)
        \end{scope}        
    \end{tikzpicture}}
\caption{
% Example \ref{ex2}: 
\blue{
Case 3 (Droplet bead-up):
Snapshots of (top row) 3D plots of the controlled surface height $h$, (middle row) contour plots of the controlled surface height, and (bottom row) contour plots of the activity field $\zeta$ at different times. The corresponding animation video can be found in the GitHub repository \citep{github}.
% Top row: snapshots of 3D plots for the controlled surface height $h$ at different times;
% Middle row: snapshots of the controlled surface height contours at different times;
% Bottom row: snapshots of the activity field $\zeta$ at different times.
}
}
\label{fig:T37a}
\end{figure}
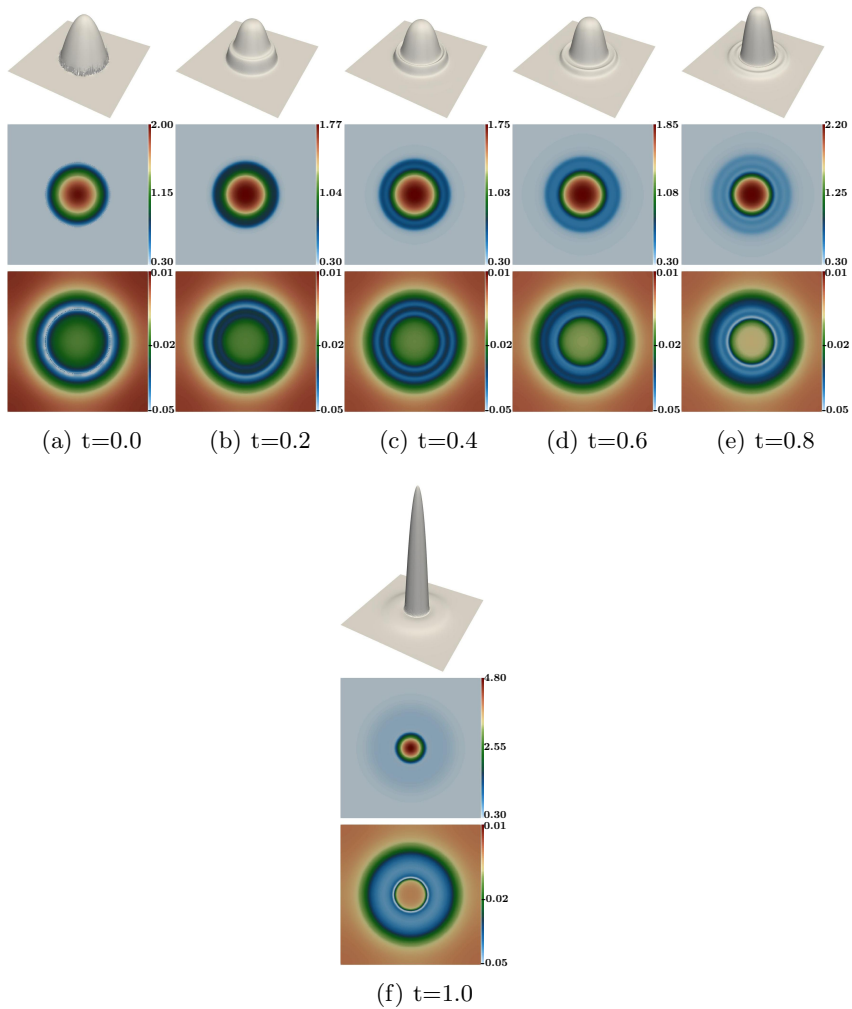

% \begin{figure}
% \centering
% \subfigure[t=0.0]{
% \includegraphics[width=0.46\textwidth]{figs/T37_0.png}}
% \subfigure[t=0.3]{
% \includegraphics[width=0.46\textwidth]{figs/T37_1.png}}
% \subfigure[t=0.7]{
% \includegraphics[width=0.46\textwidth]{figs/T37_2.png}}
% \subfigure[t=1.0]{
% \includegraphics[width=0.46\textwidth]{figs/T37_3.png}}
% \caption{
% % Example \ref{ex2}: 
% Case 3 (Droplet Bead-up). 
% Snapshots of surface height contour at different times.
% }
% \label{fig:T37a}
% \end{figure}

% \begin{figure}
% \centering
% % \includegraphics[width=0.5\textwidth]{figs/T37.png}
% \includegraphics[width=0.5\textwidth]{figs/3d_T3.pdf}
% \caption{
% % Example \ref{ex2}: 
% Case 3 (Droplet bead-up). 
% Time evolution of the controlled surface height along the north-east diagonal cutline shown in Figure~\ref{fig:T37a}.
% }
% \label{fig:T37b}
% \end{figure}

\subsubsection*{Case 4: Droplet merging.}
Finally, we demonstrate the application of MFC for droplet merging and splitting, which are more complex droplet manipulation techniques widely used in biological and chemical applications \citep{nan2023self}. For the droplet merging case, we control the coalescence of four small droplets initially placed on a two-dimensional domain, where the initial  surface height profile is 
\begin{subequations}
\begin{align}
\label{mg_init}
h_0(x, y) =&\; \epsilon+\sum_{i=1}^{4}\tfrac{5}{6}[1-75((x-x_i)^2+(y-y_i)^2)]_+,
% \\
% \blue{h_{T}(x, y)}  =&\; \epsilon+
% \tfrac{5}{12}[1-\tfrac{75}{8}((x-0.5)^2+(y-0.5)^2)]_+,
\end{align}
where the positions of the peaks of the initial droplets are $(x_1,y_1) = (0.3, 0.3)$, $(x_2,y_2) = (0.3, 0.7)$, $(x_3, y_3) = (0.7,0.3)$, and $(x_4, y_4) = (0.7,0.7)$.

\begin{figure}
\centering
% \subfigure[t=0.0]{
% \includegraphics[width=0.15\textwidth]{figs/zaC34T0h.jpg}}
% \subfigure[t=0.2]{
% \includegraphics[width=0.15\textwidth]{figs/zaC34T2h.jpg}}
% \subfigure[t=0.4]{
% \includegraphics[width=0.15\textwidth]{figs/zaC34T4h.jpg}}
% \subfigure[t=0.6]{
% \includegraphics[width=0.15\textwidth]{figs/zaC34T6h.jpg}}
% \subfigure[t=0.8]{
% \includegraphics[width=0.15\textwidth]{figs/zaC34T8h.jpg}}
% \subfigure[t=1.0]{
% \includegraphics[width=0.15\textwidth]{figs/zaC34T10h.jpg}}
\subfigure[t=0.0]{
              \begin{tikzpicture}
        \node [anchor=south west,inner sep=0] (image) at (0,0) {\includegraphics[width=0.14\textwidth]{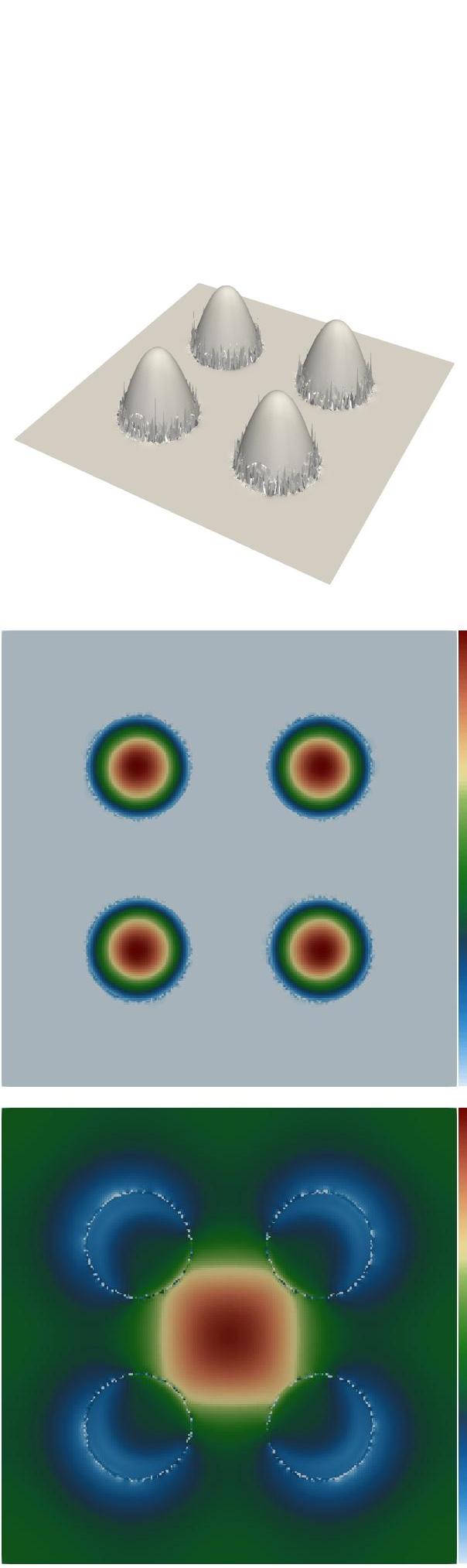}};  % Adjust the width as needed
                \begin{scope}[shift={(image.south east)}]
            % Adjust positioning of numbers vertically and horizontally
            \node at (0.15, 3.8) {\scalebox{0.45}{\textbf{1.15}}}; % First number (on top)
            \node at (0.15, 2.9) {\scalebox{0.45}{\textbf{0.73}}}; % Second number (in the middle)
            \node at (0.15, 2.0) {\scalebox{0.45}{\textbf{0.30}}}; % Third number (on bottom)
            \node at (0.15, 1.85) {\scalebox{0.45}{\textbf{0.05}}}; % First number (on top)
            \node at (0.15, 0.9) {\scalebox{0.45}{\textbf{0.01}}}; % Second number (in the middle)
            \node at (0.14, 0.05) {\scalebox{0.45}{\textbf{-0.03}}}; % Third number (on bottom)
        \end{scope}        
    \end{tikzpicture}}
        \hspace{-0.43cm}
        \subfigure[t=0.2]{
              \begin{tikzpicture}
        \node [anchor=south west,inner sep=0] (image) at (0,0) {\includegraphics[width=0.14\textwidth]{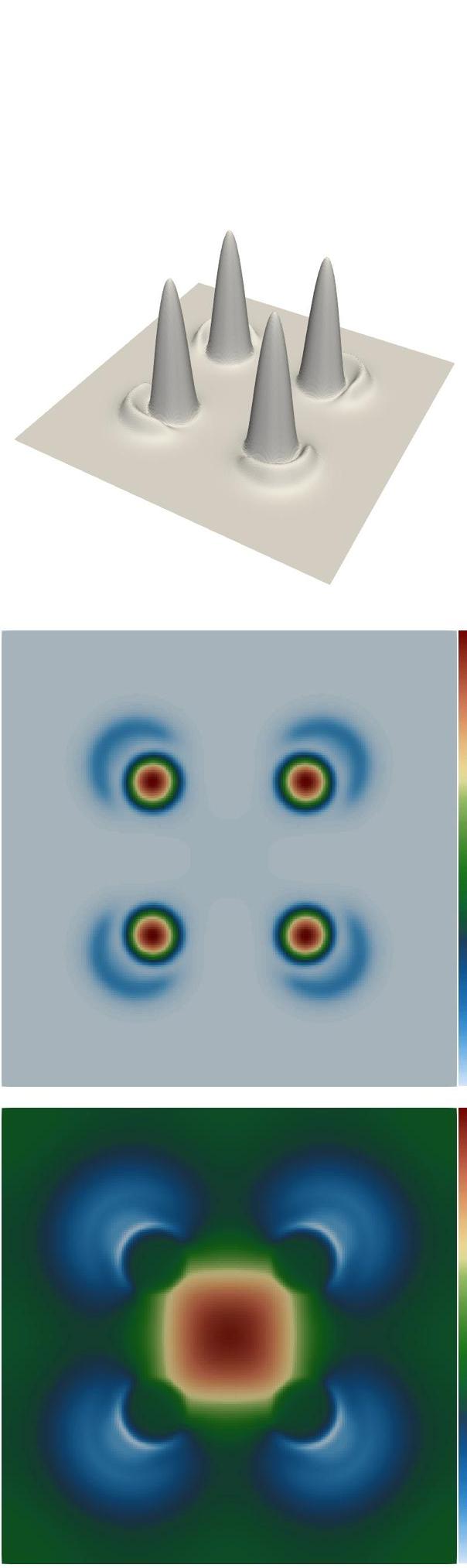}};  % Adjust the width as needed
        % Get the coordinates of the image and add the vertical numbers to the right
                \begin{scope}[shift={(image.south east)}]
            % Adjust positioning of numbers vertically and horizontally
            \node at (0.15, 3.8) {\scalebox{0.45}{\textbf{1.97}}}; % First number (on top)
            \node at (0.15, 2.9) {\scalebox{0.45}{\textbf{1.14}}}; % Second number (in the middle)
            \node at (0.15, 2.0) {\scalebox{0.45}{\textbf{0.30}}}; % Third number (on bottom)
            \node at (0.15, 1.85) {\scalebox{0.45}{\textbf{0.05}}}; % First number (on top)
            \node at (0.15, 0.9) {\scalebox{0.45}{\textbf{0.01}}}; % Second number (in the middle)
            \node at (0.14, 0.05) {\scalebox{0.45}{\textbf{-0.03}}}; % Third number (on bottom)
        \end{scope}        
    \end{tikzpicture}}
        \hspace{-0.43cm}
        \subfigure[t=0.4]{
              \begin{tikzpicture}
        \node [anchor=south west,inner sep=0] (image) at (0,0) {\includegraphics[width=0.14\textwidth]{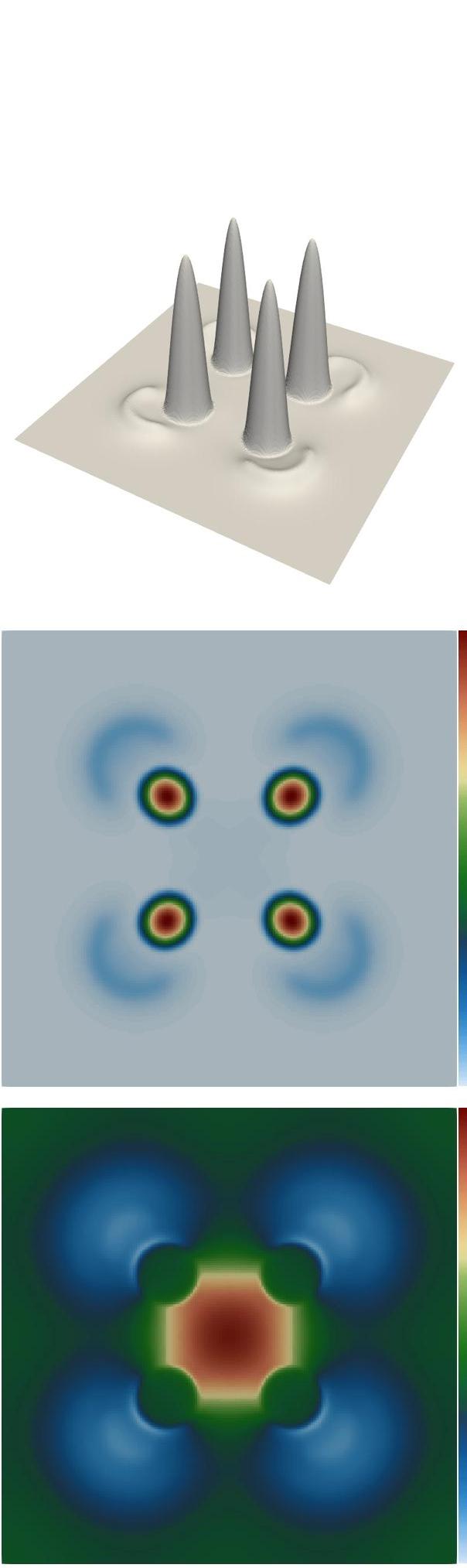}};  % Adjust the width as needed
                \begin{scope}[shift={(image.south east)}]
            % Adjust positioning of numbers vertically and horizontally
            \node at (0.15, 3.8) {\scalebox{0.45}{\textbf{2.20}}}; % First number (on top)
            \node at (0.15, 2.9) {\scalebox{0.45}{\textbf{1.25}}}; % Second number (in the middle)
            \node at (0.15, 2.0) {\scalebox{0.45}{\textbf{0.30}}}; % Third number (on bottom)
            \node at (0.15, 1.85) {\scalebox{0.45}{\textbf{0.05}}}; % First number (on top)
            \node at (0.15, 0.9) {\scalebox{0.45}{\textbf{0.01}}}; % Second number (in the middle)
            \node at (0.14, 0.05) {\scalebox{0.45}{\textbf{-0.03}}}; % Third number (on bottom)
        \end{scope}        
    \end{tikzpicture}}
        \hspace{-0.43cm}
\subfigure[t=0.6]{
              \begin{tikzpicture}
        \node [anchor=south west,inner sep=0] (image) at (0,0) {\includegraphics[width=0.14\textwidth]{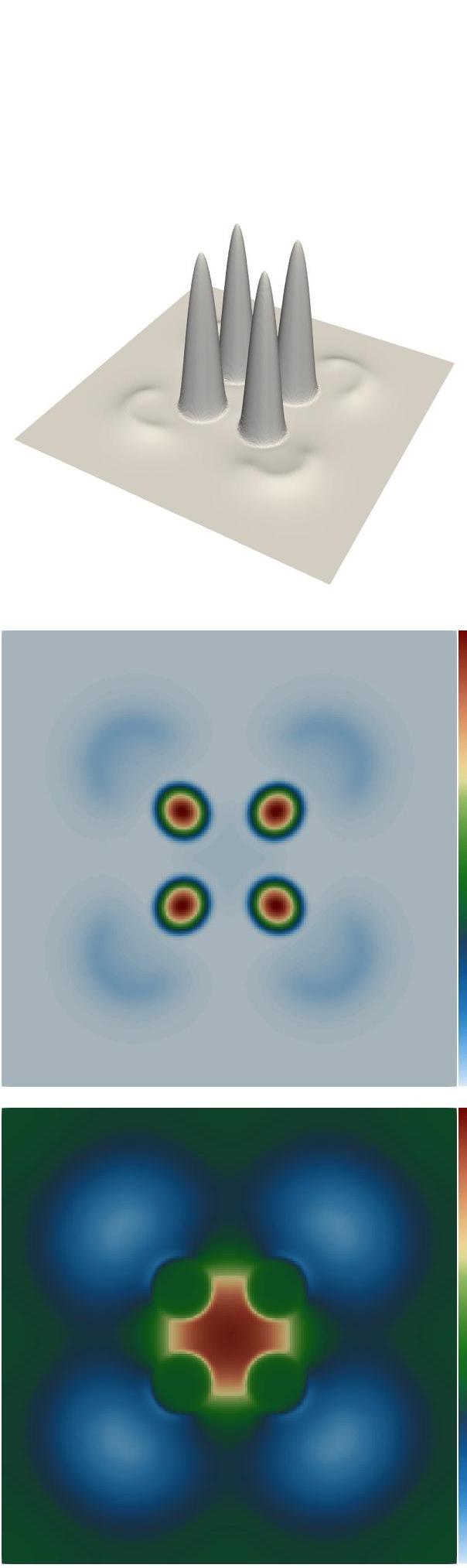}};  % Adjust the width as needed
                        \begin{scope}[shift={(image.south east)}]
            % Adjust positioning of numbers vertically and horizontally
                        % Adjust positioning of numbers vertically and horizontally
            \node at (0.15, 3.8) {\scalebox{0.45}{\textbf{2.20}}}; % First number (on top)
            \node at (0.15, 2.9) {\scalebox{0.45}{\textbf{1.25}}}; % Second number (in the middle)
            \node at (0.15, 2.0) {\scalebox{0.45}{\textbf{0.30}}}; % Third number (on bottom)
            \node at (0.15, 1.85) {\scalebox{0.45}{\textbf{0.05}}}; % First number (on top)
            \node at (0.15, 0.9) {\scalebox{0.45}{\textbf{0.01}}}; % Second number (in the middle)
            \node at (0.14, 0.05) {\scalebox{0.45}{\textbf{-0.03}}}; % Third number (on bottom)
        \end{scope}        
    \end{tikzpicture}}
        \hspace{-0.43cm}
    \subfigure[t=0.8]{
              \begin{tikzpicture}
        \node [anchor=south west,inner sep=0] (image) at (0,0) {\includegraphics[width=0.14\textwidth]{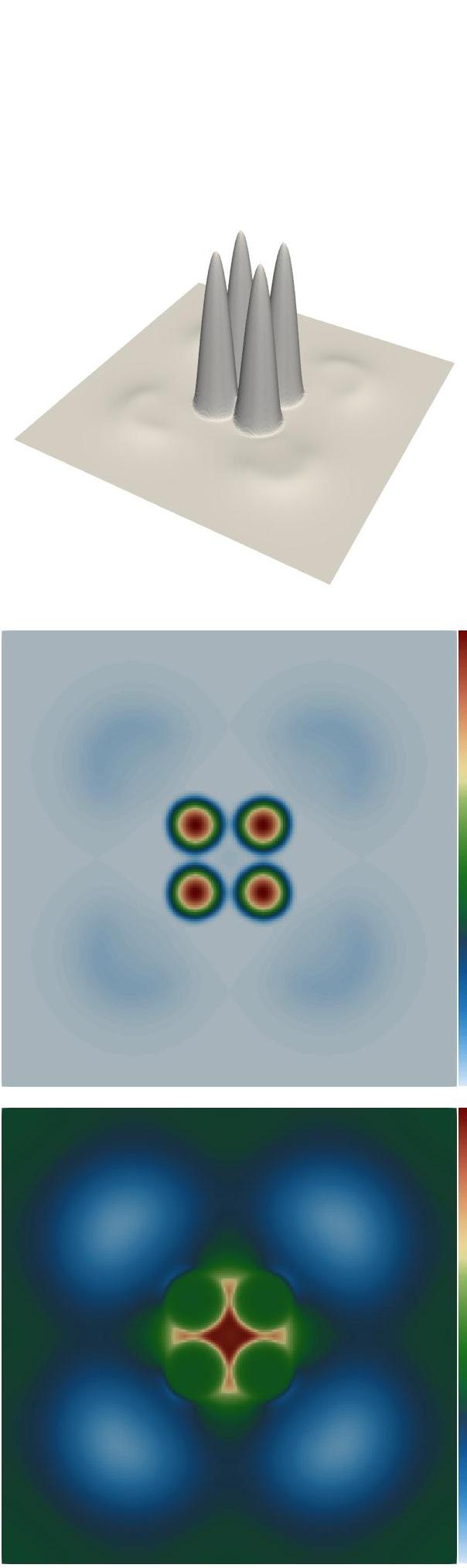}};  % Adjust the width as needed
                \begin{scope}[shift={(image.south east)}]
            % Adjust positioning of numbers vertically and horizontally
            \node at (0.15, 3.8) {\scalebox{0.45}{\textbf{2.20}}}; % First number (on top)
            \node at (0.15, 2.9) {\scalebox{0.45}{\textbf{1.25}}}; % Second number (in the middle)
            \node at (0.15, 2.0) {\scalebox{0.45}{\textbf{0.30}}}; % Third number (on bottom)
            \node at (0.15, 1.85) {\scalebox{0.45}{\textbf{0.05}}}; % First number (on top)
            \node at (0.15, 0.9) {\scalebox{0.45}{\textbf{0.01}}}; % Second number (in the middle)
            \node at (0.14, 0.05) {\scalebox{0.45}{\textbf{-0.03}}}; % Third number (on bottom)
        \end{scope}        
    \end{tikzpicture}}
        \hspace{-0.43cm}
\subfigure[t=1.0]{
              \begin{tikzpicture}
        \node [anchor=south west,inner sep=0] (image) at (0,0) {\includegraphics[width=0.14\textwidth]{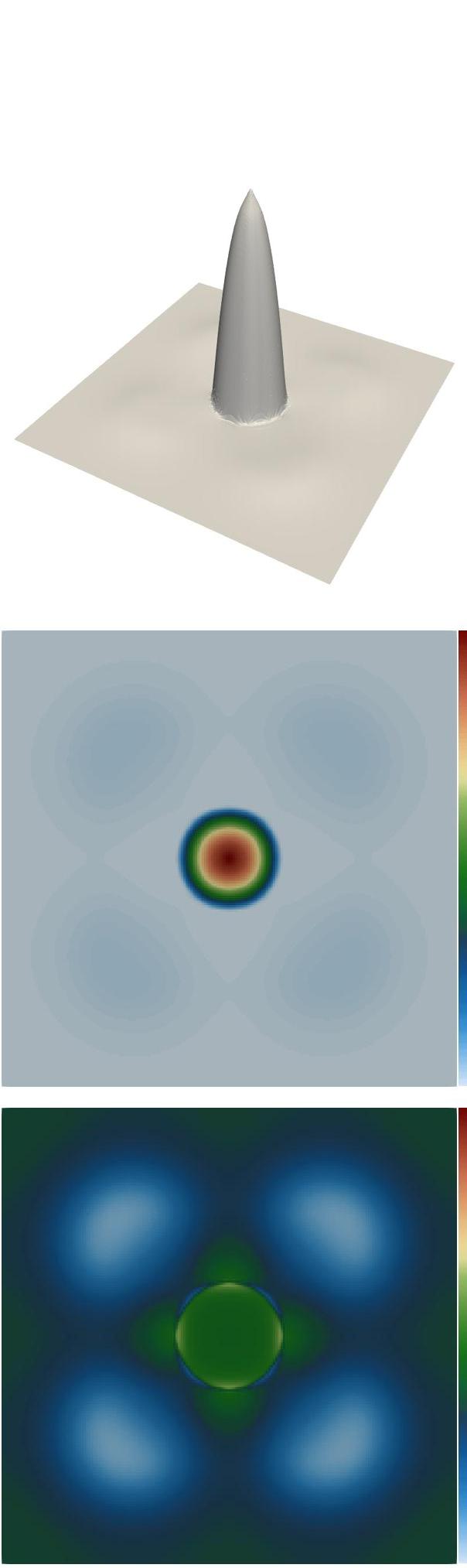}};  % Adjust the width as needed
                \begin{scope}[shift={(image.south east)}]
            % Adjust positioning of numbers vertically and horizontally
            \node at (0.15, 3.8) {\scalebox{0.45}{\textbf{2.87}}}; % First number (on top)
            \node at (0.15, 2.9) {\scalebox{0.45}{\textbf{1.55}}}; % Second number (in the middle)
            \node at (0.15, 2.0) {\scalebox{0.45}{\textbf{0.30}}}; % Third number (on bottom)
            \node at (0.15, 1.85) {\scalebox{0.45}{\textbf{0.05}}}; % First number (on top)
            \node at (0.15, 0.9) {\scalebox{0.45}{\textbf{0.01}}}; % Second number (in the middle)
            \node at (0.14, 0.05) {\scalebox{0.45}{\textbf{-0.03}}}; % Third number (on bottom)
        \end{scope}        
    \end{tikzpicture}}
\caption{
% Example \ref{ex2}: 
\blue{
    Case 4 (Droplet merging with a symmetric target profile \eqref{mgA}):
Snapshots of (top row) 3D plots of the controlled surface height $h$, (middle row) contour plots of the controlled surface height, and (bottom row) contour plots of the activity field $\zeta$ at different times. The corresponding animation video can be found in the GitHub repository \citep{github}.
%     Top row: snapshots of 3D plots for the controlled surface height $h$ at different times;
% Middle row: snapshots of the controlled surface height contours at different times;
% Bottom row: snapshots of the activity field $\zeta$ at different times.
}
}
\label{fig:T34a}
\end{figure}
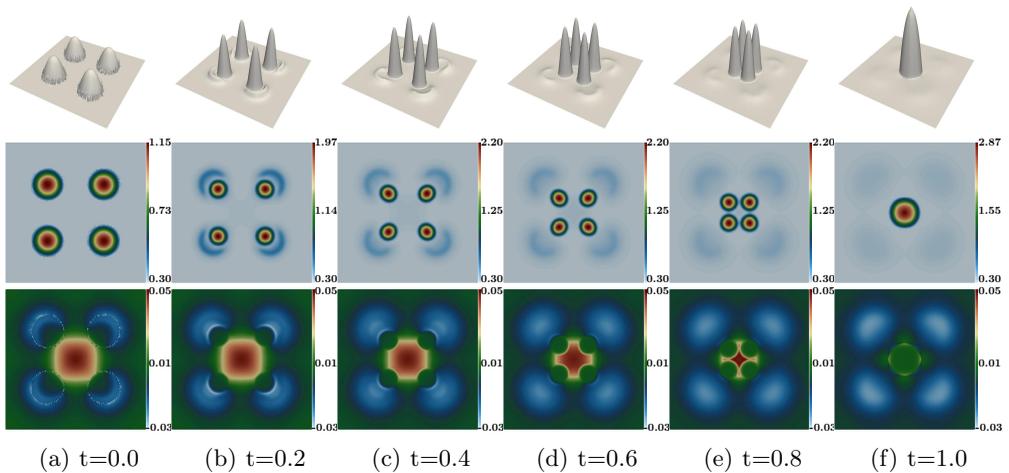
\blue{
We consider two sub-cases for the terminal surface height: the first sub-case has a symmetric terminal droplet profile centered at 
$(0.5, 0.5)$ which has the same total mass as the initial profile:
\begin{align}
    \label{mgA}
    h_{T}(x, y)  =&\; \epsilon+
\frac{5}{12}\left(1-\frac{75}{8}((x-0.5)^2+(y-0.5)^2)\right)_+,
\end{align}
while the second sub-case has a skewed droplet profile elongated in the $x$-direction
centered at $(0.6,0.55)$, which has a different total mass as the initial profile:
\begin{align}
    \label{mgB}
    h_{T}(x, y)  =&\; \epsilon+
\frac{5}{12}\left(1-\frac{75}{8}(4(x-0.6)^2+(y-0.55)^2)\right)_+,
\end{align}
The second sub-case showcases the effect of asymmetry in the evolution of the controlled surface height.
}
\end{subequations}
% \begin{align*}
% h_0(x, y) =&\; \epsilon+
% \frac{5}{6}\left(1-75((x-0.3)^2+(y-0.3)^2)\right)_+ \\
% &\;\;+\frac{5}{6}\left(1-75((x-0.3)^2+(y-0.7)^2)\right)_+\\
% &\;\;+\frac{5}{6}\left(1-75((x-0.7)^2+(y-0.3)^2)\right)_+\\
% &\;\;+\frac{5}{6}\left(1-75((x-0.7)^2+(y-0.7)^2)\right)_+,\\
% h_{\mathrm{trg}}(x, y)  =&\; \epsilon+
% \frac{5}{12}\left(1-\frac{75}{8}((x-0.5)^2+(y-0.5)^2)\right)_+.
% \end{align*}

\blue{Figure~\ref{fig:T34a} presents snapshots of the simulation results for the first sub-case with the symmetric target surface height \eqref{mgA}.
Similar to the droplet transport example discussed in Case 1, the controlled droplets quickly bead up and start moving toward the center of the domain, where the target droplet is placed. 
Capillary waves form behind the droplets as they shift towards the target position. Near the terminal time $t = 1$, the droplets coalesce, forming a single droplet at the center of the domain.  Notably, in this test case, contact line oscillations are only observed at the initial time $t=0$ due to the non-smoothness of the initial data. A radially symmetric and smooth solution profile for $h(t, \bm{x})$ is observed for $t > 0$. The activity field $\zeta$ at different times is plotted in the bottom row of Figure~\ref{fig:T34a}. This shows that in the early stage, high contractile active stress ($\zeta > 0$) appears near the center of the domain, whereas extensile active stress ($\zeta < 0$) around the individual droplets pushes the fluid towards the center of the domain. As the droplets become closer to each other, the corresponding contractile active stress becomes more concentrated near the center of the domain, and the magnitude of the activity field diminishes as the terminal droplet is formed.
}

% This models the MFC of four initial parabolic droplets 
% moving towards the center of the domain to form a single droplet through droplet coalescence.
% \begin{figure}
%     \centering
%  \mbox{ (a) 
%  % \includegraphics[width=0.46\textwidth]{figs/T34_0.png}
%  \includegraphics[width=0.46\textwidth]{figs/T4_0.0000.jpeg}
%    (b) 
%    % \includegraphics[width=0.46\textwidth]{figs/T34_1.png}
%     \includegraphics[width=0.46\textwidth]{figs/T4.0003.jpeg}
%    }\\ 
%    % \vspace{-0.1in}
%  \mbox{(c)
%  % \includegraphics[width=0.46\textwidth]{figs/T34_2.png}
%      \includegraphics[width=0.46\textwidth]{figs/T4.0007.jpeg}
%  (d) 
%  % \includegraphics[width=0.46\textwidth]{figs/T34_3.png}
%      \includegraphics[width=0.46\textwidth]{figs/T4.0011.jpeg}
%  }
%     \caption{
%     % Example \ref{ex2}: 
%     Case 4 (Droplet merging). 
% Snapshots of the controlled surface height  at different times: (a) $t=0.0$, (b) $t=0.3$, (c) $t=0.7$, (d) $t=1.0$.}
%     \label{fig:T34a}
% \end{figure}

\blue{Figure~\ref{fig:T38a} presents simulation snapshots for the second sub-case with the asymmetric target surface height \eqref{mgB}. This case exhibits more interesting transient dynamics due to the asymmetry in the target droplet located at the off-center position $(x, y) = (0.6, 0.55)$. In the early stage, the two droplets in the left half of the domain quickly shrink in size, while the two droplets in the right half of the domain gradually bead up at different speeds. In the later stage, the two remaining droplets slowly move towards each other and coalesce into the target droplet, which is elongated in the $x$ direction at the desired off-center location at time $t = 1$. The corresponding activity field plots illustrate that the high extensile active stress ($\zeta < 0$) around the two droplets on the left contributes to the rapid vanishing of the droplets. Meanwhile, the combined effects of the high contractile active stress ($\zeta > 0$) near the location of the terminal droplet and the low extensile stress ($\zeta < 0$) near the contact line of the two droplets on the right lead to the directed motion of the droplets towards the terminal location. This example specifically showcases the flexibility of our control mechanism in handling initial and terminal configurations of different total masses through the non-mass-conserving effects.
}
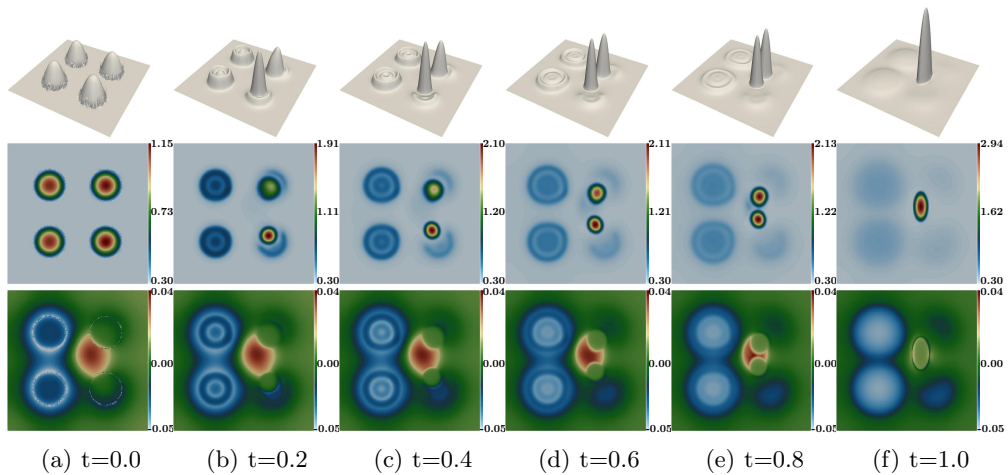
\begin{figure}
\centering
% \subfigure[t=0.0]{
% \includegraphics[width=0.15\textwidth]{figs/zaC38T0h.jpg}}
% \subfigure[t=0.2]{
% \includegraphics[width=0.15\textwidth]{figs/zaC38T2h.jpg}}
% \subfigure[t=0.4]{
% \includegraphics[width=0.15\textwidth]{figs/zaC38T4h.jpg}}
% \subfigure[t=0.6]{
% \includegraphics[width=0.15\textwidth]{figs/zaC38T6h.jpg}}
% \subfigure[t=0.8]{
% \includegraphics[width=0.15\textwidth]{figs/zaC38T8h.jpg}}
% \subfigure[t=1.0]{
% \includegraphics[width=0.15\textwidth]{figs/zaC38T10h.jpg}}
\subfigure[t=0.0]{
              \begin{tikzpicture}
        \node [anchor=south west,inner sep=0] (image) at (0,0) {\includegraphics[width=0.14\textwidth]{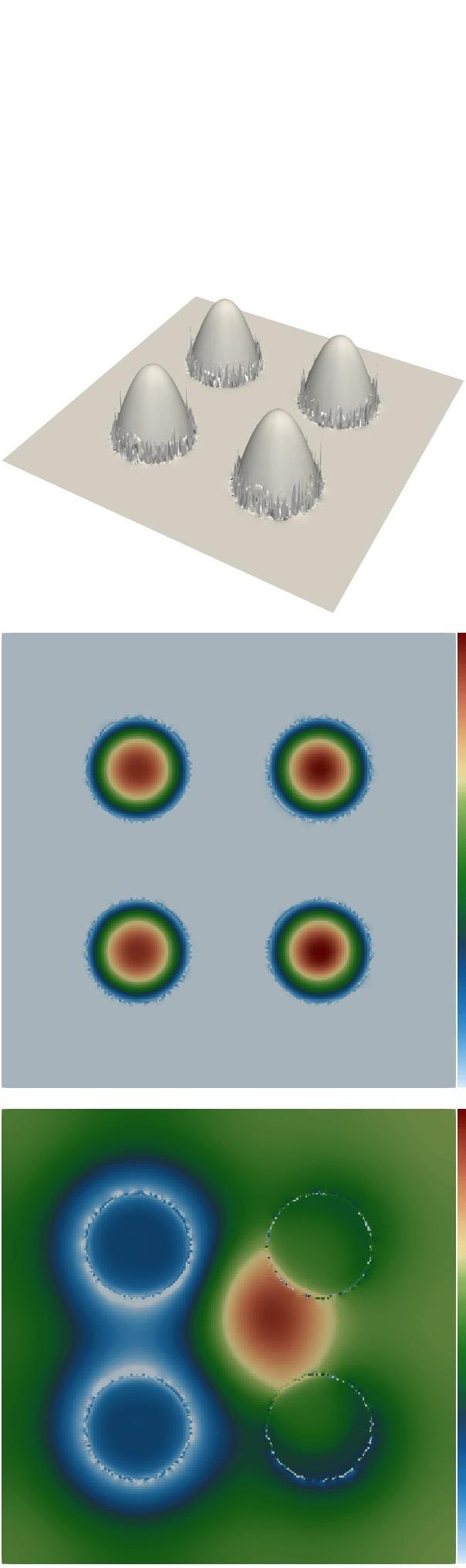}};  % Adjust the width as needed
                \begin{scope}[shift={(image.south east)}]
            % Adjust positioning of numbers vertically and horizontally
            \node at (0.15, 3.8) {\scalebox{0.45}{\textbf{1.15}}}; % First number (on top)
            \node at (0.15, 2.9) {\scalebox{0.45}{\textbf{0.73}}}; % Second number (in the middle)
            \node at (0.15, 2.0) {\scalebox{0.45}{\textbf{0.30}}}; % Third number (on bottom)
            \node at (0.15, 1.85) {\scalebox{0.45}{\textbf{0.04}}}; % First number (on top)
            \node at (0.15, 0.9) {\scalebox{0.45}{\textbf{0.00}}}; % Second number (in the middle)
            \node at (0.14, 0.05) {\scalebox{0.45}{\textbf{-0.05}}}; % Third number (on bottom)
        \end{scope}        
    \end{tikzpicture}}
        \hspace{-0.43cm}
        \subfigure[t=0.2]{
              \begin{tikzpicture}
        \node [anchor=south west,inner sep=0] (image) at (0,0) {\includegraphics[width=0.14\textwidth]{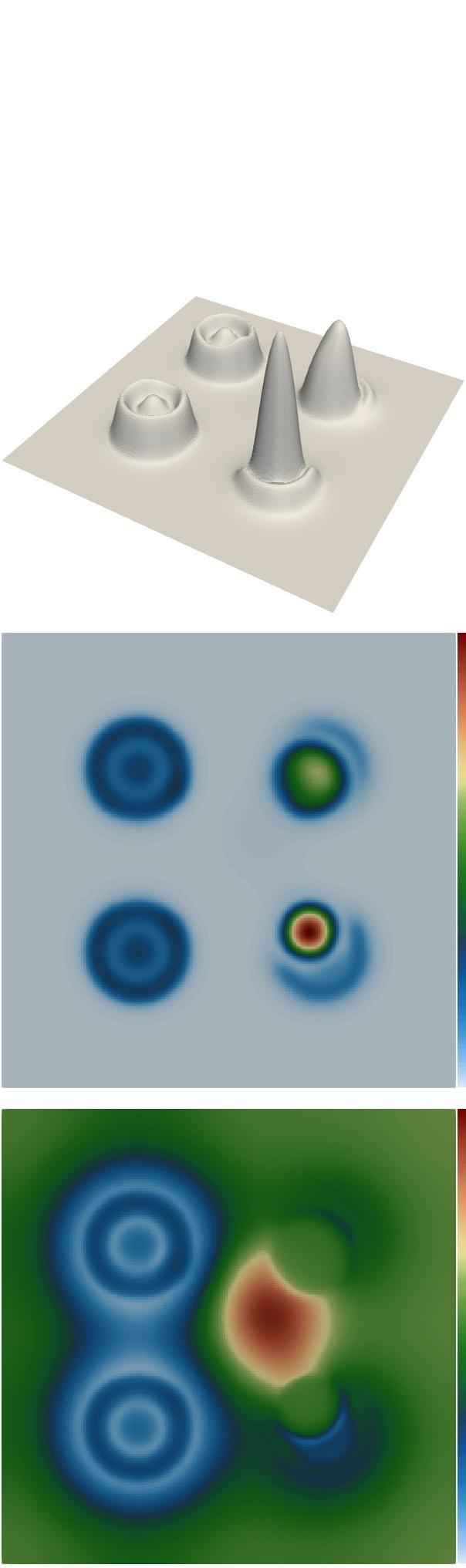}};  % Adjust the width as needed
        % Get the coordinates of the image and add the vertical numbers to the right
                \begin{scope}[shift={(image.south east)}]
            % Adjust positioning of numbers vertically and horizontally
            \node at (0.15, 3.8) {\scalebox{0.45}{\textbf{1.91}}}; % First number (on top)
            \node at (0.15, 2.9) {\scalebox{0.45}{\textbf{1.11}}}; % Second number (in the middle)
            \node at (0.15, 2.0) {\scalebox{0.45}{\textbf{0.30}}}; % Third number (on bottom)
            \node at (0.15, 1.85) {\scalebox{0.45}{\textbf{0.04}}}; % First number (on top)
            \node at (0.15, 0.9) {\scalebox{0.45}{\textbf{0.00}}}; % Second number (in the middle)
            \node at (0.14, 0.05) {\scalebox{0.45}{\textbf{-0.05}}}; % Third number (on bottom)
        \end{scope}        
    \end{tikzpicture}}
        \hspace{-0.43cm}
        \subfigure[t=0.4]{
              \begin{tikzpicture}
        \node [anchor=south west,inner sep=0] (image) at (0,0) {\includegraphics[width=0.14\textwidth]{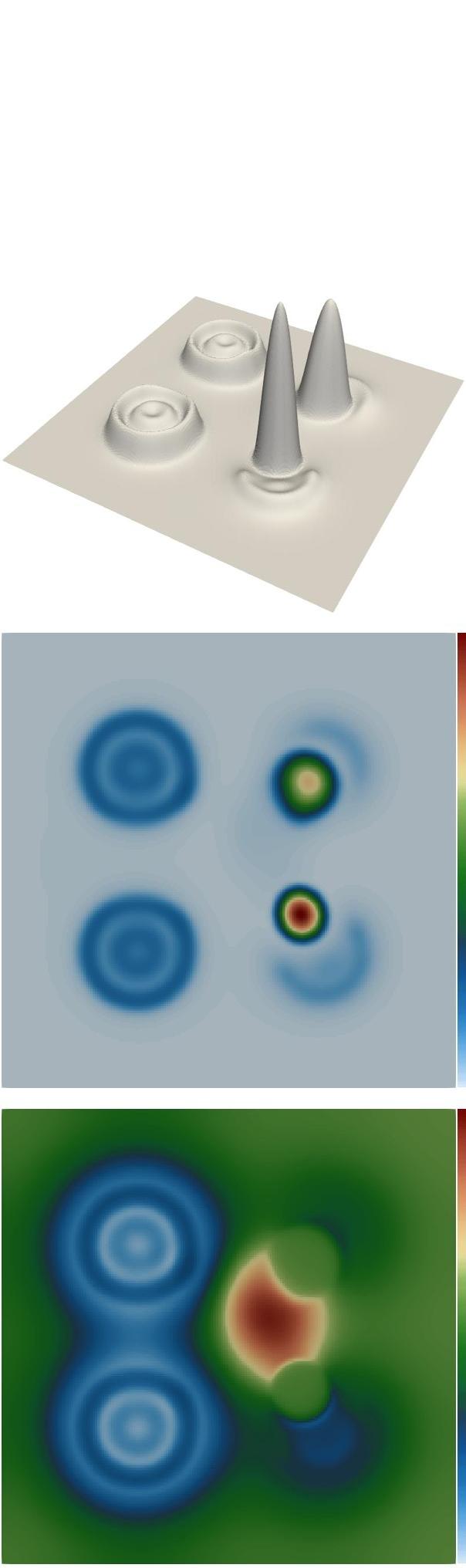}};  % Adjust the width as needed
                \begin{scope}[shift={(image.south east)}]
            % Adjust positioning of numbers vertically and horizontally
            \node at (0.15, 3.8) {\scalebox{0.45}{\textbf{2.10}}}; % First number (on top)
            \node at (0.15, 2.9) {\scalebox{0.45}{\textbf{1.20}}}; % Second number (in the middle)
            \node at (0.15, 2.0) {\scalebox{0.45}{\textbf{0.30}}}; % Third number (on bottom)
            \node at (0.15, 1.85) {\scalebox{0.45}{\textbf{0.04}}}; % First number (on top)
            \node at (0.15, 0.9) {\scalebox{0.45}{\textbf{0.00}}}; % Second number (in the middle)
            \node at (0.14, 0.05) {\scalebox{0.45}{\textbf{-0.05}}}; % Third number (on bottom)
        \end{scope}        
    \end{tikzpicture}}
        \hspace{-0.43cm}
\subfigure[t=0.6]{
              \begin{tikzpicture}
        \node [anchor=south west,inner sep=0] (image) at (0,0) {\includegraphics[width=0.14\textwidth]{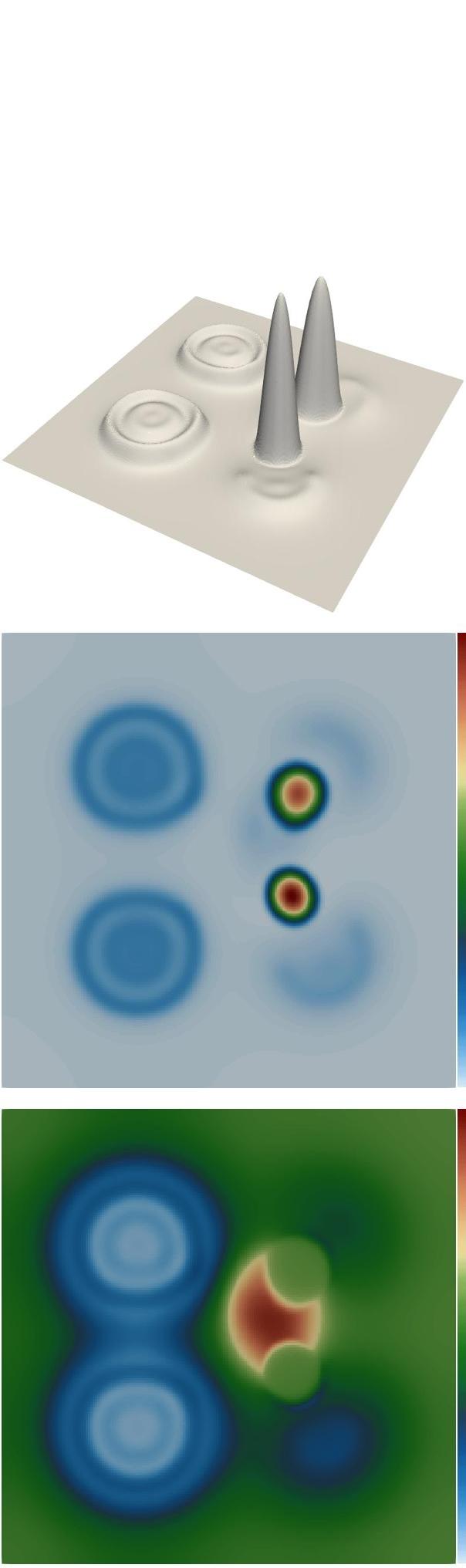}};  % Adjust the width as needed
                        \begin{scope}[shift={(image.south east)}]
            % Adjust positioning of numbers vertically and horizontally
                        % Adjust positioning of numbers vertically and horizontally
            \node at (0.15, 3.8) {\scalebox{0.45}{\textbf{2.11}}}; % First number (on top)
            \node at (0.15, 2.9) {\scalebox{0.45}{\textbf{1.21}}}; % Second number (in the middle)
            \node at (0.15, 2.0) {\scalebox{0.45}{\textbf{0.30}}}; % Third number (on bottom)
            \node at (0.15, 1.85) {\scalebox{0.45}{\textbf{0.04}}}; % First number (on top)
            \node at (0.15, 0.9) {\scalebox{0.45}{\textbf{0.00}}}; % Second number (in the middle)
            \node at (0.14, 0.05) {\scalebox{0.45}{\textbf{-0.05}}}; % Third number (on bottom)
        \end{scope}        
    \end{tikzpicture}}
        \hspace{-0.43cm}
    \subfigure[t=0.8]{
              \begin{tikzpicture}
        \node [anchor=south west,inner sep=0] (image) at (0,0) {\includegraphics[width=0.14\textwidth]{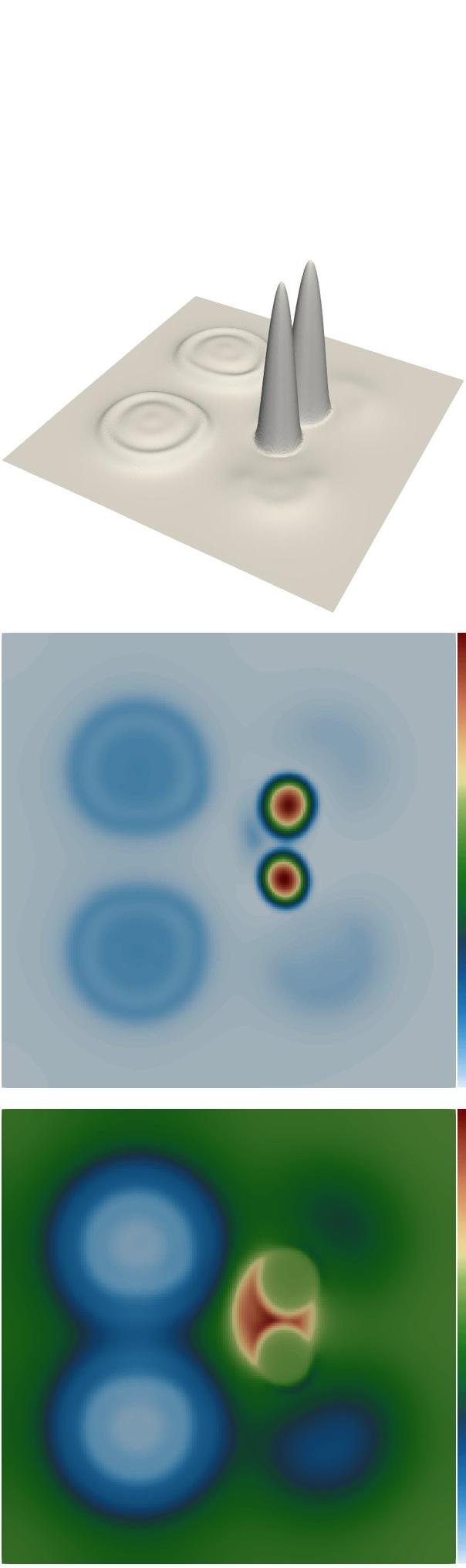}};  % Adjust the width as needed
                \begin{scope}[shift={(image.south east)}]
            % Adjust positioning of numbers vertically and horizontally
            \node at (0.15, 3.8) {\scalebox{0.45}{\textbf{2.13}}}; % First number (on top)
            \node at (0.15, 2.9) {\scalebox{0.45}{\textbf{1.22}}}; % Second number (in the middle)
            \node at (0.15, 2.0) {\scalebox{0.45}{\textbf{0.30}}}; % Third number (on bottom)
            \node at (0.15, 1.85) {\scalebox{0.45}{\textbf{0.04}}}; % First number (on top)
            \node at (0.15, 0.9) {\scalebox{0.45}{\textbf{0.00}}}; % Second number (in the middle)
            \node at (0.14, 0.05) {\scalebox{0.45}{\textbf{-0.05}}}; % Third number (on bottom)
        \end{scope}        
    \end{tikzpicture}}
        \hspace{-0.43cm}
\subfigure[t=1.0]{
              \begin{tikzpicture}
        \node [anchor=south west,inner sep=0] (image) at (0,0) {\includegraphics[width=0.14\textwidth]{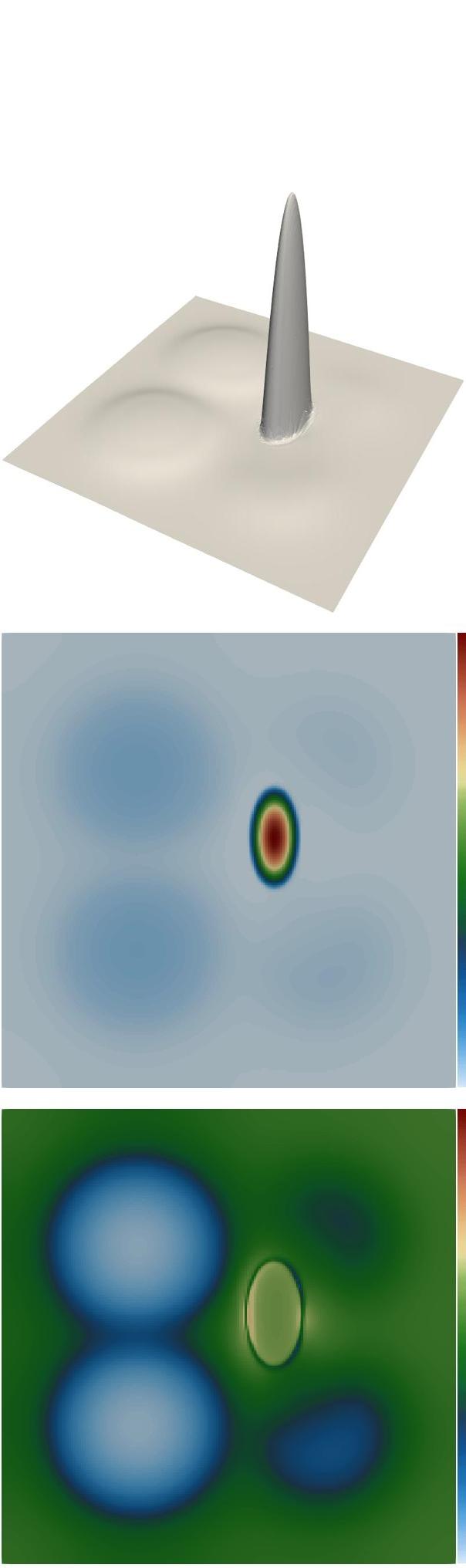}};  % Adjust the width as needed
                \begin{scope}[shift={(image.south east)}]
            % Adjust positioning of numbers vertically and horizontally
            \node at (0.15, 3.8) {\scalebox{0.45}{\textbf{2.94}}}; % First number (on top)
            \node at (0.15, 2.9) {\scalebox{0.45}{\textbf{1.62}}}; % Second number (in the middle)
            \node at (0.15, 2.0) {\scalebox{0.45}{\textbf{0.30}}}; % Third number (on bottom)
            \node at (0.15, 1.85) {\scalebox{0.45}{\textbf{0.04}}}; % First number (on top)
            \node at (0.15, 0.9) {\scalebox{0.45}{\textbf{0.00}}}; % Second number (in the middle)
            \node at (0.14, 0.05) {\scalebox{0.45}{\textbf{-0.05}}}; % Third number (on bottom)
        \end{scope}        
    \end{tikzpicture}}
\caption{
% Example \ref{ex2}: 
\blue{
    Case 4 (Droplet merging with an asymmetric target profile \eqref{mgB}):
Snapshots of (top row) 3D plots of the controlled surface height $h$, (middle row) contour plots of the controlled surface height, and (bottom row) contour plots of the activity field $\zeta$ at different times. The corresponding animation video can be found in the GitHub repository \citep{github}.
%     Top row: snapshots of 3D plots for the controlled surface height $h$ at different times;
% Middle row: snapshots of the controlled surface height contours at different times;
% Bottom row: snapshots of the activity field $\zeta$ at different times.
}
}
\label{fig:T38a}
\end{figure}

% \begin{figure}
% \centering
% \subfigure[t=0.0]{
% \includegraphics[width=0.46\textwidth]{figs/T34_0.png}}
% \subfigure[t=0.3]{
% \includegraphics[width=0.46\textwidth]{figs/T34_1.png}}
% \subfigure[t=0.7]{
% \includegraphics[width=0.46\textwidth]{figs/T34_2.png}}
% \subfigure[t=1.0]{
% \includegraphics[width=0.46\textwidth]{figs/T34_3.png}}
% \caption{Example \ref{ex2}: Case 4 (merging). 
% Snapshots of surface height contour at different times.
% }
% \label{fig:T34a}
% \end{figure}

% \begin{figure}
% \centering
% % \includegraphics[width=0.5\textwidth]{figs/T34.png}
% \includegraphics[width=0.5\textwidth]{figs/3d_T4.pdf}
% \caption{
% % Example \ref{ex2}: 
% Case 4 (Droplet merging). 
% Time evolution of the controlled surface height along the north-east diagonal cutline shown in Figure \ref{fig:T34a}.
% }
% \label{fig:T34b}
% \end{figure}

\subsubsection*{Case 5: Droplet splitting.}
For the case of droplet splitting, we consider the reverse process of the symmetric droplet merging considered in Case 4 and take the initial and target surface heights as 
\begin{align}
h_0(x, y) =&\;  \epsilon+
\tfrac{5}{12}[1-\tfrac{75}{8}((x-0.5)^2+(y-0.5)^2)]_+,\\
\blue{h_{T}(x, y)} =&\; \epsilon+\sum_{i=1}^{4}\tfrac{5}{6}[1-75((x-x_i)^2+(y-y_i)^2)]_+,
\end{align}
% \begin{align*}
% h_0(x, y) =&\; \epsilon+
% \frac{5}{12}\left(1-\frac{75}{8}((x-0.5)^2+(y-0.5)^2)\right)_+,\\
% h_{\mathrm{trg}}(x, y) =&\; \epsilon+
% \frac{5}{6}\left(1-75((x-0.3)^2+(y-0.3)^2)\right)_+ \\
% &\;\;+\frac{5}{6}\left(1-75((x-0.3)^2+(y-0.7)^2)\right)_+\\
% &\;\;+\frac{5}{6}\left(1-75((x-0.7)^2+(y-0.3)^2)\right)_+\\
% &\;\;+\frac{5}{6}\left(1-75((x-0.7)^2+(y-0.7)^2)\right)_+.
% \end{align*}
% This is a reverse process of Case 4, 
where the initial and target profiles are reversed compared to the first example in Case 4, with the droplet configurations in the initial and target profiles identical to those specified in \eqref{mgB} and \eqref{mg_init}, respectively.
% and the locations of the peaks $\{(x_i, y_i)\}$ in the target profile are identical to those in the initial profile specified in Case 4. 
Snapshots of the simulation results for the controlled dynamics are presented in Figure~\ref{fig:T35a}.
This example models the MFC of one initial parabolic droplet splitting into four smaller droplets and shifting into target positions individually. Compared to the droplet merging case, splitting a single droplet appears to be more challenging, and the terminal surface height profile obtained at the terminal time $t = 1$ still maintains a ridge connecting the small droplets. 
\blue{The activity plots in the bottom row of Figure~\ref{fig:T35a} demonstrate that near time $t = 0$, the inhomogeneous distribution of high extensile active stress ($\zeta < 0$) near the contact line of the center droplet contributes to the fast splitting of the initial droplet. Meanwhile, the high contractile active stress ($\zeta > 0$) around the terminal locations of the target droplets drives the individual droplets away from the center and pushes them towards the target locations. As the terminal time $t = 1$ approaches, high extensile stress ($\zeta < 0$) at the center of the domain and high contractile stress ($\zeta > 0$) are formed as the individual droplets spread and settle into their target shapes.
}

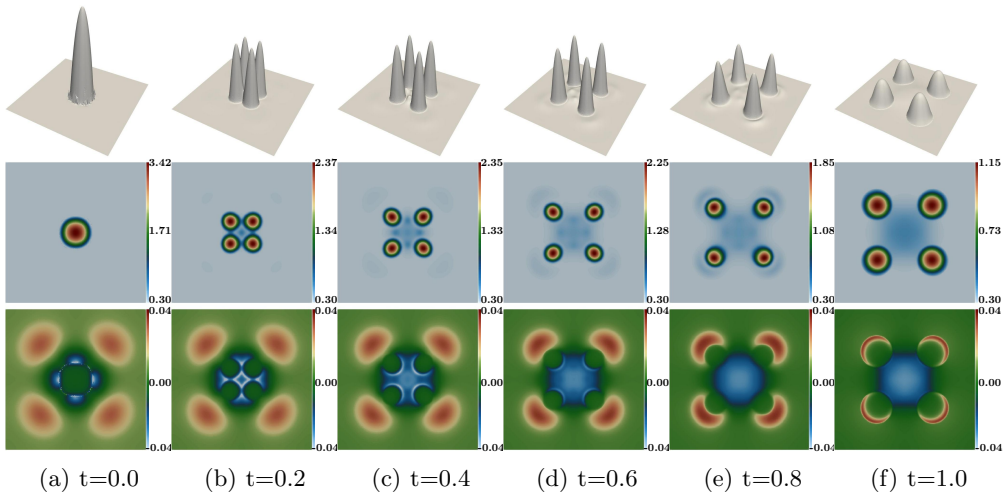
\begin{figure}
\centering
\subfigure[t=0.0]{
              \begin{tikzpicture}
        \node [anchor=south west,inner sep=0] (image) at (0,0) {\includegraphics[width=0.14\textwidth]{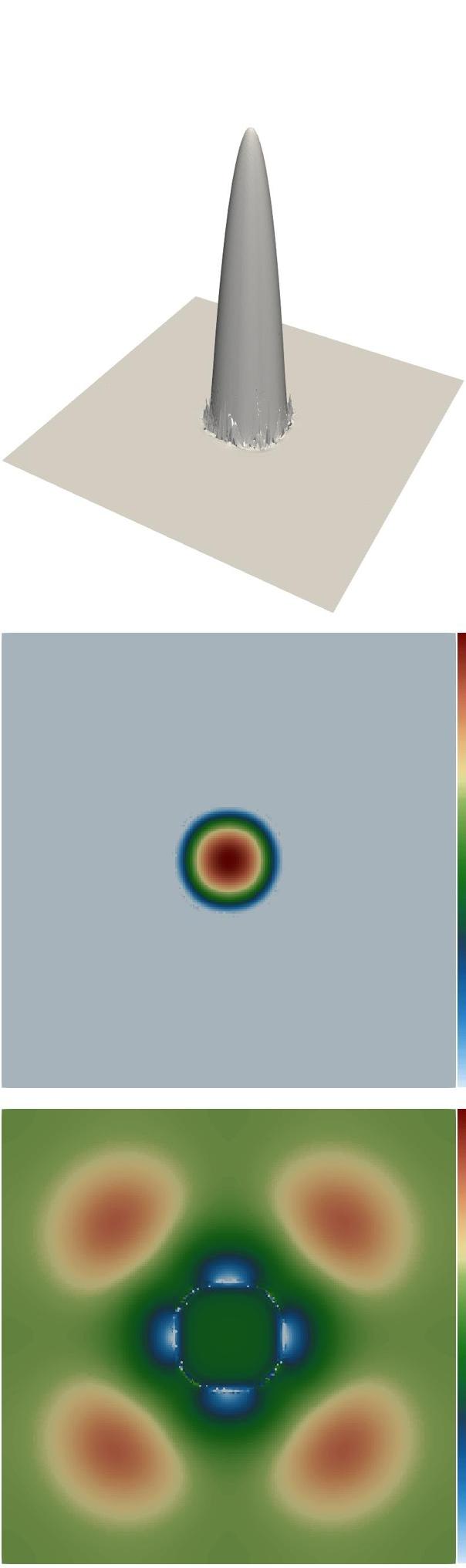}};  % Adjust the width as needed
                \begin{scope}[shift={(image.south east)}]
            % Adjust positioning of numbers vertically and horizontally
            \node at (0.15, 3.8) {\scalebox{0.45}{\textbf{3.42}}}; % First number (on top)
            \node at (0.15, 2.9) {\scalebox{0.45}{\textbf{1.71}}}; % Second number (in the middle)
            \node at (0.15, 2.0) {\scalebox{0.45}{\textbf{0.30}}}; % Third number (on bottom)
            \node at (0.15, 1.85) {\scalebox{0.45}{\textbf{0.04}}}; % First number (on top)
            \node at (0.15, 0.9) {\scalebox{0.45}{\textbf{0.00}}}; % Second number (in the middle)
            \node at (0.14, 0.05) {\scalebox{0.45}{\textbf{-0.04}}}; % Third number (on bottom)
        \end{scope}        
    \end{tikzpicture}}
        \hspace{-0.43cm}
        \subfigure[t=0.2]{
              \begin{tikzpicture}
        \node [anchor=south west,inner sep=0] (image) at (0,0) {\includegraphics[width=0.14\textwidth]{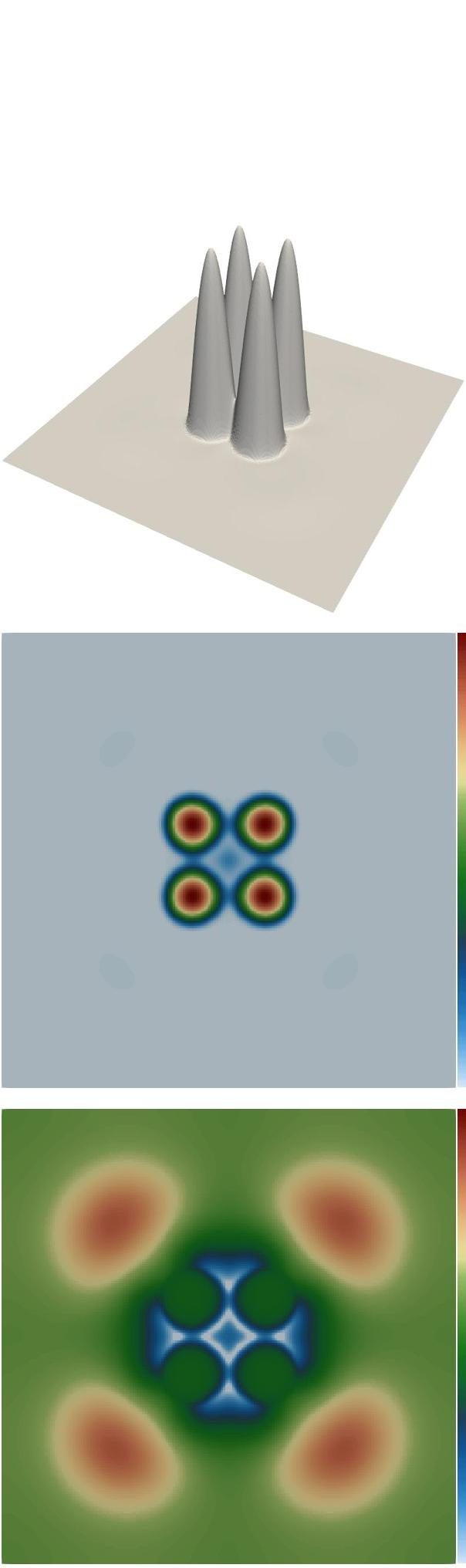}};  % Adjust the width as needed
        % Get the coordinates of the image and add the vertical numbers to the right
                \begin{scope}[shift={(image.south east)}]
            % Adjust positioning of numbers vertically and horizontally
            \node at (0.15, 3.8) {\scalebox{0.45}{\textbf{2.37}}}; % First number (on top)
            \node at (0.15, 2.9) {\scalebox{0.45}{\textbf{1.34}}}; % Second number (in the middle)
            \node at (0.15, 2.0) {\scalebox{0.45}{\textbf{0.30}}}; % Third number (on bottom)
            \node at (0.15, 1.85) {\scalebox{0.45}{\textbf{0.04}}}; % First number (on top)
            \node at (0.15, 0.9) {\scalebox{0.45}{\textbf{0.00}}}; % Second number (in the middle)
            \node at (0.14, 0.05) {\scalebox{0.45}{\textbf{-0.04}}}; % Third number (on bottom)
        \end{scope}        
    \end{tikzpicture}}
        \hspace{-0.43cm}
        \subfigure[t=0.4]{
              \begin{tikzpicture}
        \node [anchor=south west,inner sep=0] (image) at (0,0) {\includegraphics[width=0.14\textwidth]{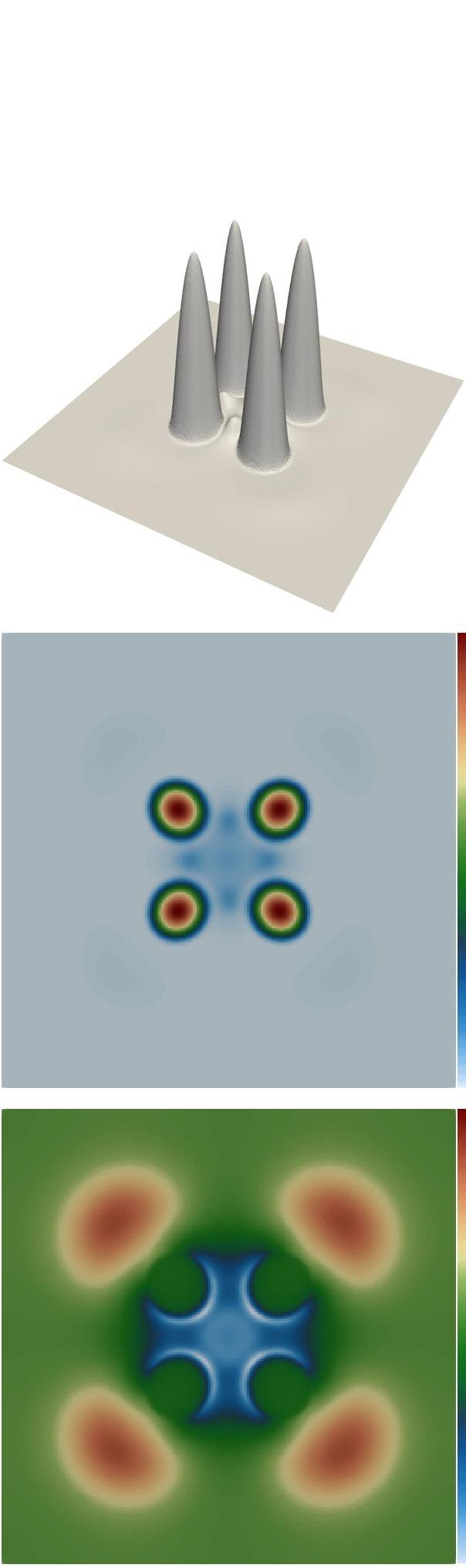}};  % Adjust the width as needed
                \begin{scope}[shift={(image.south east)}]
            % Adjust positioning of numbers vertically and horizontally
            \node at (0.15, 3.8) {\scalebox{0.45}{\textbf{2.35}}}; % First number (on top)
            \node at (0.15, 2.9) {\scalebox{0.45}{\textbf{1.33}}}; % Second number (in the middle)
            \node at (0.15, 2.0) {\scalebox{0.45}{\textbf{0.30}}}; % Third number (on bottom)
            \node at (0.15, 1.85) {\scalebox{0.45}{\textbf{0.04}}}; % First number (on top)
            \node at (0.15, 0.9) {\scalebox{0.45}{\textbf{0.00}}}; % Second number (in the middle)
            \node at (0.14, 0.05) {\scalebox{0.45}{\textbf{-0.04}}}; % Third number (on bottom)
        \end{scope}        
    \end{tikzpicture}}
        \hspace{-0.43cm}
\subfigure[t=0.6]{
              \begin{tikzpicture}
        \node [anchor=south west,inner sep=0] (image) at (0,0) {\includegraphics[width=0.14\textwidth]{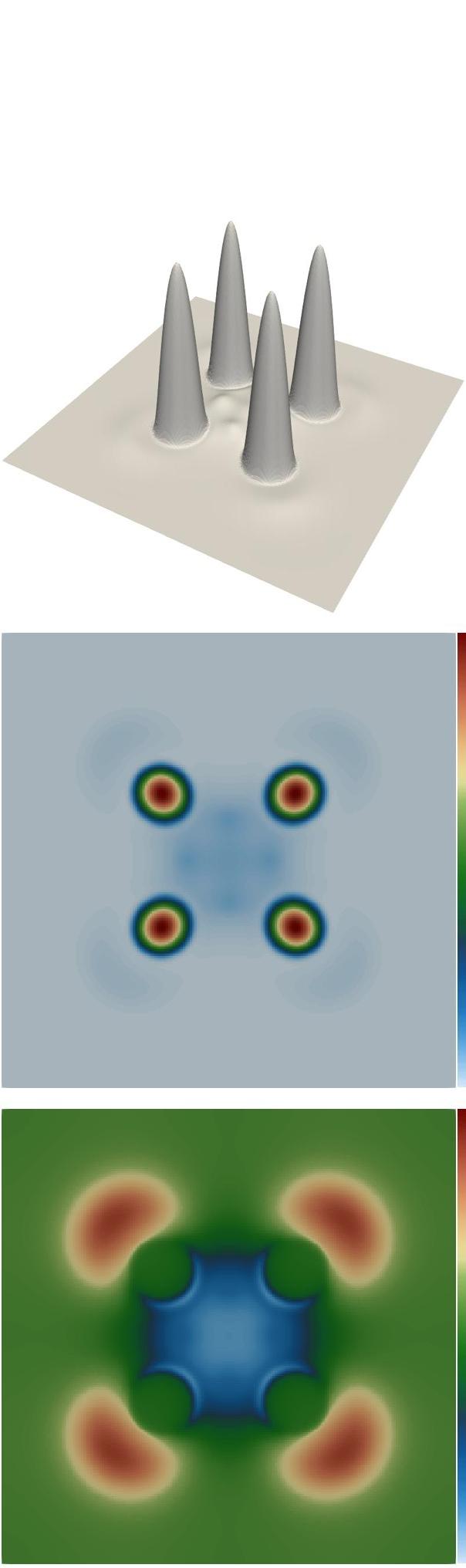}};  % Adjust the width as needed
                        \begin{scope}[shift={(image.south east)}]
            % Adjust positioning of numbers vertically and horizontally
                        % Adjust positioning of numbers vertically and horizontally
            \node at (0.15, 3.8) {\scalebox{0.45}{\textbf{2.25}}}; % First number (on top)
            \node at (0.15, 2.9) {\scalebox{0.45}{\textbf{1.28}}}; % Second number (in the middle)
            \node at (0.15, 2.0) {\scalebox{0.45}{\textbf{0.30}}}; % Third number (on bottom)
            \node at (0.15, 1.85) {\scalebox{0.45}{\textbf{0.04}}}; % First number (on top)
            \node at (0.15, 0.9) {\scalebox{0.45}{\textbf{0.00}}}; % Second number (in the middle)
            \node at (0.14, 0.05) {\scalebox{0.45}{\textbf{-0.04}}}; % Third number (on bottom)
        \end{scope}        
    \end{tikzpicture}}
        \hspace{-0.43cm}
    \subfigure[t=0.8]{
              \begin{tikzpicture}
        \node [anchor=south west,inner sep=0] (image) at (0,0) {\includegraphics[width=0.14\textwidth]{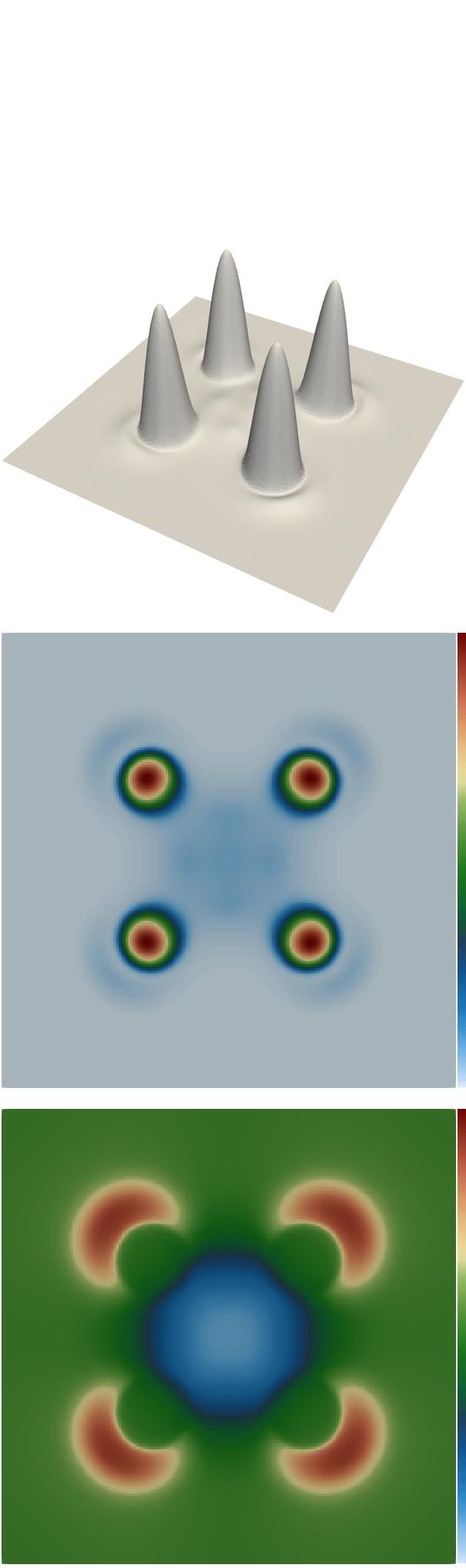}};  % Adjust the width as needed
                \begin{scope}[shift={(image.south east)}]
            % Adjust positioning of numbers vertically and horizontally
            \node at (0.15, 3.8) {\scalebox{0.45}{\textbf{1.85}}}; % First number (on top)
            \node at (0.15, 2.9) {\scalebox{0.45}{\textbf{1.08}}}; % Second number (in the middle)
            \node at (0.15, 2.0) {\scalebox{0.45}{\textbf{0.30}}}; % Third number (on bottom)
            \node at (0.15, 1.85) {\scalebox{0.45}{\textbf{0.04}}}; % First number (on top)
            \node at (0.15, 0.9) {\scalebox{0.45}{\textbf{0.00}}}; % Second number (in the middle)
            \node at (0.14, 0.05) {\scalebox{0.45}{\textbf{-0.04}}}; % Third number (on bottom)
        \end{scope}        
    \end{tikzpicture}}
        \hspace{-0.43cm}
\subfigure[t=1.0]{
              \begin{tikzpicture}
        \node [anchor=south west,inner sep=0] (image) at (0,0) {\includegraphics[width=0.14\textwidth]{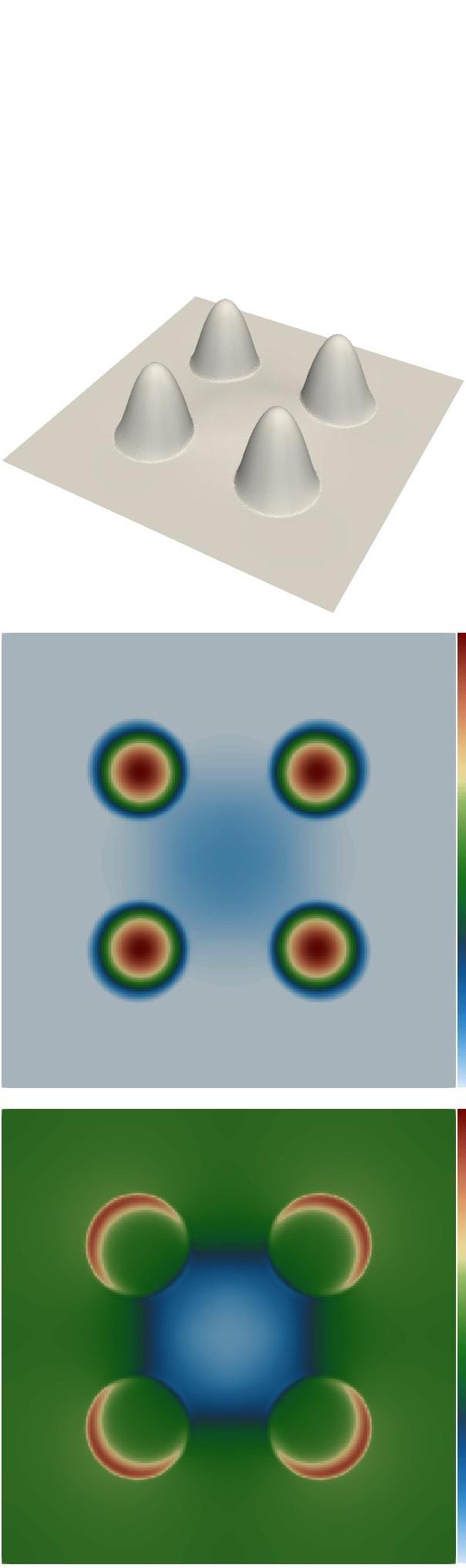}};  % Adjust the width as needed
                \begin{scope}[shift={(image.south east)}]
            % Adjust positioning of numbers vertically and horizontally
            \node at (0.15, 3.8) {\scalebox{0.45}{\textbf{1.15}}}; % First number (on top)
            \node at (0.15, 2.9) {\scalebox{0.45}{\textbf{0.73}}}; % Second number (in the middle)
            \node at (0.15, 2.0) {\scalebox{0.45}{\textbf{0.30}}}; % Third number (on bottom)
            \node at (0.15, 1.85) {\scalebox{0.45}{\textbf{0.04}}}; % First number (on top)
            \node at (0.15, 0.9) {\scalebox{0.45}{\textbf{0.00}}}; % Second number (in the middle)
            \node at (0.14, 0.05) {\scalebox{0.45}{\textbf{-0.04}}}; % Third number (on bottom)
        \end{scope}        
    \end{tikzpicture}}
% \subfigure[t=0.0]{
% \includegraphics[width=0.15\textwidth]{figs/zaC35T0h.jpg}}
% \subfigure[t=0.2]{
% \includegraphics[width=0.15\textwidth]{figs/zaC35T2h.jpg}}
% \subfigure[t=0.4]{
% \includegraphics[width=0.15\textwidth]{figs/zaC35T4h.jpg}}
% \subfigure[t=0.6]{
% \includegraphics[width=0.15\textwidth]{figs/zaC35T6h.jpg}}
% \subfigure[t=0.8]{
% \includegraphics[width=0.15\textwidth]{figs/zaC35T8h.jpg}}
% \subfigure[t=1.0]{
% \includegraphics[width=0.15\textwidth]{figs/zaC35T10h.jpg}}
\caption{
% Example \ref{ex2}: 
\blue{
    Case 5 (Droplet splitting): 
Snapshots of (top row) 3D plots of the controlled surface height $h$, (middle row) contour plots of the controlled surface height, and (bottom row) contour plots of the activity field $\zeta$ at different times. The corresponding animation video can be found in the GitHub repository \citep{github}; see Movie.avi in online supplementary material.
%     Top row: snapshots of 3D plots for the controlled surface height $h$ at different times;
% Middle row: snapshots of the controlled surface height contours at different times;
% Bottom row: snapshots of the activity field $\zeta$ at different times.
}
}
% \mbox{
% % (a) \includegraphics[width=0.46\textwidth]{figs/T35_0.png}
% % (b) \includegraphics[width=0.46\textwidth]{figs/T35_1.png}
% (a) \includegraphics[width=0.46\textwidth]{figs/T5_0.jpeg}
% (b) \includegraphics[width=0.46\textwidth]{figs/T5.0003.jpeg}
% }\\
% \mbox{
% % (c) \includegraphics[width=0.46\textwidth]{figs/T35_2.png}
% % (d) \includegraphics[width=0.46\textwidth]{figs/T35_3.png}
% (c) \includegraphics[width=0.46\textwidth]{figs/T5.0007.jpeg}
% (d) \includegraphics[width=0.46\textwidth]{figs/T5.0011.jpeg}
% }
% \caption{
% % Example \ref{ex2}: 
% Case 5 (Droplet splitting). 
% Snapshots of the controlled surface height at different times.
%  (a) $t=0.0$, (b) $t=0.3$, (c) $t=0.7$, (d) $t=1.0$. }
\label{fig:T35a}
\end{figure}

% \begin{figure}
% \centering
% \subfigure[t=0.0]{
% \includegraphics[width=0.46\textwidth]{figs/T35_0.png}}
% \subfigure[t=0.3]{
% \includegraphics[width=0.46\textwidth]{figs/T35_1.png}}
% \subfigure[t=0.7]{
% \includegraphics[width=0.46\textwidth]{figs/T35_2.png}}
% \subfigure[t=1.0]{
% \includegraphics[width=0.46\textwidth]{figs/T35_3.png}}
% \caption{Example \ref{ex2}: Case 5 (splitting). 
% Snapshots of surface height contour at different times.
% }
% \label{fig:T35a}
% \end{figure}

% \begin{figure}
% \centering
% % \includegraphics[width=0.5\textwidth]{figs/T35.png}
% \includegraphics[width=0.5\textwidth]{figs/3d_T5.pdf}
% \caption{
% % Example \ref{ex2}: 
% Case 5 (Droplet splitting). 
% Time evolution of the controlled surface height along the north-east diagonal cutline shown in Figure \ref{fig:T35a}.
% }
% \label{fig:T35b}
% \end{figure}

%\subsubsection{Droplet flatten.}
%\subsubsection{Droplet sharpening.}
%\subsubsection{Droplet splitting.}
%\subsubsection{Droplet merging.}

\section{Discussion}
\label{sec5}
In this paper, we formulate and numerically compute mean field control problems for droplet dynamics governed by a thin-film equation with a non-mass-conserving flux. 
Our formulation starts with droplet dynamics, which are gradient flows of free energies in optimal transport metric spaces with nonlinear mobility functions. 
\blue{To demonstrate the application of our approach, we develop a lubrication model for a thin volatile liquid film with an active suspension, where the control is achieved through its activity field.}
We design and compute these mean field control problems of droplet dynamics using the primal-dual hybrid gradient algorithms with high-order finite element approximation schemes. Numerical examples of two-dimensional uncontrolled and controlled droplet dynamics demonstrate the effectiveness of the proposed control mechanisms.

{We expect that the developed mean field control mechanism could open the door to studying experimental design problems of droplet dynamics. To bridge the gap between the methodology and practical experimental implementation, one needs to address potential challenges arising from model formulation, parameter constraints, and experimental measurements. To apply the developed algorithm to other droplet control problems, we need general strategies for choosing free energies and mobility functions.} For example, one may design suitable free energies coupled with other external field constraints to adapt our proposed mean field control approach for droplet dynamics via temperature \citep{ji2021thermally} or electric fields \citep{chu2023electrohydrodynamics,eaker2016liquid}.
\blue{For controlling thin liquid films on general geometries, such as films flowing down a vertical cylinder \citep{ruyer2008modelling,ji2019dynamics} or on a spherical substrate \citep{greer2006fourth}, the gradient flow structure in \eqref{eq:model} needs to be adapted to account for the geometrical constraints.}
\blue{The current MFC formulation does not impose any explicit constraints on the control variable, but in practice, the feasible parameter range for the strength of the activity field (and controls for other applications) needs to be incorporated into the model by imposing control restrictions. The robustness of the developed algorithm to noise often found in experimental measurements in the initial configuration and control variables also needs to be addressed thoroughly.}
In simulations, one of the challenges is the nonconvex formulations of general mean field variational problems. Suitable regularization functionals are needed to maintain the stability of simulations. \blue{Moreover, our numerical results indicate that compared to directly solving the critical point system \eqref{newton} using the Newton-Raphson method, Algorithm \ref{alg:1} can be quite slow when the default parameters $\sigma_u=1$ and $\sigma_\phi=1$ are used to solve the approximated JKO scheme \eqref{saddleH-JKOa}. It would be interesting to study the convergence behaviors of JKO schemes 
and their improvement in optimization procedures.
We leave these physical modeling and numerical studies for future work. }

\section*{Acknowledgements}
G. Fu's work is supported by NSF DMS-2410740. 
H. Ji's work is supported by NSF DMS-2309774. 
W. Pazner’s work is supported by NSF RTG DMS-2136228 and an ORAU Ralph E. Powe Junior Faculty Enhancement Award.
W. Li's work is supported by AFOSR YIP award No. FA9550-23-1-0087, NSF DMS-2245097, and NSF RTG: 2038080.

\section*{Declaration of Interests}
The authors report no conflict of interest.

% \bibliographystyle{siam}
% \bibliography{references,HJKO,HPC,droplets}

% \section*{Appendix}
\appendix

% \section{Mathematical formulation of \eqref{volatileActive_main}} 
% \label{app:formulation}

\section{Proof of Proposition \ref{MFCD}}\label{appA}
In this appendix, we prove the Proposition \ref{MFCD}. 

\begin{proof}[Proof of Proposition \ref{MFCD}]
Denote the Lagrange multiplier of the MFC problem \eqref{mfcR2} as $\phi\colon [0, T]\times\Omega\rightarrow\mathbb{R}$. Consider the following saddle point problem:
\begin{equation}
\begin{split}
\inf_{\bmm, s, h, h_T}\sup_\phi \quad \mathcal{L}(\bmm, s, h, h_T, \phi),
\end{split}
\end{equation}
where
\begin{equation}
\begin{split}
\mathcal{L}(\bmm, s, h, h_T, \phi)=& \int_0^T \int_\Omega \Big[ \frac{\|\bmm\|^2}{2V_1(h)}+\frac{|s|^2}{2V_2(h)}+\phi\Big(\partial_t h+ \nabla\cdot \bmm-s\Big)\Big] dxdt\\
&+\int_{0}^T\Big[\frac{\beta^2}{2}\mathcal{I}(h)-\mathcal{F}(h)\Big] dt
+\mathcal{G}(h_T)+\beta\mathcal{E}(h_T).
\end{split}
\end{equation}
Assume $h>0$. By solving the saddle point problem of $\mathcal{L}$, i.e., taking the $L^2$ first variation of $\mathcal{L}$ on variables $\m$, $s$, $h$, $h_T$, we derive
\begin{equation*}
\left\{\begin{split}
&\frac{\delta}{\delta \bmm}\mathcal{L}=0,\\
& \frac{\delta}{\delta s}\mathcal{L}=0,\\
 &\frac{\delta}{\delta h}\mathcal{L}=0,\\
 &\frac{\delta}{\delta \phi}\mathcal{L}=0,\\
  &\frac{\delta}{\delta h_T}\mathcal{L}=0,\\
\end{split}\right.\quad\Rightarrow\quad\left\{\begin{split}
&\frac{\bmm}{V_1}=\nabla \phi,\\
& \frac{s}{V_2}=\phi,\\
 &-\frac{1}{2}\frac{\|\bmm\|^2}{V_1^2}V_1'-\frac{1}{2}\frac{|s|^2}{V_2^2}V_2'+\frac{\delta}{\delta h}\Big[\frac{\beta^2}{2}\mathcal{I}(h)-\mathcal{F}(h)\Big]-\partial_t\phi=0,\\
 &\partial_th+\nabla\cdot \bmm-s=0,\\
 & \phi_T+\frac{\delta}{\delta h_T}\Big(\mathcal{G}(h_T)+\beta\mathcal{E}(h_T)\Big)=0. 
\end{split}\right.
\end{equation*}
We finish the derivation of the mean field control system. 

We next derive the $L^2$ first variation of the Fisher information functional $\mathcal{I}$. Recall 
\begin{equation}
\mathcal{I}(h)=\frac{1}{2}
\int_\Omega \left(V_1(h)|\nabla P(h)|^2 + V_2(h)|P(h)|^2\right)dx, 
\end{equation}
where $P(h)=\Pi(h)-\pstar- \alpha^2{\Delta h}$, with the notation $\Delta h=\nabla^2h$.  
Consider a smooth test function $\delta h\in C^{\infty}([0,T]; \mathbb{R})$. Then 
\begin{equation}
\begin{split}
&\mathcal{I}(h+\epsilon \delta h)\\= &\frac{1}{2}
\int_\Omega \left(V_1(h+\epsilon\delta h)|\nabla P(h+\epsilon \delta h)|^2 + V_2(h+\epsilon \delta h)|P(h+\epsilon \delta h)|^2\right)dx\\
=&\quad \frac{1}{2}
\int_\Omega (V_1(h)+\epsilon V_1'(h)\delta h)|\nabla P(h)+\epsilon \nabla(\Pi'(h) \delta h-\alpha^2 \Delta \delta h)|^2 dx\\
&+\frac{1}{2}\int_\Omega (V_2(h)+\epsilon V_2'(h) \delta h)|P(h)+\epsilon (\Pi'(h)\delta h-\alpha^2\Delta\delta h)|^2dx+O(\epsilon^2)\\
=&\mathcal{I}(h)+\epsilon \int_\Omega \left(\frac{1}{2}V_1'(h)\delta h |\nabla P(h)|^2+V_1(h)\nabla P(h)\cdot \nabla(\Pi'(h) \delta h-\alpha^2 \Delta \delta h)\right)dx\\
&\qquad+\epsilon \int_\Omega \left(\frac{1}{2}V_2'(h)\delta h |P(h)|^2+V_2(h)P(h)(\Pi'(h) \delta h-\alpha^2 \Delta \delta h)\right)dx+O(\epsilon^2), 
\end{split}
\end{equation}
where $O(\epsilon^2)$ is the asymptotic notation. Using the definition of $L^2$ first variation operator, 
\begin{equation}
\mathcal{I}(h+\epsilon \delta h)-\mathcal{I}(h)= \epsilon \int_\Omega \frac{\delta }{\delta h}\mathcal{I}(h)\delta h dx+O(\epsilon^2), 
\end{equation}
and applying the integration by parts, we derive the $L^2$ first variation of functional $\mathcal{I}$. 
\end{proof}
\red{
\section{Additional numerical experiments}\label{appB}
In this appendix, we numerically study the convergence of the proposed approximated JKO scheme \eqref{saddleH-JKOa} under spatial/temporal mesh refinements. 
We consider the same setup as in Section \ref{ex1} to solve the thin-film equations \eqref{eq:model} in both one- and two-dimensions. 
Specifically, we compute the $L^2$-error of the surface height 
$\|h_h-h_{ref}\|_{L^2(\Omega)}$ at time $t= 0.05$, where the reference solution $h_{ref}$ is obtained using the finite element scheme \eqref{FEM} with polynomial degree $k+1=4$ on a fine mesh with $128^d$ uniform elements, where $d\in\{1,2\}$ is the spatial dimension, and the time step size is $\Delta t = 10^{-4}$.}

\red{
Table \ref{tab1} presents the history of convergence for the $L^2$-norm error under mesh refinements for the scheme \eqref{saddleH-JKOa} with the polynomial degree $k\in\{0, 1, 2\}$.
When the polynomial degree is $k=0$, we use a sequence of uniform meshes with $N^d$ elements for $N\in\{32, 64, 128\}$ and take a time step size of
$\Delta t = 32/N\times 10^{-3}$. The first-order convergence is observed in this case. 
When the polynomial degree is $k\in\{1,2\}$, we use a sequence of uniform meshes with size $N^d$ for $N\in\{16, 32, 64\}$ and take a time step size of
$\Delta t = (16/N)^{k+1}\times 10^{-3}$. The second-order convergence is observed for $k=1$, and the third-order convergence is observed for $k=2$.
These results suggest the proposed scheme \eqref{saddleH-JKOa} achieves a first-order convergence in time and a $(k+1)$-th order of convergence in space
when polynomials of degree $k\ge 0$ are used.
\begin{table}
    \centering
    \begin{tabular}{c c c|c  c| c c}
        $k$ & $N$ & $\Delta t$ & 1D $L^2$-error & rate & 2D $L^2$-error & rate\\[.2ex]
   \hline
   &    32   &1e-3  & 1.61e-2 & --  & 4.49e-2 & --\\[.2ex]
 0 &    64   &5e-4  & 6.82e-3 & 1.24& 1.72e-2 & 1.39\\[.2ex]
   &    128  &2.5e-4& 3.13e-3 & 1.12& 6.98e-3 & 1.30\\
   \hline
   &    16   &1e-3    & 7.79e-3 & --  &3.13e-2 & -- \\[.2ex]
 1 &    32   &2.5e-4  & 1.18e-3 & 2.72&4.77e-3 & 2.71\\[.2ex]
   &    64   &6.25e-5 & 2.63e-4 & 2.17&1.05e-3 & 2.18\\
   \hline
   &    16   &1e-3      & 5.26e-3 & --  &2.37e-2 & --\\[.2ex]
 2 &    32   &1.25e-4   & 4.82e-4 & 3.45&2.05e-3 & 3.53\\[.2ex]
   &    64   &1.56e-5 & 5.53e-5 & 3.12  &2.30e-4  & 3.15 \\
   \hline
    \end{tabular}
    \caption{History of convergence for the $L^2$-error in surface height at time $t=0.05$. }
    \label{tab1}
\end{table}
}

\red{
We conclude this appendix by presenting the evolution of the total mass, 
$\mathcal{M}(t) = \int_{\Omega} h~d\bm x$, when applying the scheme \eqref{saddleH-JKOa} to solve the model \eqref{eq:model}.
Recall that, from \eqref{mc}, the model \eqref{eq:model} 
allows the total mass to change over time due to non-mass-conserving contributions when the phase change rate $\gamma >0$, while the total mass is conserved for $\gamma = 0$.
To highlight the built-in mass-conservation property of the scheme \eqref{saddleH-JKOa}, we consider the 
two-dimensional setup in Section \ref{ex1} with the phase change rate $\gamma = 0$. 
For this problem, the total mass remains at $\mathcal{M}(t)=1$ for all time since the initial data satisfies $\mathcal{M}=1$.
Figure~\ref{fig:mc} presents the time evolution of the total mass error $|\mathcal{M}(t)-1|$  using the polynomial degree $k=3$
on a $32^2$ mesh and the time step size $\Delta t = 10^{-3}$. 
It is observed that the total mass error is within $5\times 10^{-14}$.
\begin{figure}
    \centering
    \begin{tikzpicture}
        \begin{axis}[
            title={Total Mass Error vs Time},
            xlabel={Time},
            ylabel={$|\mathcal{M}(t)-1|$},
            grid=major,
            width=10cm,
            height=4cm,
        ]
        \addplot
        table {data.txt};  % This reads the data from your file
        \end{axis}
    \end{tikzpicture}
    \caption{Evolution of the total mass error $|\mathcal{M}(t)-1|$ over time for 
    the scheme \eqref{saddleH-JKOa} applied to the  model \eqref{eq:model}
    with $\gamma =0$, 
    the dimension $d=2$, a polynomial of degree $k=3$, a rectangular mesh with $32^2$ elements, and a time step size $\Delta t = 10^{-3}$.
    }
    \label{fig:mc}
\end{figure}
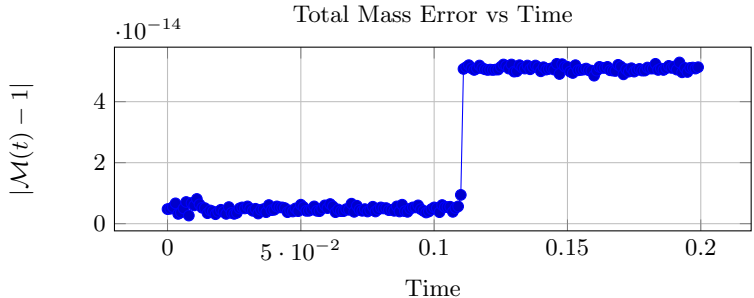
}
% \begin{table}
%     \centering
%     \begin{tabular}{c|c  c| c c}
%         $\Delta t$ & $L^2$-error in 1D & rate & $L^2$-error in 2D & rate\\[.2ex]
%    \hline
%        1e-3   & 3.75e-3 &--\\[.2ex]
%        5e-4   & 2.13e-3 & 0.82\\[.2ex]
%        2.5e-4 & 1.14e-3 & 0.90\\[.2ex]
%        1.25e-4& 5.86e-4 & 0.95\\
%     \end{tabular}
%     \caption{History of convergence for the $L^2$-error in surface height at time $t=0.2$ 
%     for the scheme \eqref{saddleH-JKOa} under temporal mesh refinements. 
%    Spatial discretization: polynomial degree $k=3$ on a uniform $128^d$ mesh.}
%     \label{tab1}
% \end{table}
% \begin{table}
%     \centering
%     \begin{tabular}{c c|c  c| c c}
%     $k$&    $N$ & $L^2$-error in 1D & rate & $L^2$-error in 2D & rate\\[.2ex]
%    \hline
%    &    64    & 7.75e-2 & --\\[.2ex]
%  0  &    128  & 3.84e-2 & 1.01\\[.2ex]
%    &    256   & 1.91e-2 & 1.01\\
%    \hline
%    &    32   & 5.30e-2 & --\\[.2ex]
%  1  &    64  & 1.22e-2 & 2.12\\[.2ex]
%    &    128  & 1.47e-3 & 3.05\\
%    \hline
%    &    16   & 1.96e-2 & --\\[.2ex]
%  3  &   32   & 1.90e-3 & 3.36\\[.2ex]
%    &    64   & 1.70e-4 & 3.48\\
%     \end{tabular}
%     \caption{History of convergence for the $L^2$-error in surface height at time $t=0.2$ 
%     for the scheme \eqref{saddleH-JKOa} 
%     with polynomial degree $k$ on meshes with uniform $N^d$ elements under spatial mesh refinements. 
%    Time step size: $\Delta t = 10^{-4}$ for $k=0$
%    and $k=1$, $\Delta t = 10^{-5}$ for $k=3$.}
%     \label{tab2}
% \end{table}

\bibliographystyle{jfm}
\bibliography{references,HJKO,HPC,droplets}
\end{document}